\documentclass[12pt]{amsart}
\usepackage[all]{xy} 
\usepackage[pagewise]{lineno}
\usepackage[usenames,dvipsnames]{pstricks}
\usepackage{amssymb,mathrsfs,euscript,enumitem,amsmath,bbold}
\usepackage{amssymb}
\usepackage[backgroundcolor=yellow,shadow,colorinlistoftodos]{todonotes}
\usepackage{quoting}

\catcode`\@=11
\def\today{{October 17, 2022}}
\def\@evenfoot{\rule{0pt}{20pt}[\today] \hfill [{\tt \jobname.tex}]}
\def\@oddfoot{\rule{0pt}{20pt}{[\tt \jobname.tex}]\hfill [\today]}
\catcode`\@=13
\allowdisplaybreaks

\makeatletter
\providecommand\@dotsep{5}
\def\listtodoname{List of Todos}
\def\listoftodos{\@starttoc{tdo}\listtodoname}
\makeatother

\newtheorem{theorem}{Theorem}
\newtheorem{corollary}[theorem]{Corollary}

\newtheorem{lemma}[theorem]{Lemma}

\newtheorem{proposition}[theorem]{Proposition}

\newtheorem*{DiamonD}{Diamond}

\theoremstyle{definition}
\newtheorem{example}[theorem]{Example}

\newtheorem{remark}[theorem]{Remark}
\newtheorem{definition}[theorem]{Definition}
\newtheorem{Warning}{Warning}

\begin{document}

\def\oddGr{{\mathfrak {K}}_\ggGrc} 
\def\bbN{{\mathbb N}}
\def\colorop #1(#2;#3){{#1}
   \left(\rule{0pt}{15pt}\right.
         \hskip -3mm \begin{array}{c}
	              #3\\#2
                     \end{array}
         \hskip -3mm \left. 
   \rule{0pt}{15pt} \right)
}
\def\frakC{{\mathfrak C}}
\def\ttV{{\tt V}}
\def\SRTr{{\tt SRTr}}
\def\A{{\mathbb A}}
\def\Freegg{{\Free_{\hskip -.2em \tt gg}}}
\def\FreeSRTr{{\Free_{\hskip -.2em \tt SRTr}}}
\def\FreeTr{{\Free_{\hskip -.2em \tt Tr}}}
\def\Dgg{{D_{\tt gg}}}
\def\DTr{{D_{\tt Tr}}}
\def\DSRTr{{D_{\tt SRTr}}}
\def\zap{{\mathcal U}}
\def\termP{{\sf 1}_\ttP}
\def\termSRTr{{\sf 1}_\SRTr}
\def\termO{{\sf 1}_\ttO}\def\Min{{\mathfrak M}}
\def\minP{{\mathfrak M_\ttP}}
\def\minTr{{\mathfrak M}_\Tr}
\def\minRTr{{\mathfrak M}_\RTr}
\def\minO{{\mathfrak M_\ttO}}
\def\Oper#1{\hbox{$#1$-${\tt Oper}_1$}} 
\def\Collect#1{\hbox{$#1$-${\tt Coll}_1$}}
\def\OperV#1{\hbox{$#1$-${\tt Oper}^\ttV_1$}} 
\def\CollectV#1{\hbox{$#1$-${\tt Coll}^\ttV_1$}}
\def\minWhe{{\mathfrak M}_{\Whe}}
\def\minggGrc{{\mathfrak M}_{\ggGrc}}
\def\termWhe{{\sf 1}_{\Whe}}
\def\termggGrc{{\sf 1}_{\ggGrc}}
\def\Whe{{\tt Whe}}
\def\Tr{{\tt Tr}}
\def\PRTr{{\tt PRTr}}
\def\PTr{{\tt PTr}}
\def\Dio{{\tt Dio}}
\def\hGr{{\frac12\Gr}}
\def\RTr{{\tt RTr}}
\def\termTr{{\sf 1}_{\Tr}}
\def\termRTr{{\sf 1}_{\RTr}}
\def\ggGrc{{\tt ggGrc}}
\def\grad{{\rm grad}}
\def\bfH{{\mathbf H}}
\def\minGrc{{\mathfrak M}_{\Grc}}
\def\minSRTr{{\mathfrak M}_{\SRTr}}
\def\termGrc{{\sf 1}_{\Grc}}
\def\UPsilon{\Upsilon}
\def\GAmma{\Gamma}
\def\ttC{{\tt C}}
\def\ttA{{\tt A}}
\def\ttB{{\tt B}}
\def\ttD{{\tt D}}
\def\R{{\mathbb R}}
\def\lex{{\rm lex}}
\def\colim#1{\mathop{{\rm colim}}%
             \limits_{\rule{0em}{1em}\mbox{\scriptsize $#1$}}}
\def\Lev{{\tt Lev}}
\def\QV{{\tt QVrt}}
\def\-{{\!-\!}}
\def\bfT{{\boldsymbol {\EuScript T}}}
\def\bfS{{\boldsymbol {\EuScript S}}}
\def\Grc{\tt Grc}
\def\gTr{{\rm gTr}}
\def\gr{{\rm gr}}
\def\Free{{\mathbb F}}
\def\lTw{{\tt lTw}}
\def\Gr{{\tt Gr}}
\def\oucC{{\overline \ucC}}
\def\oDelta{{\overline \Delta}}
\def\ocC{{\overline \cC}}
\def\ucobar{{\underline \cobar}}
\def\ucC{{\underline \cC}}
\def\uD{{\underline \D}}
\def\In{{\rm In}}
\def\vert{{\rm vert}}\def\Vert{\vert}
\def\Ker{{\rm Ker}}
\def\J{{\mathfrak I}}
\def\cC{{\EuScript C}}
\def\Collectord{\hbox{$\Ord$-${\tt Coll}_1$}} 
\def\sgn{{\rm signum}}
\def\preim{{\rm preim}}
\def\can{{\rm can}}
\def\cobar{{\mathbb \Omega}}
\def\eul#1{{\chi\big(#1\big)}}
\def\ss{{\mathfrak s}}
\def\D{{\mathbb D}}
\def\-{\!-\!}
\def\+{\!+\!}
\def\uoO{{\underline\oO}}
\def\uoP{{\underline\oP}}
\def\susp{\uparrow \hskip -.35em}
\def\antishriek{{\hbox{\scriptsize\rm \raisebox{.2em}{!`}}}}
\def\desusp{\downarrow\!}
\def\pcirc{\hbox{\tiny $\diamondsuit$}}
\def\sspcirc{\hbox{\tiny $\diamondsuit^{\hbox{$\ss$}}_{\hbox{$r$}}$}}
\def\M{{\EuScript M}}
\def\K{{\underline {\EuScript K}}}
\def\f{{\gamma}}
\def\bbN{{\mathbb N}}
\def\fib{\triangleright}
\def\uAss{\underline{\hbox{$\mathscr A \hskip -.2em ss$}}}
\def\ainf{{\uAss\,}_\infty}
\def\pa{{\partial}}
\def\minimal{{\mathfrak M}}
\def\oC{{\EuScript C}}
\def\uoS{{\underline \oS}}
\def\oO{{\EuScript O}}
\def\Op#1{\hbox{$#1$-{\tt Oper}}}
\def\uF{\underline{\mathbb F}}
\def\coll#1#2{{\{#1(\underline #2)\}_{#2 \geq 1}}}
\def\F{{\mathbb F}}
\def\oS{{\EuScript S}}
\def\des{{\rm des}}
\def\Oterm{{\sf 1}_{\Ord}}
\def\Pterm{{\sf 1}_{\ttP}}
\def\Ptermplus{{\sf 1}_{\ttP^+}}
\def\Ord{{\mathbb \Delta_{\rm semi}}}
\def\comp{{\rm comp}}
\def\doubless#1#2{{
\def\arraystretch{.5}
\begin{array}{c}
\mbox{$\scriptstyle #1$}
\\
\mbox{$\scriptstyle #2$}
\end{array}\def\arraystretch{1}
}}
\def\Lin{\Vect}
\def\onsOp{{\rm ns}{\EuScript O}{\rm p}}
\def\oOp{{\EuScript O{\rm p}}}
\def\End{{\EuScript E}nd}
\def\oP{{\EuScript P}}
\def\ttO{{\tt O}}
\def\rada#1#2{{#1,\ldots,#2}}
\def\Rada#1#2#3{#1_{#2},\dots,#1_{#3}}
\def\CC{\hbox{\it CC}}
\def\Span{{\rm Span}}
\def\Set{{\tt Set}}
\def\term{\hbox {$\it per \hskip -.1em   {\mathcal A}s$}}
\def\tildeterm{\hbox {$\it per \hskip -.1em  {\widetilde {\mathcal A}s}$}}
\def\tw{\hbox {$\it per \hskip -.1em  {\widetilde {\mathcal A}s^!}$}}
\def\ssterm{\hbox {\scriptsize $\it per \hskip -.1em   {\mathcal A}s$}}
\def\PP{{\mathbb P}}
\def\id{1\!\!1}
\def\Coll{{\tt Coll}}
\def\epi{{\hbox{ $\twoheadrightarrow$ }}}
\def\bfk{{\mathbb k}}
\def\bbk{{\mathbb k}}
\def\bbR{{\mathbb R}}
\def\Vect{{\tt Vec}}
\def\ot{\otimes}
\def\inv#1{{#1}^{-1}}
\def\rada#1#2{{#1,\ldots,#2}}
\def\Surj{{\tt Surj}}
\def\Ass{\underline{{\mathcal A}{\it ss}}}
\def\Fin{{\tt Fin}}
\def\RTr{{\tt RTr}}
\def\ttP{{\tt P}}
\def\qb{quasi\-bijection}

\def\bbk{{\mathbb k}}
\def\H{{\mathbf H}}
\def\C{{\EuScript C}}
\def\Edg{{\rm edg}}
\def\N{{\mathcal N}}

\parskip3pt plus 1pt minus .5pt
\baselineskip 17.25pt  plus 1.5pt minus .5pt

\usetikzlibrary{arrows,decorations.markings}
\title[Minimal models for 
(hyper)operads]{Minimal models for graph-related (hyper)operads}
\author[M.~Batanin]{Michael Batanin}
\thanks{The first author acknowledges  the financial
support of Praemium Academi\ae\ of M.~Markl, of  Max Plank
Institut f\"{u}r Mathematik in Bonn and Institut des Hautes \'Etude Scientifiques in Paris.}
\author[M.~Markl]{Martin~Markl}
\thanks{The second author was supported by 
grant GA \v CR 18-07776S, Praemium Academi\ae\ and RVO: 67985840. He
also acknowledges the support by the National
Science Foundation under Grant No.~DMS-1440140 while he was
in residence at the Mathematical Sciences Research Institute in
Berkeley, California, during the Spring 2020 semester.}
\author[J.~Obradovi\'c]{Jovana~Obradovi\'c}
\thanks{The third author was supported by 
Praemium Academi\ae\ of M.~Markl and RVO: 67985840.}

\makeatletter
\@namedef{subjclassname@2020}{\textup{2020} Mathematics Subject Classification}
\makeatother


\address{M.M. and M.B.: Mathematical Institute, 
The Czech Academy of Sciences, {\v Z}itn{\'a} 25,\hfill\break 
         115 67 Prague 1, The Czech Republic}
\email{bataninmichael@gmail.com, batanin@math.cas.cz, markl@math.cas.cz}
\address{J.O.:  Mathematical Institute SANU,
   Knez Mihailova 36, p.f. 367, 11001 Belgrade,
   Serbia}
\email{joavana@mi.sanu.ac.rs}

\begin{abstract}
We construct explicit minimal models 
for the (hyper)operads governing modular, cyclic and ordinary
operads, and wheeled properads, respectively.
Algebras for these models are homotopy 
versions of the corresponding structures.
\end{abstract}

\maketitle
\bibliographystyle{plain}

\setcounter{secnumdepth}{3}
\setcounter{tocdepth}{1}


\tableofcontents

\section*{Introduction}

The fundamental feature of Batanin-Markl's theory of operadic
categories~\cite{duodel} is that the objects under
study are viewed as algebras over (generalized) operads in a specific
operadic category, cf.\ also the introduction to \cite{SydneyI}.
Thus, for instance, ordinary operads arise as algebras over the terminal
operad ${\sf 1}_\RTr$ in the operadic category $\RTr$ of rooted
trees, modular operads are algebras over the terminal operad
$\termggGrc$ in the operadic category $\ggGrc$ 
of genus-graded connected graphs, \&c. 

\vskip .5em
\noindent 
{\bf Our aim} is to construct explicit minimal models 
for the (hyper)operads governing modular, cyclic and ordinary
operads, and wheeled properads. We believe that the methods developed
here can be easily modified to obtain minimal models for operads
governing other common operad- or PROP-like
structures. According to general philosophy~\cite{haha},
algebras for these models describe strongly 
homotopy versions of the corresponding objects
whose salient feature is the transfer property over
weak homotopy equivalences. This might be compared to the following
classical~situation. 

Associative algebras are algebras over the non-$\Sigma$ operad $\Ass$. Algebras
over the minimal model of  $\Ass$ are Stasheff's strongly homotopy
associative algebras, also called $A_\infty$-algebras, 
\hbox{cf.~\cite[Example~4.8]{zebrulka}.} This
situation fits well into the framework of the current article,
since $\Ass$ is the terminal non-$\Sigma$ operad or, which is the same,
the terminal operad in the operadic category of finite
ordered sets and their order-preserving epimorphisms.

The case of strongly homotopy cyclic operads  was treated by the
third author
in~\cite{JO},
 while modular operads were addressed by B.~Ward
in~\cite{ward}. Both articles use the language of colored operads while 
the operadic category lingo used here is, as we
believe, more concise and efficient, after the necessary preparatory
material developed in~\cite{SydneyI, SydneyII} has been available.

\vskip .5em
\noindent 
{\bf Comparison with other approaches and the context.}
One of the major challenges of the theory of algebraic operads is to
understand their strongly homotopy versions, also called
$\infty$-operads or higher homotopy operads, and the related
deformation theory.  All approaches known to us are based on
the interpretation of the operads in question as algebras over a
specific `hyperoperad.' The strongly homotopy versions then appear as
algebras over a cofibrant, in some cases even minimal, resolution of
that hyperoperad. To construct the required resolutions, one tries to
mimics the methods of
the theory of `classical' algebraic~operads.  

Below we give a brief
account of the approaches preceding the present article, and compare
them to ours. The flavor of $\infty$-operads in spaces, see
e.g.~\cite{Lu,MW}, is quite different, so we do not discuss them here.

The first work that systematically treated operads as algebras over
a `hyperoperad' was
the 2003 preprint by P.~van der Laan~\cite{laan:03} who interpreted
nonsymmetric (non-$\Sigma$) operads as algebras over a~colored 
operad $\onsOp_{\mathbb N}$, with natural 
numbers ${\mathbb N}$ as the set of colors. More generally, nonsymmetric
$\frakC$-colored operads are algebras over a colored operad $\onsOp_\frakC$
with the colors
\[
\frakC^+ = \left\{
\left(
  \begin{array}{c}
    c
\\
c_1,\ldots,c_n
  \end{array}
\right), \ c, c_1,\ldots,c_n \in \frakC,
\ n \in {\mathbb N}\right\},
\]
where $c_1,\ldots,c_n$ are the input colors, and $c$ the color of
the output. In turn, nonsymmetric
$\frakC^+$-colored operads are algebras over a
$\frakC^{++}$-colored operad $\onsOp_{\frakC^{+}}$, where 
\[
\frakC^{++} =
 \left\{\left(
  \begin{array}{c}
    c^+
\\
c^+_1,\ldots,c^+_n
  \end{array}
\right), \ c^+, c^+_1,\ldots,c^+_n \in \frakC^+,
\ n \in {\mathbb N}\right\}, 
\]
and so on. Van der Laan proved that $\onsOp_{\mathbb N}$ is quadratic
Koszul, with the binary generators
\begin{equation}
\label{Neco se sitnici?}
\circ_i \in \colorop {\onsOp}(m,n; m+n-1),\ m,n \in  {\mathbb N}, \
1 \leq i \leq m,
\end{equation}
representing the partial compositions
\[
\circ_i : \oP(m) \ot \oP(n) \longrightarrow \oP(m\!+\!n\!-\!1),\ m,n \in  {\mathbb N}, \
1 \leq i \leq m.
\]

The operad $\onsOp_\frakC$ with an arbitrary set
of colors admits a similar presentation. 
Since $\onsOp_\frakC$ is quadratic Koszul, it has a nice
canonical resolution whose algebras
are strongly homotopy non-symmetric $\frakC$-colored operads. An independent
combinatorial description of this  resolution  was
given by the
third author in~\cite{JO}.

The case of {\em symmetric\/} operads is dramatically different. They
are algebras over an ${\mathbb N}$-colored operad $\oOp$ generated,
along with the quadratic generators~(\ref{Neco se sitnici?}), also by
the linear~ones
\begin{equation}
\label{Za chvili na pohotovost - bojim se moc.}
g_\sigma \in \colorop {\oOp}(n;n), \ n \in  {\mathbb N},\ \sigma \in
\Sigma_n \setminus \{1_n\},
\end{equation}
encoding the symmetric group action.
A cofibrant resolution of $\oOp$ was described, using the curved
Koszul duality, by M.~Dehling and B.~Vallette in the
fascinating article~\cite{deh-val}. Strongly
homotopy symmetric operads appearing in this way 
involve also resolutions of the symmetric group action.

Since our main applications, such as
modular operads or wheeled PROPs, live over a field of characteristic zero,
we do not want to touch the actions of the symmetric groups, but
have them hidden in the toolbox.   
Good analogy is the Koszul duality theory for
algebraic operads~\cite[Chapter~3]{MSS}, where the symmetric group
actions enter the picture 
already at the level of the generating collections.

Group actions can be incorporated with the use of (hyper)operads
whose colors are objects of groupoids.
For instance, such a  groupoid-colored operad governing symmetric operads
does not require 
generators~(\ref{Za chvili na pohotovost - bojim se moc.}), since the
symmetric group actions appear as the  morphism spaces
\[
{\rm Hom}_\bbN(m,n) =
\begin{cases}
\emptyset & \hbox {if $m \not = n$, and}
\\
\Sigma_m   & \hbox {if $m =n$.}
\end{cases}
\]
between the colors in $\bbN$.  Groupoid-colored operads are
equivalent to Feynman categories of Kaufman and Ward \cite{BKW,KW}
and, indeed, a foundation of Koszul duality for operads in the context
of Feynman categories has been developed in \cite{KW,ward0,ward}.

In our approach, the `hyperoperads' are operads, in the generalized
sense of~\cite[Definition~1.11]{duodel}, over certain
operadic categories of graphs which already contain the symmetric group 
actions as particular automorphisms.  Having the symmetric group
actions swept under the carpet, our setup is
analogous to the classical theory of algebraic operads in
characteristic~zero.

Let us explain how symmetric operads used above as an example 
are treated in our approach. 
They appear as algebras for the terminal operad $\termRTr$ over the
operadic category $\RTr$ of rooted trees, see Subsection~\ref{Prvni
  tyden v Berkeley konci.} for details. 
The vertices of the trees in $\RTr$ 
are linearly ordered, and also the incoming edges of each vertex
and the legs are
linearly ordered by the local resp.~global orders, cf.~Figure~\ref{Pristi tyden mne ceka laborator.}. 

The orders  determine  which operation a given rooted tree
represents.
For instance, the tree in Figure~\ref{Pristi tyden mne
  ceka laborator.} represents the operation that to elements $a_1,a_2 \in
\oP(2)$ and $a_3 \in \oP(3)$  of 
a symmetric operad $\oP$ assigns the operation  in $\oP(5)$  acting on the
`variables' $x_1,\ldots,x_5$ as
$
a_2\big(a_3(x_4,x_1,x_2),a_1(x_5,x_3)\big).
$
The category $\RTr$ contains $2! \! \times\! 2!\! \times\! 3!$ rooted trees of
the same shape, with the same global order, with the same order of the vertices,
but with possibly different local orders. All these trees are related by `local
isomorphisms' which incorporate the symmetric groups actions to  our approach.
A schematic picture of a
configuration of `pancakes' describing operations of cyclic 
operads can be found in~\cite[pages~95-96]{DMJ}.
\begin{figure}[t]
  \centering
\[
\psscalebox{1.0 1.0} 
{
\begin{pspicture}(0,-2.805)(7.6,2.805)
\definecolor{colour0}{rgb}{0.99215686,0.9843137,0.9843137}
\definecolor{colour0}{rgb}{0.92156863,0.9137255,0.9137255}
\psline[linecolor=black, linewidth=0.04, arrowsize=0.21667cm 2.0,arrowlength=1.7,arrowinset=0.0]{<-}(3.4,2.195)(3.4,0.195)
\psline[linecolor=black, linewidth=0.04, arrowsize=0.05291667cm 2.0,arrowlength=1.4,arrowinset=0.0]{<-}(3.4,0.195)(1.6,-0.405)
\psline[linecolor=black, linewidth=0.04, arrowsize=0.05291667cm 2.0,arrowlength=1.4,arrowinset=0.0]{<-}(1.6,-0.405)(0.4,-2.205)
\psline[linecolor=black, linewidth=0.04, arrowsize=0.05291667cm 2.0,arrowlength=1.4,arrowinset=0.0]{<-}(1.6,-0.405)(2.6,-2.205)
\psline[linecolor=black, linewidth=0.04, arrowsize=0.05291667cm 2.0,arrowlength=1.4,arrowinset=0.0]{<-}(3.4,0.195)(5.4,-0.405)
\psline[linecolor=black, linewidth=0.04, arrowsize=0.05291667cm 2.0,arrowlength=1.4,arrowinset=0.0]{<-}(5.4,-0.405)(3.8,-2.205)
\psline[linecolor=black, linewidth=0.04, arrowsize=0.05291667cm 2.0,arrowlength=1.4,arrowinset=0.0]{<-}(5.4,-0.405)(5.6,-2.205)
\psline[linecolor=black, linewidth=0.04, arrowsize=0.05291667cm 2.0,arrowlength=1.4,arrowinset=0.0]{<-}(5.4,-0.405)(7.2,-2.205)
\pscircle[linecolor=black, linewidth=0.04, fillstyle=solid,fillcolor=colour0, dimen=outer](5.4,-0.405){0.4}
\pscircle[linecolor=black, linewidth=0.04, fillstyle=solid,fillcolor=colour0, dimen=outer](3.4,0.195){0.4}
\pscircle[linecolor=black, linewidth=0.04, fillstyle=solid,fillcolor=colour0, dimen=outer](1.6,-0.405){0.4}
\rput(1.6,-0.45){$a_1$}
\rput(3.4,0.15){$a_2$}
\rput(5.4,-0.43){$a_3$}
\psframe[linecolor=black, linewidth=0.04, fillstyle=solid,fillcolor=colour0, dimen=outer](7.6,-2.205)(6.8,-2.805)
\psframe[linecolor=black, linewidth=0.04, fillstyle=solid,fillcolor=colour0, dimen=outer](6.0,-2.205)(5.2,-2.805)
\psframe[linecolor=black, linewidth=0.04, fillstyle=solid,fillcolor=colour0, dimen=outer](4.2,-2.205)(3.4,-2.805)
\psframe[linecolor=black, linewidth=0.04, fillstyle=solid,fillcolor=colour0, dimen=outer](3.0,-2.205)(2.2,-2.805)
\psframe[linecolor=black, linewidth=0.04, fillstyle=solid,fillcolor=colour0, dimen=outer](0.8,-2.205)(0.0,-2.805)
\psframe[linecolor=black, linewidth=0.04, fillstyle=solid,fillcolor=colour0, dimen=outer](4,2.795)(2.85,2.195)
\rput[t](0.4,-2.4){$x_5$}
\rput[t](2.6,-2.4){$x_3$}
\rput[t](3.8,-2.4){$x_2$}
\rput[t](5.6,-2.4){$x_4$}
\rput[t](7.2,-2.4){$x_1$}
\rput[bl](3.05,2.32){{\rm root}}
\rput[bl](0.8,-1.2){\scriptsize 1}
\rput[bl](2.2,-1.2){\scriptsize 2}
\rput[bl](2.4,0.05){\scriptsize 2}
\rput[bl](4.6,0.025){\scriptsize 1}
\rput[bl](4.3,-1.2){\scriptsize 3}
\rput[bl](5.7,-1.605){\scriptsize 1}
\rput(6.4,-1.005){\scriptsize 2}
\end{pspicture}
}
\]
\caption{\label{Pristi tyden mne ceka laborator.}A tree in $\RTr$.}
\end{figure}

\vskip .5em
\noindent 
{\bf Warning.} 
As in the approach based on colored operads, the
objects we study appear in our approach as algebras for a
certain, in most but not all cases terminal, `hyperoperad' over a
suitable operadic category, though they themselves {\it need not}  be operads in
the sense of~\cite{duodel}. Thus, for instance, there is no operadic
category having cyclic operads as its operads, but cyclic operads are
algebras for the terminal `hyperoperad' over the operadic category of trees.   

\vskip .5em
\noindent 
{\bf The models.} Here we point to the places
where the advertised constructions can be found. 

$\bullet$
The minimal model
$\minggGrc$ of the operad $\termggGrc$ governing modular operads 
is constructed in Subsection~\ref{Podari se mi koupit to auto?}. 
Algebras for this minimal model are strongly homotopy modular operads.

$\bullet$
The minimal model
$\minTr$ of the operad $\termTr$ governing cyclic operads is constructed in
Subsection~\ref{Dnes prvni vylet na kole}. 
Algebras for this minimal model are strongly homotopy cyclic operads.

$\bullet$
The minimal model
$\minWhe$ of the operad $\termWhe$ governing wheeled properads is constructed in
Subsection~\ref{Poslu Jarce obrazky kyticek.}. 
Algebras for this minimal model are strongly homotopy wheeled properads.

$\bullet$
There are two operadic categories such that the algebras
for their terminal operads are ordinary operads -- the category
$\RTr$ of rooted trees and its full subcategory $\SRTr$ of strongly rooted
trees. The minimal models $\minRTr$ resp.\
$\minSRTr$ of the corresponding terminal operads $\termRTr$ resp.\ 
$\termSRTr$ are constructed in
Subsections~\ref{Prvni tyden v Berkeley konci.} resp.~\ref{Ze by byla
  monografie uz konecne dokoncena?}. Both  $\minRTr$ and
$\minSRTr$ have the same algebras, namely
strongly homotopy ordinary operads. The reason why
we consider two categories governing the same structures is explained below.

\vskip .5em
\noindent {\bf Methods used.}  We begin with the particular case
of the operadic category $\Grc$ of connected graphs. Algebras for the
terminal operad $\termGrc$ in that category are modular operads
without the genus grading. We explicitly
define, in Section~\ref{hadrova_panenka}, a~minimal $\Grc$-operad
$\minGrc = (\Free(D),\pa)$ and a map
$\minGrc\stackrel\rho\longrightarrow \termGrc$ of differential graded
$\Grc$-operads. Theorem~\ref{Woy-Woy} states that $\rho$ is 
a~level-wise homological isomorphism, meaning that $\minGrc$ is a
minimal model of $\termGrc$. Proof of Theorem~\ref{Woy-Woy} is a
combination of the following~facts.

On one hand, using the apparatus developed in \cite{SydneyII}, we
describe, in Subsection~\ref{Dnes_je_Michalova_oslava.}, the piece
$\Free(D)(\Gamma)$, $\Gamma \in \Grc$, of the free operad $\Free(D)$
as a colimit over the poset ${\gTr}(\Gamma)$ of graph-trees
associated to $\Gamma$, which are abstract trees whose vertices are
decorated by graphs from $\Grc$ and which fulfill suitable
compatibility conditions involving $\Gamma$.

On the other hand, to each $\Gamma \in \Grc$ we associate, in
Subsection~\ref{Srni}, a hypergraph
$\bfH_\Gamma$ and to that hypergraph a poset ${\mathcal
  A}(\bfH_\Gamma)$ of its constructs, which are certain abstract trees
with vertices decorated by subsets of the set of internal edges of
$\Gamma$. We prove, in Proposition~\ref{jovanica}, 
that the poset  ${\gTr}(\Gamma)$  
is order-isomorphic to the poset ${\mathcal A}(\bfH_\Gamma)$.

Lemma~\ref{abspol} 
asserts that ${\mathcal
  A}(\bfH_\Gamma)$ is in turn order-isomorphic to the face lattice of a certain convex
polytope ${\mathcal C}(\bfH_\Gamma, \Pi)$. The construction of this polytope is a generalization of the construction from \cite{CIO} and has an interesting interpretation in terms of game theory. This game theoretic interpretation is not a central theme of our paper but we decided to include a brief description of this topic because it opens up some new perspectives on operad theory and  can also be useful in calculations. 

Finally, using the `ingenious'
Lemma~\ref{Jeste_ani_nevim_kde_budu_v_Melbourne_bydlet.}, 
we show that the faces of ${\mathcal C}(\bfH_\Gamma, \Pi)$ can be
oriented so that the cellular chain complex of ${\mathcal C}(\bfH_\Gamma, \Pi)$ is isomorphic, as a differential graded vector
space, to  $(\Free(D)(\Gamma),\partial)$. Since  ${\mathcal C}(\bfH_\Gamma, \Pi)$ is
acyclic in positive dimension, the same must be true for
$(\Free(D)(\Gamma),\partial)$. It remains to show that $\rho$ induces
an isomorphism of degree $0$ homology, but this is simple. The
conclusion 
is that $\minGrc$ is indeed a minimal model of $\termGrc$.

In constructing the minimal models of the terminal operads $\termggGrc$,
$\termTr$ and $\termWhe$ in the operadic categories $\ggGrc$ of genus-graded
connected graphs, $\Tr$ of trees and $\Whe$ of 
ordered (`wheeled') connected graphs, respectively, 
we use the fact observed in~\cite[Section~4]{SydneyI} that
these categories are discrete operadic opfibrations over $\Grc$.
Their minimal models are then, thanks to
Corollary~\ref{grantova_zprava}, the restrictions of the minimal model for $\termGrc$
along the corresponding opfibration map. 

The situation of the terminal operad $\termRTr$ in the operadic
category $\RTr$ of rooted trees is different, since this category is 
not an opfibration over $\Grc$. It is, however, a 
discrete operadic fibration with finite fibers, so
Corollary~\ref{grantova_zprava} of Section~\ref{Mourek a Terezka} applies as well.

We finally introduce a full subcategory $\SRTr \subset \RTr$
consisting of strongly rooted trees. The algebras for the terminal
$\SRTr$-operad $\termSRTr$ are the same as $\termRTr$-algebras, i.e.\
ordinary operads.  We consider this subcategory since it is the most
economic description of ordinary operads. Although it is neither a
fibration, nor an opfibration over~$\Grc$, we show in
Subsection~\ref{Prvni tyden v Berkeley konci.} that the minimal model
for $\termSRTr$ can be obtained by a straightforward modification of
the construction of the minimal model for $\termGrc$ given
in~Section~\ref{hadrova_panenka}.

\vskip .5em
\noindent {\bf Limitations and generalizations.}
Minimal models studied via the methods developed in the present
work appear as the cellular chain complexes of sequences of contractible
polytopes. Since the homology of
such complexes is a one-dimensional vector space sitting in
degree $0$,  our approach clearly applies only to (hyper)operads that are
terminal in an appropriate category of (hyper)operads.
This limitation however still leaves room for the study of structures such as  
non-symmetric modular operads and modular hybrids introduced
in~\cite{OC}, 
dioperads~\cite{gan}, and a couple of others not
addressed in the present article. 

An interesting situation occurs for terminal (hyper)operads that are 
quadratic but not Koszul self-dual. 
This is the case, e.g., of the operad $\termggGrc$
governing modular operads. The proof of~\cite[Theorem~9.6]{SydneyII} 
establishing the Koszulity of $\termggGrc$ uses the explicit minimal
model $\minggGrc$ constructed in the present paper. The dual dg operad
${\mathbb D}(\termggGrc)$ then in turn provides an explicit minimal model for the
Koszul dual of $\termggGrc$, which is the non-terminal 
operad~$\oddGr$ whose algebras are odd modular
operads. See Sections~5 and~9 of~\cite{SydneyII} for the terminology
and definitions. 
The methods of the present paper
may therefore lead to explicit minimal models for some non-terminal
(hyper)operads as well.

\vskip .5em
\noindent 
{\bf Applications.} As explained e.g.\ in~\cite{markl12:defor} or~\cite{ib}, 
an explicit minimal model of a traditional operad $\oP$ leads 
to an explicit $L_\infty$ (= strongly homotopy Lie) 
algebra which, via the related simplicial
Maurer-Cartan space, provides full information about the moduli space
of deformations of $\oP$-algebras. We believe that the same is true
also in the generalized context of this paper.

In particular, the constructions presented here should provide
understanding of deformations of modular, cyclic and traditional
operads, as well as wheeled properads, including the associated
cohomology theory and higher homotopy operations analogous to the Massey
products for modular operads constructed in~\cite{ward}. Since all
minimal models constructed here possess quadratic differentials, the
governing $L_\infty$-algebra is actually just the `ordinary' differential
graded Lie algebra, which makes the related theory conceptually very simple.

\vskip .5em
\noindent 
{\bf Plan of the paper.} In Section~\ref{mexicke pivo} 
we recall necessary facts
about hypergraph polytopes, and free operads in operadic categories.
Section~\ref{hadrova_panenka} 
is devoted to the construction of the minimal model for
the terminal $\Grc$-operad, and presentation of the necessary
preparatory material. Section~\ref{Mourek a Terezka} 
addresses minimal models for
terminal operads in the operadic categories of genus-graded graphs,
trees, wheeled graphs and strongly rooted trees.

\vskip .5em
\noindent
{\bf Conventions.} 
All algebraic objects will be considered over a
field~$\bbk$ of characteristic zero. 
By $|X|$ we denote either the cardinality if $X$ is a finite set, or the geometric
realization if $X$ is a graph. If not specified otherwise,  
(hyper)operads featured here will
live in the monoidal category of differential graded $\bfk$-vector spaces. The {\em
  terminal operad\/} in a given operadic category is the one all of
whose components equal
$\bbk$ and whose structure operations are the identities. These operads are
linearizations of the corresponding terminal set-operads, which hopefully
justifies our relaxed terminology.

\noindent
{\bf Acknowledgment.}
We express our gratitude to the anonymous referee for his/her 
useful suggestions and comments that 
led to substantial improvement of our~paper. 

\section{Recollections}
\label{mexicke pivo}

This section contains a preparatory material regarding hypergraph
polytopes and operadic categories.    
The basic references are~\cite{CIO,JO} for the former
and~\cite{SydneyI,SydneyII,duodel} for the latter.

\subsection{Hypergraph polytopes}
They are abstract polytopes whose geometric realization can be
obtained by truncating the vertices, edges and other faces of
simplices, in any finite dimension. In particular, the family of
$n$-dimensional hypergraph polytopes consists of an interval of simple
polytopes starting with the $n$-simplex and ending with the
$n$-dimensional permutahedron.

 \subsubsection*{Hypergraph terminology}
 
 For a set $H$ let ${\mathcal P}(H)$ be its power-set.
 
 A hypergraph is a pair ${\bf H}=(H,{\bf H})$ of a finite set $H$ of
 {\em vertices} and a subset
 ${\bf H}\subseteq {\mathcal P}(H)\backslash\emptyset$ of {\em
   hyperedges}, such that, for all
 $x\in H$, $\{x\}\in {\bf H}$ (note that this property implies 
that $\bigcup {\bf H}=H$ and  justifies the
 convention to use the bold letter ${\bf H}$ for both the hypergraph
 itself and its set of hyperedges). A hypergraph ${\bf H}$ is {\em
   connected} if there are no non-trivial partitions $H=H_1\cup H_2$,
 such that
 \[
{\bf H}=\{X\in {\bf H}\,|\, X\subseteq H_1\}\cup\{Y\in {\bf H}\,|\,
 Y\subseteq H_2\}.
\] 
A hypergraph ${\bf H}$ is {\em saturated} when,
 for every $X,Y\in{\bf H}$ such that $X\cap Y\neq\emptyset$, we have
 that $X\cup Y\in {\bf H}$. Every hypergraph can be saturated by
 adding the missing (unions of) hyperedges. Let us introduce the
 notation $${\bf{H}}_X:=\{Z\in {\bf{H}}\,|\, Z\subseteq X\},$$ for a
 hypergraph ${\bf H}$ and $X\subseteq H$.  The {\em saturation} of
 ${\bf{H}}$ is then formally defined as the hypergraph
$${\it Sat}({\bf{H}}):=\{ X\,|\, {\emptyset\subsetneq X\subseteq H\;\mbox{and}\;{\bf{H}}_X\;\mbox{is connected}}\}.$$

For a hypergraph ${\bf H}$ and $X\subseteq H$, we  also set
 $${\bf H}\backslash X:={\bf{H}}_{H\backslash X}.$$ Observe that for each finite hypergraph there exists a partition
$H=H_1\cup\ldots\cup H_m$, such that each hypergraph ${\bf{H}}_{H_i}$ is connected and ${\bf{H}}=\bigcup({\bf{H}}_{H_i})$.  The ${\bf{H}}_{H_i}$'s are called the {\em connected components} of ${\bf{H}}$. We shall  write
${\bf{H}}_i$  for
${\bf{H}}_{H_i}$.     We shall use the notation 
\begin{center}
${\bf{H}}\backslash X  \leadsto {\bf H}_1,\ldots,{\bf H}_n$ 
\end{center}
to indicate that  ${\bf H}_1,\ldots,{\bf H}_n$ are  the   connected components of ${\bf H}\backslash X$.  

\subsubsection*{Abstract polytope of a hypergraph}  
We next recall from \cite{CIO} the definition of the abstract polytope
$${\mathcal A}({\bf H})=(A({\bf H})\cup\{\emptyset\},\leq_{\bf H})$$
associated to a connected hypergraph ${\bf H}$.
 
 \medskip
 
The elements of the set $A({\bf H})$, to which we refer as the {\em  constructs} of   ${\bf H}$, are  the non-planar, vertex-decorated rooted trees defined recursively as follows. 
  \begin{itemize}
  \item[{\texttt{(C0)}}] If ${\bf H}$ is the empty hypergraph, then
    $A({\bf H})=\{\emptyset\}$, i.e. ${\mathcal A}({\bf H})$ is the
    singleton poset containing $\emptyset$.
\end{itemize}

 Otherwise, let $\emptyset\neq X\subseteq H$ be a subset of the set of vertices of ${\bf H}$.
  \begin{itemize}
  \item[{\texttt{(C1)}}] If $X=H$, then the abstract rooted tree with a single vertex   labeled by $X$ and without any  inputs, is a construct of ${\bf H}$; we denote it by $H$.
  \item[{\texttt{(C2)}}] If $X\subsetneq H$,   if   ${\bf H}\backslash X  \leadsto {\bf H}_1,\ldots,{\bf H}_n$, and if  $C_1,\ldots,C_n$ are constructs of ${\bf H}_1,\ldots,{\bf H}_n$, respectively, then the
tree whose root vertex is decorated by $X$ and that has $n$ inputs, on which the respective $C_i\,$'s are grafted,  
 is a construct of ${\bf H}$; we denote it by $X\{C_1,\ldots,C_n\}$.
  \end{itemize}

In what follows, we shall refer to the vertices of constructs by the sets
decorating them, since they are a fortiori all distinct. The notation 
$C:\bf{H}$  will mean that  $C$ is a construct of~$\bf{H}$. 
  
\medskip The partial order $\leq_{\bf H}$ on non-empty constructs is
generated by the 
edge-contraction:
\[
Y\{X\{C_{11},\ldots,C_{1m}\},C_2,\ldots,C_n\}\leq_{\bf H}(Y\cup
X)\{C_{11},\ldots,C_{1m},C_2,\ldots,C_n\}
\]
and the relation
\[
\hbox{if $C'_1 \leq_{\bf H_1} C''_1$ \ then \ $X\{\rada {C_1'}{C_n}\}
\leq_{\bf H}X\{\rada {C_1''}{C_n}\}$}.
\]
In addition, for each
construct $C$ of ${\bf H}$, we have that $\emptyset\leq_{\bf H} C$.

\smallskip

The faces of ${\mathcal A}({\bf H})$ are ranked by integers  ranging
from $-1$ to $|H|-1$. The face $\emptyset$ is the unique face of rank
$-1$, whereas the rank of a construct $C$ is  $|H|-|\mbox{vert}(C)|$.
In particular, constructs  whose vertices are all decorated with singletons
are faces of rank $0$, whereas  the  construct $H$   
is the unique face of rank $|H|-1$.
 
 \subsubsection*{Convex realization of ${\mathcal A}({\bf H})$ as core
   of a game}   

We recall some terminology of cooperative game theory \cite{Sh}. Let
$H$ be a finite set. A cooperative game of $n=|H|$ players is a
function 
\[
\Pi: {\mathcal P}(H)\backslash\emptyset \to \mathbb{R}_{\geq 0}.
\] 
A classical  interpretation of such a game is that every nonempty
subset $I = \{i_1,\ldots,i_k\}\in {\mathcal P}(H)\backslash\emptyset $
(called {\it coalition of players}) has certain utility $\Pi(I)$
(blocking power of the coalition) in its disposition which can be
distributed among members of $I.$   An outcome of such a distribution
is a real valued vector $(x_{i_1},\ldots , x_{i_k})\in
\mathbb{R}^{I}_{\geq 0}$ such that $\sum_{i\in I} x_i \le \Pi(I)$.  
Let 
\[
\textstyle
\pi^+_I:=\big\{(x_1,\dots,x_n) \in \mathbb{R}^{\times n}_{\geq 0}\,|\, \sum_{i\in I}x_i\geq \Pi(I)\big\}
\]
and
\[
\textstyle
\pi_I:=\big\{(x_1,\dots,x_n) \in \mathbb{R}^{\times n}_{\geq 0}\,|\, \sum_{i\in I}x_i =\Pi(I)\big\}.
\]
Then  {\it the core of the  game} is defined as a convex set
  \[
{\mathcal C}(\Pi):=\bigcap_{I\in \mathcal{P}(H)\backslash\{H\}} \pi^+_I \cap \pi_H.
\] 
The vectors of this set can be interpreted as some sort of stable distributions among players, where there is no intention among players to form a smaller coalition which can deliver a~better distribution of utility for its members, that is, no smaller coalition would wish to block such a distribution.

A cooperative game is called {\it convex} if the inequality 
\[
\Pi(X\cup Y) \ge \Pi(X) + \Pi(Y) - \Pi(X\cap Y)
\] 
holds for all $X,Y\in \mathcal{P}(H)\backslash\emptyset$. For a {\em
  strictly convex game\/} (also known as {\em upper supermodular
  function\/} in lattice theory) this inequality holds strictly for all $X,Y$
such that $X$ is not a subset of $Y$ or $Y$ is not a subset of $X$.  A
classical result of Shapley~\cite{Sh} is that the core of a convex
game is nonempty and, moreover, the core of a strictly convex game is
an $(n-1)$-dimensional convex polytope which is combinatorially equivalent to the permutahedron on $n$ letters.  

\begin{example}\label{G} Let $\mathcal{G}$ be a cooperative game given by the  function $\mathcal{G}(I) = 3^{|I|}$ for a  coalition $I\subset H.$ It is not hard to check that this is a strictly convex game. In fact, this was checked implicitly by  Do\v sen and Petri\'c in \cite[Lemma 9.1]{DP-HP}. 
\end{example}

\begin{example} \label{L} Let $\mathcal{L}$ be a cooperative gave
  given by the  function 
\[
\mathcal{L}(I) = 1+2+ \cdots + |I| = |I|(|I|+1)/2.
\] 
This is again a strictly convex game which is easy to check. The core of this game is the classical convex realization of the permutahedron $P_n$ as the convex hull of \hbox{$\{\sigma(1,\ldots,n) \ | \ \sigma\in S_n\}$}.   

\end{example} 
The reader is referred to \cite{Sh} for a zoo of various examples of convex games.

 Let now ${\bf H}=(H,{\bf H})$ be a hypergraph. An {\it $\bf H$-cooperative game} is a cooperative game $\Pi$ of $|H|$-players. We call such a game {\it (strictly) convex} if $\Pi$ is (strictly) convex game. We now  want to adapt the concept of the core of such a game to take into account the hypergraph structure. Namely, we forcibly (that is by law) forbid to form coalitions which are not the hyperedges of the saturation of $\bf H.$  That is     
\[
{\mathcal C}({\bf H},\Pi):=\bigcap_{Y\in {\it Sat}({\bf
  H})\backslash\{H\}}\pi^+_I \cap \pi_H.
\] 
 Another way to say it is that we form a new game in which all coalitions which are not the hyperedges of the saturation of $\bf H$ have blocking power $0$ but the blocking power of other coalitions remains the same.  Obviously the core of such a game is precisely ${\mathcal C}({\bf H},\Pi).$ 
\begin{lemma}
\label{abspol} Let $\Pi$ be a strictly convex $\bf H$-cooperative game. 
The poset ${\mathcal A}({\bf H})$ is 
order-isomorphic to the face lattice of a convex polytope ${\mathcal C}({\bf H},\Pi)$
obtained as a truncation of the \hbox{$(|H|\!-\!1)$}-dimensional simplex. In
particular, ${\mathcal A}({\bf H})$ is
an abstract polytope of rank \hbox{$|H|\!-\!1$}.
\end{lemma}

\begin{proof}
The order-isomorphism between the poset of constructs of ${\bf H}$ and
the poset of geometric faces of ${\mathcal C}({\bf H},\mathcal{G})$, where $\mathcal{G}$ is the game from the Example~\ref{G},  is defined in
\cite[Section 3.3]{CIO} (this polytope was denoted  ${\mathcal
  G}(\mathbf{H})$ there). The fact that  ${\mathcal A}({\bf H})$ is
an abstract polytope follows in this particular case 
from Lemma 9.1 of \cite{DP-HP}.
It is however not hard to check that the arguments 
of \cite{CIO} and \cite{DP-HP} work  for any strictly
convex game $\Pi$ instead of $\mathcal{G}$. 
The lemma  can also be deduced from the classical combinatorial description of the core by Shapely \cite{Sh}.
\end{proof}

\begin{remark} One can use the game $\mathcal L$ from Example \ref{L}
  to get Loday's realization of the associahedron and its generalizations,
  see \cite{CSZ,Laplante-Anfossi} for a survey and comparison of
  different convex realizations of generalized associahedra and other
  polytopes, and also~\cite{CD,Dev} or~\cite{sta} for earlier sources
  addressing the problem of geometric realization of polytopes.  
\end{remark} 

\subsection{Free operads in the operadic category of graphs}
\label{Dnes_je_Michalova_oslava.}

The basic operadic category in this section will be the category
$\Grc$ of connected ordered 
graphs introduced in \cite[Section~3]{SydneyI} and
Example~5.7 loc.~cit.\ to which we also refer for terminology and
notation. Results for
other categories of graphs will be  straightforward modifications of
this situation. Recall that the adjective {\em ordered\/} means that the (finite) 
set of vertices of $\GAmma$ is (linearly) ordered, 
as well as are the (finite) sets of half-edges 
adjacent to each vertex of $\GAmma$,
and that also the (finite) set of  legs of $\GAmma$ is ordered. To simplify
the terminology, by a {\em graph\/} we always mean in this section an
object of $\Grc$.
As the first step in describing the component 
$\Free(E)(\GAmma)$, $\GAmma \in \Grc$, of the free
operad $\Free(E)$ generated by a $1$-connected collection $E$ we
identify, in Theorem~\ref{Uz_nechci_letat_zakladni_vycvik.} below,  
the set $\pi_0(\lTw(\GAmma))$ of connected components of the
groupoid $\lTw(\GAmma)$ of labelled towers \cite[Section~3]{SydneyII}
with a certain class of trees defined~below. 
Recall that we work with a skeletal version of the category of finite
ordered sets, therefore two arbitrary order-isomorphic finite sets are the same.

Recall that a map of graphs is a {\em quasibijection\/} if all its
fibers are trivial, i.e.\ are corollas whose local and global
orders agree~\cite[Section~3]{SydneyI}. 
By~\cite[Lemma~3.15]{SydneyI}, all quasibijections in $\Grc$
are isomorphisms. A map of graphs is called {\em
  order-preserving\/} if the induced map of vertices preserves the
orders. An order-preserving map is {\em elementary\/} if all its
fibers are trivial except precisely one  which is
required to have at least one internal edge.

Before we continue, we introduce a particular class of maps between
graphs, called {\em canonical contractions\/} (or {\em cc's\/} for
short) of a subgraph. The informal definition is the following.

Let $\Gamma \subset \Gamma'$ be a subgraph and $\Gamma''$ be obtained
from $\Gamma'$ by contracting all internal edges of $\Gamma$ into a
vertex. The canonical
contraction $\pi: \Gamma' \to \Gamma''$ is then the `obvious
projection.' We however need to specify labellings and orders of the
vertices and flags of $\Gamma$ and $\Gamma''$, so a more formal
definition is needed.

Assume that $\Gamma'= (V',F')\in \Grc$ is a graph with the set of vertices
$V'$, the set of flags $F'$ and the structure map $g': F' \to V'$, see 
\cite[Definition~3.1]{SydneyI}. Choose a nonempty subset $V \subset V'$ and 
a~nonempty set $E$ of edges of $\Gamma'$ formed by the half-edges in $g'^{-1}(V)
\subset F'$ such that the subgraph of $\Gamma'$ spanned by $E$ is
connected. Let us denote by $V'/V$ the ordered set
\[
V'/V := (V' \setminus V) \cup \{\min (V)\};
\]
the notation being justified by the  canonical
set-isomorphism of $V/V'$ as above with the set-theoretic quotient $V'$
by the subset $V$. Let finally $V'' := V'/V$ 
and
\begin{equation} 
\label{Dnes_s_Jarkou_a_Oliverem.}
\phi : V' \to V'' = V'/V
\end{equation}
be the `projection' that is the identity on
$V' \setminus V$ while it sends all elements of $V$ to $\min (V)$.

We construct $\Gamma''$ as the graph whose set of
  vertices is $V''$ and whose set of flags is $F'':=
  F' \setminus E$.  The defining map $g'':F''\to V''$ is the
  restriction of the composite $\phi \circ g'$, as  in
\begin{equation}
\label{Treti den v Berkeley}
\xymatrix@C=3.5em{F' \ar[d]_{g'}  &  \  F'':=
  F' \setminus E \ar@{_{(}->}[l]_(.7){\psi}
\ar[d]^{g''}
\\
V' \ar@{->>}[r]^{\phi} &V''.
}
\end{equation}
The involution $\sigma'' : F'' \to F''$ is the restriction of the
involution $\sigma' : F' \to F'$ of $\Gamma'$.
The map $g''$ defined by~(\ref{Treti den v Berkeley})
is however not order-preserving as
required by the definition of a graph. We therefore reorder $F''$ by
imposing the lexicographic order requiring that, for  $a,b \in
F''$,
\[
a< b \hbox { if and only if } 
\begin{cases}
\hbox{$g''(a) < g''(b)$ in $V''$, or}
\\
\hbox{$g''(a) = g''(b)$ and $a < b$ in $F''$}.   
\end{cases}
\]
This formula obviously does not change the local orders of flags in
$F''$ around a given vertex.

We finally define the cc
$\pi:\Gamma'\to \Gamma''$ as the couple $(\psi,\phi)$ with
$\psi:F''\hookrightarrow F'$ the inclusion. The unique nontrivial
fiber of $\pi$ is the graph $\Gamma$ given by the restriction 
$F\stackrel{g}{\longrightarrow} V$ of $g'$ to $F := g'^{-1}(V)$
whose involution is trivial everywhere except for the flags forming the
edges in $E$, in which case it coincides with the involution of
$\Gamma'$. A simple example of a canonical contraction can be found in
Figure~\ref{Vcera jsem si koupil kolo.} below.
 
We may sometimes loosely denote $\Gamma'' := \Gamma'/\Gamma$. Canonical
contractions in the above sense are modifications of pure contractions
of \cite[Definition~3.5]{SydneyI} in that here we do not
require the map of vertices to be order-preserving, which is
compensated by introducing the lexicographic order on the flags of
$\Gamma''$. Canonical contractions are close to elementary morphisms
in that they have precisely one nontrivial fiber with at least one
nontrivial internal edge, but they need not be order-preserving. For
the purposes of the proofs in this 
article only, we will call such morphisms {\em
  pre-elementary\/}. By definition, a pre-elementary morphism is
elementary if and only if it is order-preserving. Canonical contractions provide
representatives of morphisms with the property specified in the
following lemma.

\begin{lemma}
\label{Svedeni neustava.}
Let $\tau: \Delta' \to \Delta''$ be a map in $\Grc$ whose all fibers
are corollas except precisely one (which thus has at least
one internal edge). Then there exists a unique
canonical contraction $\pi: \Delta' \to \Delta$ and a unique
isomorphism $\sigma : \Delta'' \to \Delta$ such that the diagram
\begin{equation}
\label{Jaruska si nezapla telefon.}
\xymatrix{&\Delta' \ar[ld]_\tau^(.4){\rm cc} \ar[rd]^\pi&
\\
\Delta'' \ar[rr]^\sigma_\cong && \Delta
}
\end{equation}
commutes.
\end{lemma}

\begin{proof}
Assume that $\Delta' = (V',F')$, $\Delta'' = (V'',F'')$, 
and that
$\tau$ is given by the pair $(\phi,\psi)$ of maps in the diagram
\[
\xymatrix@C=3.5em{F' \ar[d]_{g'}  & F''\ar@{_{(}->}[l]_{\psi}  
\ar[d]^{g''}
\\
V' \ar@{->>}[r]^{\phi} &\ V''.
}
\]
Let the only fiber of $\tau$ which is not a corolla be the one over some $x\in
V''$.  Then the canonical contraction $\pi$ in the lemma is given by
the data
$V := \phi^{-1}(\{x\}) \subset V'$ and $E := F' \setminus F''$. It is
simple to check that there exists a unique isomorphism $\sigma$,
symbolically expressed as $\tau^{-1}\pi$, making
diagram~\eqref{Jaruska si nezapla telefon.} commutative and that the
canonical projection $\pi$ for which such an isomorphism exists 
is unique as well. 
\end{proof}

Let us return to the main topics of this section.
A  {\em graph-labelled tree\/}, or {\em graph-tree\/} for short, is a 
rooted tree $T$ 
such that the union of the sets of  input leaves and of internal edges 
is labelled by a finite ordered
set $V$ subject to the condition that
an internal edge $e$ of $T$ is labelled  by the minimum of the labels
of the input leaves of the subtree of $T$ `below' $e$, 
i.e.\ of the maximal subtree of $T$ away from the root of $T$ 
whose root vertex is $e$.
Moreover, vertices of a~graph tree $T$ are labelled by graphs
in $\Grc$. This labelling shall satisfy two conditions. 

\noindent 
{\em Compatibility 1.}
The ordered set of vertices of $\Gamma_u$ labelling a vertex $u$ of
$T$ equals the ordered set of the labels of the input edges of $u$. 

\noindent 
{\em Compatibility 2.} Let $e$ be an internal edge of $T$ pointing
from (the vertex labelled by)  $\Gamma_u$ to  (the vertex labelled by)  
$\Gamma_v$. Then the ordered set of the
half-edges of $\Gamma_v$ adjacent to its vertex corresponding to $e$
is the same as the ordered set of the legs of $\Gamma_u$.

Since we are going to study free operads generated by $1$-connected
collections only, we assume that the graphs labelling the vertices
of a graph-tree have at least one internal edge.

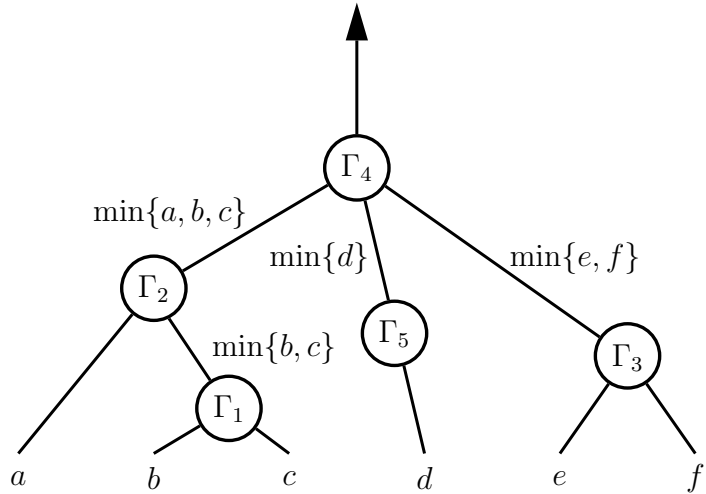
\begin{figure}[h]
\begin{center}
\psscalebox{1.0 1.0}
{
\begin{pspicture}(0,-3.130979)(9.17,3.130979)
\definecolor{colour0}{rgb}{0.92156863,0.9137255,0.9137255}
\psline[linecolor=black, linewidth=0.04, arrowsize=0.21667cm 2.0,arrowlength=1.7,arrowinset=0.0]{->}(4.58,1.0490209)(4.58,3.224021)
\psline[linecolor=black, linewidth=0.04, arrowsize=0.05291667cm 2.0,arrowlength=1.4,arrowinset=0.0]{->}(7.28,-2.775979)(8.18,-1.4884791)
\psline[linecolor=black, linewidth=0.04, arrowsize=0.05291667cm 2.0,arrowlength=1.4,arrowinset=0.0]{->}(9.08,-2.775979)(8.18,-1.4884791)
\psline[linecolor=black, linewidth=0.04, arrowsize=0.05291667cm 2.0,arrowlength=1.4,arrowinset=0.0]{->}(5.48,-2.775979)(4.58,1.0490209)
\psline[linecolor=black, linewidth=0.04, arrowsize=0.05291667cm 2.0,arrowlength=1.4,arrowinset=0.0]{->}(8.18,-1.4884791)(4.58,1.0490209)
\psdots[linecolor=black, fillstyle=solid, dotstyle=o, dotsize=0.4, fillcolor=white](8.18,-1.4759792)
\rput(0.08,-3.1){$a$}
\rput(1.88,-3.1){$b$}
\rput(3.68,-3.1){$c$}
\rput(5.48,-3.1){$d$}
\rput(7.28,-3.1){$e$}
\rput(9.08,-3.1){$f$}
\rput(3.48,-1.3759792){$\scriptsize {\rm min}\{b,c\}$}
\rput(2.1,0.42402086){$\scriptsize {\rm min}\{a,b,c\}$}
\rput(7.48,-0.17597915){${\rm min}\{e,f\}$}
\psline[linecolor=black, linewidth=0.04, arrowsize=0.05291667cm 2.0,arrowlength=1.4,arrowinset=0.0]{->}(0.08,-2.775979)(1.88,-0.5759792)
\psline[linecolor=black, linewidth=0.04, arrowsize=0.05291667cm 2.0,arrowlength=1.4,arrowinset=0.0]{->}(1.88,-0.5759792)(4.58,1.0240208)
\psdots[linecolor=black, fillstyle=solid, dotstyle=o, dotsize=0.4, fillcolor=white](4.58,1.0240208)
\psdots[linecolor=black, fillstyle=solid, dotstyle=o, dotsize=0.9, fillcolor=colour0](4.58,1.0240208)
\psdots[linecolor=black, fillstyle=solid, dotstyle=o, dotsize=0.9, fillcolor=colour0](8.18,-1.4759792)
\rput(4.58,1.0240208){$\Gamma_4$}
\rput(8.18,-1.5){$\Gamma_3$}
\psline[linecolor=black, linewidth=0.04](1.88,-2.775979)(2.88,-2.1759791)
\psline[linecolor=black, linewidth=0.04](3.68,-2.775979)(2.88,-2.1759791)
\psline[linecolor=black, linewidth=0.04](2.88,-2.1759791)(1.88,-0.67597914)
\psdots[linecolor=black, fillstyle=solid, dotstyle=o, dotsize=0.9, fillcolor=colour0](2.88,-2.1759791)
\psdots[linecolor=black, fillstyle=solid, dotstyle=o, dotsize=0.9, fillcolor=colour0](1.88,-0.5759792)
\rput(1.88,-0.5759792){$\Gamma_2$}
\rput(2.88,-2.1759791){$\Gamma_1$}
\psdots[linecolor=black, fillstyle=solid, dotstyle=o, dotsize=0.9, fillcolor=colour0](5.08,-1.1759791)
\rput(5.08,-1.1759791){$\Gamma_5$}
\rput(4.08,-0.17597915){${\rm min}\{d\}$}
\end{pspicture}
}
\end{center}
\caption{A graph-tree.\label{V nedeli jsem tahal na Kelimkovi.}}
\end{figure}

\begin{example}
A portrait of a graph-tree is
given in Figure~\ref{V nedeli jsem tahal na Kelimkovi.}. The set $V$ equals
in this case to $\{a,b,c,d,e,f\} $ with 
some (linear) order. The graph $\Gamma_4$ has
three vertices labelled by the elements of the subset
\[
\big\{\min\{a,b,c\}, \min\{d\}=d,\min\{e,f\}\big\} \subset \{a,b,c,d,e,f\} 
\]
with the induced linear order. The graph $\Gamma_5$ has only one
vertex labelled by $d$.
\end{example}

Let $e$ be an internal edge of a graph-tree $T$ pointing from
$\Gamma_u$ to  $\Gamma_v$. Then the tree $T/e$ obtained by contracting
the edge $e$ has an induced structure of a graph-tree given as
follows. The leaves and internal edges of $T/e$ bear the same labels
as they did in $T$. Also the vertices of $T/e$ except of the one, say
$x$, created
by the collapse of $e$, are labelled by the same graphs as in
$T$. Finally, the vertex $x$ is labelled by the graph $\Gamma_x$ given by the
vertex insertion of $\Gamma_u$ into the vertex of $\Gamma_v$ labelled
by $e$. Since, by Compatibility~2, the ordered set of 
legs of $\Gamma_u$ is the same as the
ordered set of the half-edges
adjacent to the vertex of $\Gamma_v$ labelled by $e$,
the vertex insertion is uniquely and well-defined. One clearly has
\[
\Vert(\Gamma_x) = (\Vert(\Gamma_v) \setminus \{\hbox{the vertex
  labelled by } e\}) \cup \Vert(\Gamma_u),
\]
where the union in the right hand side is disjoint thanks to
Compatibility~1. The set $\Vert(\Gamma_x)$ bears an order induced from
the inclusion $\Vert(\Gamma_x) \subset V$.
 
Repeating the collapsings described above we finally obtain a graph-tree
with one vertex (i.e.\ a rooted corolla) whose only vertex is labelled
by some graph $\GAmma \in \Grc$ with the ordered set of vertices~$V$. We denote the
graph $\GAmma$ thus obtained, which clearly does not depend on the order in
which we contracted the edges of $T$, by $\gr(T)$. 

\begin{theorem}
\label{Uz_nechci_letat_zakladni_vycvik.}
The set of connected components of the groupoid $\lTw(\GAmma)$ is
canonically isomorphic to the set $\gTr(\GAmma)$ of graph-trees  $T$ with
$\gr(T) = \GAmma$. 
\end{theorem}

\begin{figure}[h]
\begin{center}
\psscalebox{1.0 1.0} 
{
\begin{pspicture}(0,-4.0582557)(13.51,4.0582557)
\psline[linecolor=black, linewidth=0.04, arrowsize=0.2cm 2.0,arrowlength=1.7,arrowinset=0.0]{->}(4.4,1.9762975)(4.4,4.1512976)
\psline[linecolor=black, linewidth=0.04, arrowsize=0.05291667cm 2.0,arrowlength=1.4,arrowinset=0.0]{->}(7.3,-4.0487027)(8.0,-2.7612026)
\psline[linecolor=black, linewidth=0.04, arrowsize=0.05291667cm 2.0,arrowlength=1.4,arrowinset=0.0]{->}(8.7,-4.0487027)(8.0,-2.7612026)
\psline[linecolor=black, linewidth=0.04, arrowsize=0.05291667cm 2.0,arrowlength=1.4,arrowinset=0.0]{->}(5.9,-4.0487027)(4.4,1.9762975)
\psline[linecolor=black, linewidth=0.04, arrowsize=0.05291667cm 2.0,arrowlength=1.4,arrowinset=0.0]{->}(8.1,-2.8487024)(4.4,1.9762975)
\psline[linecolor=black, linewidth=0.04, arrowsize=0.05291667cm 2.0,arrowlength=1.4,arrowinset=0.0]{->}(1.7,-4.0487027)(2.7,0.7512975)
\psline[linecolor=black, linewidth=0.04, arrowsize=0.05291667cm 2.0,arrowlength=1.4,arrowinset=0.0]{->}(2.7,0.7512975)(4.4,1.9512974)
\psline[linecolor=black, linewidth=0.04](3.1,-4.0487027)(3.7,-1.6487025)
\psline[linecolor=black, linewidth=0.04](4.5,-4.0487027)(3.7,-1.6487025)
\psline[linecolor=black, linewidth=0.04](3.7,-1.6487025)(2.7,0.7512975)
\rput(3.7,-1.8487025){$\Gamma_1$}
\psline[linecolor=black, linewidth=0.02](9.1,-2.8487024)(1.3,-2.8487024)
\psline[linecolor=black, linewidth=0.02](1.3,-1.6487025)(9.1,-1.6487025)
\psline[linecolor=black, linewidth=0.02](1.3,-0.44870254)(9.1,-0.44870254)
\psline[linecolor=black, linewidth=0.02](1.3,0.7512975)(9.1,0.7512975)
\psline[linecolor=black, linewidth=0.02](1.3,1.9512974)(9.1,1.9512974)
\psdots[linecolor=black, fillstyle=solid, dotstyle=o, dotsize=0.9, fillcolor=white](8.1,-2.8487024)
\psdots[linecolor=black, fillstyle=solid, dotstyle=o, dotsize=0.9, fillcolor=white](3.6,-1.7487025)
\psdots[linecolor=black, fillstyle=solid, dotstyle=o, dotsize=0.9, fillcolor=white](4.5,1.9512974)
\psdots[linecolor=black, fillstyle=solid, dotstyle=o, dotsize=0.9, fillcolor=white](2.7,0.7512975)
\psdots[linecolor=black, fillstyle=solid, dotstyle=o, dotsize=0.9, fillcolor=white](5.0,-0.34870255)
\rput[l](4.3,1.95){$\Gamma_4$}
\rput[l](2.5,0.7512975){$\Gamma_2$}
\rput[l](4.8,-0.355){$\Gamma_5$}
\rput[l](3.4,-1.7487025){$\Gamma_1$}
\rput[l](7.9,-2.8487024){$\Gamma_3$}
\psline[linecolor=black, linewidth=0.02, linestyle=dashed, dash=0.17638889cm 0.10583334cm, arrowsize=0.05291667cm 2.0,arrowlength=1.4,arrowinset=0.0]{->}(8.5,1.3512975)(10.3,1.3512975)(10.3,2.9512975)
\psline[linecolor=black, linewidth=0.02, linestyle=dashed, dash=0.17638889cm 0.10583334cm, arrowsize=0.05291667cm 2.0,arrowlength=1.4,arrowinset=0.0]{->}(9.1,0.15129745)(11.1,0.15129745)(11.1,1.7512975)
\psline[linecolor=black, linewidth=0.02, linestyle=dashed, dash=0.17638889cm 0.10583334cm, arrowsize=0.05291667cm 2.0,arrowlength=1.4,arrowinset=0.0]{->}(9.1,-1.0487026)(11.9,-1.0487026)(11.9,0.35129747)
\psline[linecolor=black, linewidth=0.02, linestyle=dashed, dash=0.17638889cm 0.10583334cm, arrowsize=0.05291667cm 2.0,arrowlength=1.4,arrowinset=0.0]{->}(9.1,-2.2487025)(12.7,-2.2487025)(12.7,-0.84870255)
\psline[linecolor=black, linewidth=0.02, linestyle=dashed, dash=0.17638889cm 0.10583334cm, arrowsize=0.05291667cm 2.0,arrowlength=1.4,arrowinset=0.0]{->}(9.1,-3.4487026)(13.5,-3.4487026)(13.5,-1.8487025)
\rput[b](0.5,-2.8487024){{\rm level 1}}
\rput(0.5,-1.5487026){{\rm level 2}}
\rput[b](0.5,-0.44870254){{\rm level 3}}
\rput[b](0.5,0.7512975){{\rm level 4}}
\rput[b](0.5,1.9512974){{\rm level 5}}
\rput[b](13.1,-3.2487025){$\Delta_0$}
\rput[b](12.3,-2.0487025){$\Delta_1$}
\rput[b](11.5,-0.84870255){$\Delta_2$}
\rput[b](10.7,0.35129747){$\Delta_3$}
\rput[b](9.9,1.5512974){$\Delta_4$}
\psline[linecolor=black, linewidth=0.02, arrowsize=0.05291667cm 2.0,arrowlength=1.4,arrowinset=0.0]{->}(9.5,0.35129747)(9.5,1.1512975)
\psline[linecolor=black, linewidth=0.02, arrowsize=0.05291667cm 2.0,arrowlength=1.4,arrowinset=0.0]{->}(9.5,-0.84870255)(9.5,-0.048702545)
\psline[linecolor=black, linewidth=0.02, arrowsize=0.05291667cm 2.0,arrowlength=1.4,arrowinset=0.0]{->}(9.5,-2.0487025)(9.5,-1.2487025)
\psline[linecolor=black, linewidth=0.02, arrowsize=0.05291667cm 2.0,arrowlength=1.4,arrowinset=0.0]{->}(9.5,-3.2487025)(9.5,-2.4487026)
\rput(9.9,-2.8487024){$\tau_1$}
\rput(9.9,-1.6487025){$\tau_2$}
\rput(9.9,-0.44870254){$\tau_3$}
\rput(9.9,0.7512975){$\tau_4$}
\end{pspicture}
}
\end{center}
\caption{Introducing levels to the tree in Figure~\ref{V nedeli jsem
    tahal na Kelimkovi.}. The labels of its leaves and edges are the same
  as in Figure \ref{V nedeli jsem tahal na Kelimkovi.} \label{Zase_pojedu_proti_vetru!}}
\end{figure}
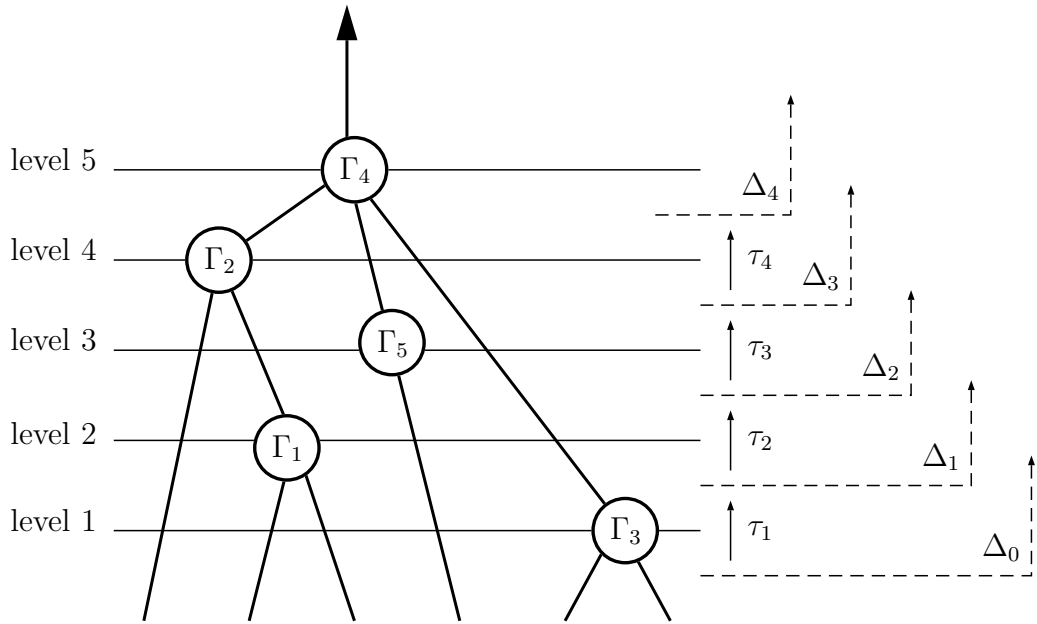

\def\ttau{{\tilde \tau}}
\def\tDelta{{\widetilde \Delta}}

\begin{proof}
Recall from \cite[Section~3]{SydneyII} that the objects of $\lTw(\GAmma)$
are labelled towers
\begin{equation}
\label{t1}
\bfT = (\bfT,\ell): \
\GAmma \stackrel \ell\longrightarrow 
\Delta_0 \stackrel {\tau_1} \longrightarrow \Delta_1  \stackrel  {\tau_2} \longrightarrow
\Delta_2\stackrel  {\tau_3} \longrightarrow 
\cdots  \stackrel  {\tau_{k-1}}   \longrightarrow  \Delta_{k-1},
\end{equation}
where $\rada {\Delta_0}{\Delta_{k-1}}$ are graphs in $\Grc$, $\ell$ an
isomorphism, and $\Rada \tau1{k-1}$ elementary maps. We will
construct a map
\[
A : \gTr(\GAmma) \longrightarrow \pi_0\big( \lTw(\GAmma)\big)
\] 
of sets as follows. Assume that $T \in \gTr(\GAmma)$ is a graph-tree with $k$
vertices. We distribute the vertices of $T$ to levels 
such that each level contains
precisely one vertex, see Figure~\ref{Zase_pojedu_proti_vetru!} for an
example. Let $T_{i-1}$, $1 \leq i \leq k$, be the graph-tree obtained
from $T$ by truncating everything above level $i$, level $i$ included,
see Figure~\ref{Zase_pojedu_proti_vetru!} again. Denote $\tDelta_{i-1} :=
\gr(T_{i-1})$, $1 \leq i \leq k$. Notice that $\tDelta_0 = \GAmma$ by definition. 
One thus obtains a sequence of pre-elementary maps
\begin{equation}
\label{Jaruska_ma_jizdy.}
\GAmma \stackrel \id\longrightarrow 
\tDelta_0 \stackrel {\ttau_1} \longrightarrow \tDelta_1  \stackrel  {\ttau_2} \longrightarrow
\tDelta_2\stackrel  {\ttau_3} \longrightarrow 
\cdots  \stackrel  {\ttau_{k-1}}   \longrightarrow  \tDelta_{k-1}
\end{equation}
in which the map $\ttau_i : \tDelta_{i-1} \to \tDelta_i$, $1 \leq i \leq
k-1$, is defined as follows. Let $u$ be the only
vertex on the $i$th level and $e$ its out-going edge. Then $\ttau_i$ is
the map that contracts the subgraph $\Gamma_u$ of $\tDelta_{i-1}$ into
the vertex of $e$ labelled by $e$. In other words, $\ttau_i$ is the
canonical contraction $\tDelta_{i-1} \to \tDelta_i =
\tDelta_{i-1}/\Gamma_u$ and thus a pre-elementary map.

For instance,
in the situation of Figure~\ref{Zase_pojedu_proti_vetru!}, the graph
$\tDelta_0$ has vertices $\{a,b,c,d,e,f\}$ and
\[
\Vert(\tDelta_1) = \big\{a,b,c,d,\min\{e,f\}\big\}.
\]
The map $\ttau_1$ contacts the subgraph $\Gamma_3$ of $\tDelta_0$ into
the vertex $\min\{e,f\}$ of $\tDelta_1$. Likewise,
\[
\Vert(\tDelta_2) = \big\{a,\min\{b,c\},d,\min\{e,f\}\big\}
\]
and $\ttau_2$ contacts the subgraph $\Gamma_1$ into the vertex
$\min\{b,c\}$ of $\tDelta_2$. Notice that $\ttau_1$, resp.~$\ttau_2$
is order-preserving if and only if $\{e,f\}$ resp.~$\{b,c\}$ is an
interval. This shows that $\ttau_i$'s in~(\ref{Jaruska_ma_jizdy.})
need not in general be elementary, i.e.\ preserving the orders of the
set of vertices.
 
Out next task will be to modify the tower~(\ref{Jaruska_ma_jizdy.})
into a tower as in~(\ref{t1}) with all $\tau_i$'s elementary. 
To do so we use the fact that the
category $\Grc$ is factorizable~\cite[Lemma 3.16]{SydneyI}, meaning that each morphism
can be written as $\phi \sigma$, where $\phi$ is order-preserving and
$\sigma$ a quasibijection. It also follows from~\cite[Lemma~2.1]{SydneyI} resp~\cite[Lemma 2.2]{SydneyI} that, if $\sigma$ is a
quasibijection and $\psi$ a pre-elementary map, then both $\sigma\psi$
and $\psi\sigma$ are pre-elementary as well. Recall also that in
$\Grc$ all quasibijections are invertible and their inverses are
quasibijections again.

The process of modification is described in Figure~\ref{Psano v
  Mercine.}. We start at the bottom, by putting $\Delta_{k-1} :=
\tDelta_{k-1}$ and decomposing $\ttau_{k-1}$ into a quasibijection
$\sigma_{k-2}$ followed by an order-preserving $\tau_{k-1}$. By the
above remarks, $\tau_{k-1} = \ttau_{k-1} \sigma_{k-1}^{-1}$ is
pre-elementary and, since it is order-preserving, it is elementary. Now decompose $\sigma_{k-2}
\ttau_{k-2}$ as a quasibijection $\sigma_{k-3}$ followed by an order
preserving $\tau_{k-2}$. By the same reasoning, $\tau_{k-2} =
\sigma_{k-2} \ttau_{k-2} {\sigma'}^{-1}_{\hskip -.3em k-3}$ is
elementary. We go all the way up, ending with $\ell := \sigma_0$.
\begin{figure}[h]
\centering
\[
\xymatrix{
\Gamma\ar[d]^{\id}\ar[rd]^\ell
\\
\tDelta_0 \ar[d]^{\ttau_1}\ar[r]^{\sigma_0}_\sim &\Delta_0 \ar[d]^{\tau_1}
\\
\tDelta_1\ar[d]^{\ttau_2}\ar[r]^{\sigma_1}_\sim & \Delta_1\ar[d]^{\tau_2}
\\
\vdots\ar[d]^{\ttau_{k-3}} & \vdots\ar[d]^{\tau_{k-3}}
\\
\tDelta_{k-3}\ar[d]^{\ttau_{k-2}}\ar[r]^{\sigma_{k-3}}_\sim & \Delta_{k-3}\ar[d]^{\tau_{k-2}}
\\
\tDelta_{k-2}\ar[d]^{\ttau_{k-1}}\ar[r]^{\sigma_{k-2}}_\sim & \Delta_{k-2}\ar[d]^{\tau_{k-1}}
\\
\tDelta_{k-1}\ar@{=}[r]^{\id} & \Delta_{k-1}
}
\hskip 3em
\xymatrix@R=1.95em{
&\Gamma\ar[d]^{\id}\ar[rd]^{\ell''} \ar[ld]_{\ell'}&
\\
\Delta'_0 \ar[d]^{\tau'_1}&\tDelta_0 
\ar[d]^{\ttau_1}
\ar[r]^{\sigma''_0}_\sim \ar[l]_{\sigma'_0}^\sim &\Delta''_0 \ar[d]^{\tau''_1}
\\
\Delta'_1\ar[d]^{\tau'_2}&\tDelta_1\ar[d]^{\ttau_2}\ar[r]^{\sigma''_1}_\sim
\ar[l]_{\sigma'_1}^\sim
& 
\Delta''_1\ar[d]^{\tau''_2}
\\
\vdots\ar[d]^{\tau'_{k-3}}&\vdots\ar[d]^{\ttau_{k-3}} & \vdots\ar[d]^{\tau''_{k-3}}
\\
 \Delta'_{k-3}\ar[d]^{\tau'_{k-2}}&\tDelta_{k-3}\ar[d]^{\ttau_{k-2}}\ar[r]^{\sigma''_{k-3}}_\sim 
\ar[l]_{\sigma'_{k-3}}^\sim & \Delta''_{k-3}\ar[d]^{\tau''_{k-2}}
\\
\Delta'_{k-2}\ar[d]^{\tau'_{k-1}}&\tDelta_{k-2}\ar[d]^{\ttau_{k-1}}\ar[r]^{\sigma''_{k-2}}_\sim
\ar[l]_{\sigma'_{k-2}}^\sim & \Delta''_{k-2}\ar[d]^{\tau''_{k-1}}
\\
\Delta'_{k-1}&\tDelta_{k-1}\ar@{=}[r]^{\id} \ar@{=}[l]_{\id} & \Delta''_{k-1}
}
\hskip 3em
\xymatrix{
\Gamma\ar[d]^{\id}
\\
\tDelta_0 \ar[d]^{\widehat\tau_1} \ar[rd]^{\ttau_1}&
\\
\Delta_1\ar[d]^{\tau_2}\ar[r]^{\sigma_1}_\sim & \tDelta_1\ar[d]^{\ttau_2}
\\
\Delta_{2}\ar[d]^{\tau_{3}}\ar[r]^{\sigma_{2}}_\sim & \tDelta_{2}\ar[d]^{\ttau_{3}}
\\
\vdots\ar[d]^{\tau_{k-3}} & \vdots\ar[d]^{\ttau_{k-3}}
\\
\Delta_{k-2}\ar[d]^{\tau_{k-1}}\ar[r]^{\sigma_{k-2}}_\sim & \tDelta_{k-2}\ar[d]^{\ttau_{k-1}}
\\
\Delta_{k-1}\ar[r]^{\sigma_{k-1}} & \tDelta_{k-1}
}
\]  
\caption{\label{Psano v Mercine.}%
Replacing~(\ref{Jaruska_ma_jizdy.}) by a labelled
  tower of elementary maps (left); independence of the choices
  of factorizations (center); replacing~\eqref{t1} by a tower of canonical
  contractions (right).}
\end{figure}
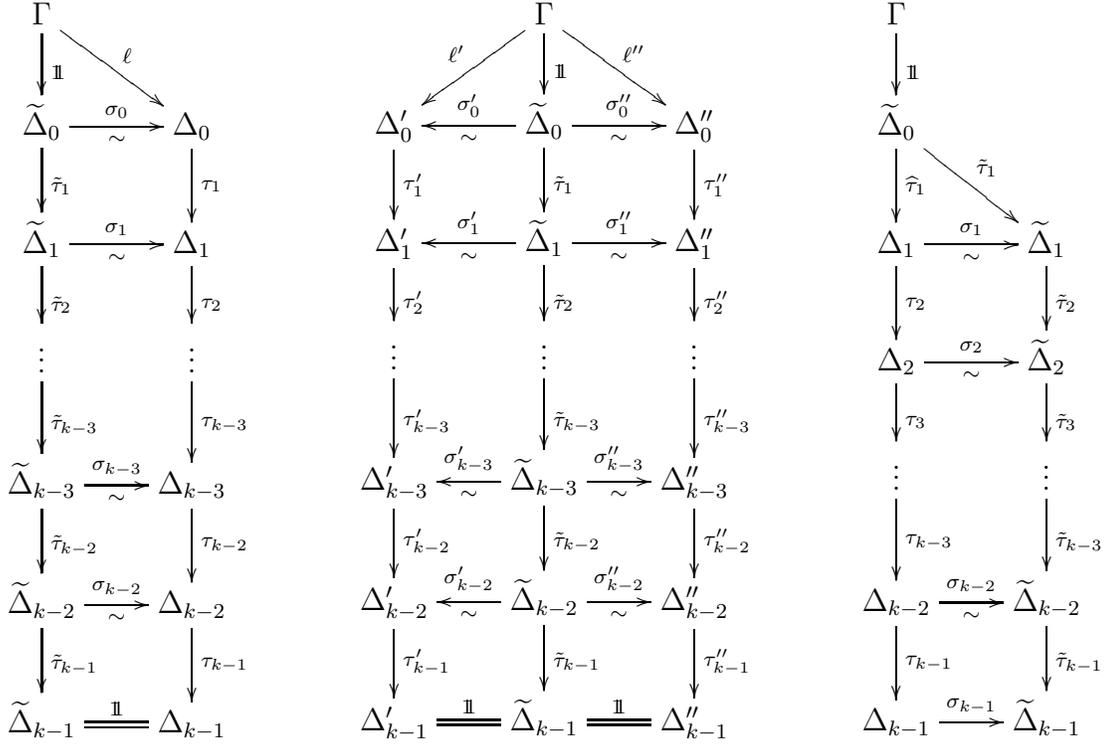

The actual value of $a(T)$ might depend on the choice of
factorizations but, as the diagram at the center  of Figure \ref{Psano v
  Mercine.} shows, the results are related by an isomorphism of
the 1st type in the sense of \cite[Section~3]{SydneyII} whose
definition we recall in Figure~\ref{f1}. Indeed, take $\sigma_{k-1}
:= \id$ and $\sigma_i := \sigma_i'' {\sigma'}_{\hskip -.2em i}^{-1}$
for $0 \leq i \leq 
k-2$ in that figure. 
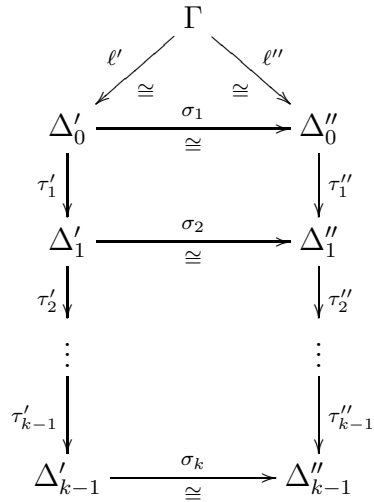
\begin{figure}[h]
\[
\xymatrix{
&\Gamma  \ar[ld]_{\ell'}^\cong \ar[rd]^{\ell''}_\cong  &
\\
\Delta_0' \ar[d]_{\tau'_1} \ar[rr]^{\sigma_1}_\cong & & \Delta_0'' \ar[d]^{\tau''_1}
\\
\Delta_1'  \ar[d]_{\tau'_2} \ar[rr]^{\sigma_2}_\cong & & \Delta_1'' \ar[d]^{\tau''_2}
\\
\vdots \ar[d]_{\tau'_{k-1}}  && \vdots \ar[d]^{\tau''_{k-1}}
\\
\Delta_{k-1}' \ar[rr]^{\sigma_{k}}_\cong && \Delta_{k-1}''
}
\]
\caption{A commutative diagram defining a morphism of labelled towers of the
  first type.
The maps $\Rada \sigma1k$ are isomorphisms.
\label{f1}}
\end{figure}
The  value of $a(T)$ might also depend on the choice of levels of
$T$, but any two such values are related by an isomorphism of
the 2nd type in the sense of \cite[Section~3]{SydneyII}. 
The connected component of $a(T)$ therefore does not depend on the
choices, 
so we may define $A(T) := \pi_0( a(T)) \in \pi_0\big( \lTw(\GAmma)\big)$.

Let us proceed to the construction of the inverse  
$
B : \pi_0\big( \lTw(\GAmma)\big)  \longrightarrow \gTr(\GAmma)
$ 
of $A$.
Suppose that we are given a labelled tower $\bfT$ as in~(\ref{t1}). Our strategy
will be to modify it
into the form where $\ell$ is the identity and 
the remaining maps are canonical contractions. 
We start by absorbing $\ell$ into 
$\tau_1$ in~(\ref{t1}) by replacing it with
\begin{equation}
\label{Projedu_se_jeste_na_kole?}
\GAmma \stackrel \id\longrightarrow 
\tDelta_0 \stackrel {\widehat\tau_1} \longrightarrow \Delta_1  \stackrel  {\tau_2} \longrightarrow
\Delta_2\stackrel  {\tau_3} \longrightarrow 
\cdots  \stackrel  {\tau_{k-1}}   \longrightarrow  \Delta_{k-1},
\end{equation}
where $\tDelta_0 := \Gamma$ and $\widehat\tau_1 := \tau_1 \circ \ell$. Notice that all morphisms
in~(\ref{Projedu_se_jeste_na_kole?}) satisfy the hypothesis of
Lemma~\ref{Svedeni neustava.}. The next steps are illustrated by the
diagram on the right of Figure~\ref{Jaruska_ma_jizdy.}. In that
diagram, $\ttau_1$ is the canonical projection obtained by taking, in
Lemma~\ref{Svedeni neustava.}, $\Delta' = \tDelta_0$, $\Delta'' =
\Delta_1$ and $\tau = \hat\tau_1$. Now, the composition
$\tau_2 \sigma_1^{-1}$ satisfies the assumptions of
Lemma~\ref{Svedeni neustava.}, and $\ttau_2$ is the canonical
projection obtained from that lemma by taking $\Delta' = \tDelta_1$,
$\Delta'' = \Delta_2$ and 
$\tau = \tau_2 \sigma_1^{-1}$. We then continue all the way down till we
eventually construct the canonical projection $\ttau_{k-1}$.

We thus modified the labelled tower $\bfT$ in~(\ref{t1}) to the 
tower~\eqref{Jaruska_ma_jizdy.} in which all
$\ttau_i$'s are canonical contractions. We will say that such a  tower 
has the {\em canonical form\/}.
Denote by $V_i$ the set of vertices of
$\tDelta_i$ in~\eqref{Jaruska_ma_jizdy.}, 
$0 \leq i \leq k\!-\!1$. It follows from the definition of
canonical contractions that $V_0 \supset V_1 \supset \cdots \supset V_{k-1}$. 
Moreover, each $V_i$ contains a distinguished element $x_i$ 
over which the unique nontrivial fiber of $\tau_i$ lives. We
extend the notation by putting $V_k := \{*\}$, the one-point set,
and $x_k := *$.  
The vertex parts of $\tau_i$'s give rise to the sequence
\begin{equation}
\label{Je_cas_chystat_se_do_Australie.}
\xymatrix@1{
V_0 \ \ar@{->>}[r]^{\phi_1} &
\ V_1 \ \ar@{->>}[r]^{\phi_2} 
&\ V_2\ \ar@{->>}[r]^{\phi_3} &\ \cdots \ \ar@{->>}[r]^{\phi_{k-1}} 
&\ V_{k-1}\ \ar@{->>}[r]^{\phi_k}
 &\ V_k = \{*\} 
}
\end{equation} 
of epimorphisms
with the property that $\min(\phi_1^{-1} \circ \cdots \circ
\phi_i^{-1}(x_i)) = x_i$, $1 \leq i \leq k\-1$. Such a~sequence of epimorphisms
of finite ordered sets determines in the standard manner a 
rooted tree with levels, with its leaves
labelled by $V_0$, with the root $*$ and the remaining vertices
$\Rada x1{k-1}$. Forgetting the levels, decorating the root by
$\tDelta_0 = \Gamma$ and $x_i$ by the fiber $\Gamma_i$ of $\tau_i$, $1\leq i \leq
k\-1$, leads to a graph-tree $B(\bfT) \in \gTr(\GAmma)$.   

The reason why $B(\bfT)$ is well-defined, that is,
$B(\bfT') = B(\bfT'')$ if $\bfT'$ and $\bfT''$ are isomorphic
labelled towers, is that
for isomorphisms of the second type, 
see \cite[Section~3]{SydneyII} for terminology, the difference
disappears after forgetting the levels of the tree corresponding
to~(\ref{Je_cas_chystat_se_do_Australie.}), while it is not difficult
to see that the canonical forms of labelled towers related by an
isomorphisms of the first type are {\em the same\/}.

It is clear that $\hbox{$(B \!\circ\! A)$}(T) = T$ for $T\in
\gTr(\GAmma)$. Given a labelled tower $\bfT \in \lTw(\GAmma)$, the concrete form
of the tower $a(
B(\bfT)) \in \lTw(\GAmma)$ representing $\hbox{$(A\! \circ\!
B)$}(\bfT) \in \pi_0(\lTw(\GAmma))$ depends on the choice of levels for the tree
$B(\bfT)$. But any two such towers are related by a type two
isomorphism.   
Since modifying a tower into its canonical form does not change its
isomorphism class, we established that $B$ is also a left inverse of $A$.
\end{proof}

The set $\gTr(\GAmma)$ and therefore also the set
$\pi_0(\lTw(\GAmma))$ of connected components of the category
$\lTw(\GAmma)$ has a natural poset structure induced by the relation
$T \prec T/e$ for a graph-tree $T\in \gTr(\GAmma)$ and its edge
$e$. Its categorical origin is the following.

Let us denote, only for the purpose of this explanation, by $\ttC$ the 
category whose objects are the same as the objects of  $\lTw(\GAmma)$,
i.e.\ the labelled towers $\bfT$ as
in~(\ref{t1}). We postulate that there is a unique morphism $\bfT \to
\bfS$,  $\bfT \not=
\bfS$, in $\ttC$
if and only if $\bfS$ is obtained from  $\bfT$ by composing two or more
adjacent morphisms $\tau_i$'s that
have mutually joint fibers, in the sense
of~\cite[Definition~5.4]{SydneyI}. The only other morphisms in $\ttC$
are the identities.

We denote by $\lTw(\GAmma) \int \ttC$ the category with the objects of
$\lTw(\GAmma)$ whose 
morphisms are formal compositions of a morphism of $\lTw(\GAmma)$ with a
morphism of $\ttC$. The poset \hbox{$(\pi_0(\lTw(\GAmma)),\prec)$}
considered in the standard manner as a category is then
canonically isomorphic to the pushout in {\tt Cat} of the diagram
\[
\xymatrix{\lTw(\GAmma) \ar[d]\ar@{^(->}[r] & \lTw(\GAmma) \int \ttC
\\
\pi_0(\lTw(\GAmma))
}
\]
in which $\pi_0(\lTw(\GAmma))$ is taken as a discrete category.

.

We are finally going to give an explicit formula for 
the free $\Grc$-operad $\Free(E)$ generated by a
$1$-connected collection $E$ evaluated at a graph $\GAmma$. Recall that $E$
is a representation, in the category of graded vector spaces, 
of the groupoid $\QV(e)$ whose objects are graphs
in $\Grc$ and morphisms are virtual isomorphisms which are, in this
specific case, isomorphisms of graphs which need not respect the
orders of the legs. The $1$-connectivity means that $E(\GAmma) \not= 0$
implies that $\GAmma \in \Grc$ has at least one internal edge. 

\begin{Warning}
\label{Dnes jdu s Jaruskou na koncert.}
Let us consider the classical free non-$\Sigma$ operad $\Free(E) =
\{\Free(E)(n)\}_{n \geq 1}$ generated by
a collection $E$ of graded vector spaces. A common mistake is to 
assume that the elements of $\Free(E)$ are (represented by) trees with
vertices decorated by elements of $E$. This is true only when $E$ is
concentrated in even degrees. Otherwise we need one more piece of
information, namely a choice of levels of the underlying tree.

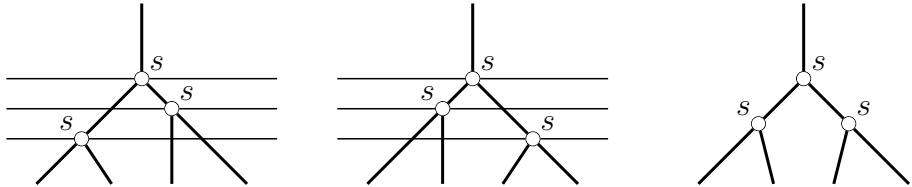
\begin{figure}[h]
\begin{center}
\psscalebox{1.0 1.0} 
{
\begin{pspicture}(0,-1.2070711)(12.014142,1.2070711)
\psline[linecolor=black, linewidth=0.04](10.6,1.2070711)(10.6,0.20707108)(9.2,-1.1929289)
\psline[linecolor=black, linewidth=0.04](10.6,0.20707108)(12.0,-1.1929289)
\psline[linecolor=black, linewidth=0.04](6.2,1.2070711)(6.2,0.20707108)(4.8,-1.1929289)
\psline[linecolor=black, linewidth=0.04](6.2,0.20707108)(7.6,-1.1929289)
\psline[linecolor=black, linewidth=0.04](1.8,1.2070711)(1.8,0.20707108)(0.4,-1.1929289)
\psline[linecolor=black, linewidth=0.04](1.8,0.20707108)(3.2,-1.1929289)
\psline[linecolor=black, linewidth=0.04](5.8,-0.19292893)(5.8,-1.1929289)
\psline[linecolor=black, linewidth=0.04](7.0,-0.59292895)(6.6,-1.1929289)
\psline[linecolor=black, linewidth=0.04](2.2,-0.19292893)(2.2,-1.1929289)
\psline[linecolor=black, linewidth=0.04](1.0,-0.59292895)(1.4,-1.1929289)
\psline[linecolor=black, linewidth=0.04](10.0,-0.39292893)(10.2,-1.1929289)
\psline[linecolor=black, linewidth=0.04](11.2,-0.39292893)(11.0,-1.1929289)
\psline[linecolor=black, linewidth=0.02](0.0,-0.19292893)(3.6,-0.19292893)
\psline[linecolor=black, linewidth=0.02](0.0,-0.59292895)(3.6,-0.59292895)
\psline[linecolor=black, linewidth=0.02](0.0,0.20707108)(3.6,0.20707108)
\psline[linecolor=black, linewidth=0.02](4.4,0.20707108)(8.0,0.20707108)
\psline[linecolor=black, linewidth=0.02](4.4,-0.19292893)(8.0,-0.19292893)
\psline[linecolor=black, linewidth=0.02](4.4,-0.59292895)(8.0,-0.59292895)
\rput(2.0,0.40707108){$s$}
\rput(6.4,0.40707108){$s$}
\rput(10.8,0.40707108){$s$}
\rput(0.8,-0.39292893){$s$}
\rput(2.4,0.0070710755){$s$}
\rput(5.6,0.0070710755){$s$}
\rput(7.2,-0.39292893){$s$}
\rput(9.8,-0.19292893){$s$}
\rput(11.4,-0.19292893){$s$}
\psdots[linecolor=black, dotstyle=o, dotsize=0.2, fillcolor=white](1.8,0.20707108)
\psdots[linecolor=black, dotstyle=o, dotsize=0.2, fillcolor=white](1.0,-0.59292895)
\psdots[linecolor=black, dotstyle=o, dotsize=0.2, fillcolor=white](2.2,-0.19292893)
\psdots[linecolor=black, dotstyle=o, dotsize=0.2, fillcolor=white](5.8,-0.19292893)
\psdots[linecolor=black, dotstyle=o, dotsize=0.2, fillcolor=white](6.2,0.20707108)
\psdots[linecolor=black, dotstyle=o, dotsize=0.2, fillcolor=white](7.0,-0.59292895)
\psdots[linecolor=black, dotstyle=o, dotsize=0.2, fillcolor=white](10.6,0.20707108)
\psdots[linecolor=black, dotstyle=o, dotsize=0.2, fillcolor=white](10.0,-0.39292893)
\psdots[linecolor=black, dotstyle=o, dotsize=0.2, fillcolor=white](11.2,-0.39292893)
\end{pspicture}
}
\end{center}
\caption{Levels versus no levels.\label{Dnes_prijede_Andulka.}}
\end{figure}

Assume for instance that $s \in E(2)$ is a degree $1$ generator. 
The leftmost tree in Figure~\ref{Dnes_prijede_Andulka.} represents $(s \circ_2 s)
\circ_1 s \in \Free(E)(4)$ while the middle one $(s \circ_1 s)
\circ_3 s$ in the same piece of $\Free(E)$. 
By the parallel associativity of the
$\circ_i$-operations 
\[
(s \circ_2 s)
\circ_1 s = - (s \circ_1 s)
\circ_3 s,
\]
thus the two decorated trees represent {\em different\/} elements. 
If we do not specify the levels  in
the rightmost tree in Figure~\ref{Dnes_prijede_Andulka.}, we do not
know to which one we refer to. 
The same caution is necessary also in case of
free $\Grc$-operads. 
\end{Warning}

Let us return to our description of the free operad $\Free(E)$. 
For a graph-tree $T$ we denote by $\Lev(T)$ the chaotic groupoid whose
objects are all possible arrangements of
levels of $T$. For a given $\lambda \in \Lev(T)$, let
$\Gamma_i$, $1 \leq i \leq k\-1$, be the fiber of $\tau_i$ in the
tower~(\ref{Jaruska_ma_jizdy.}) associated to $T$ with levels
$\lambda$. Here `chaotic' means that the category $\Lev(T)$ has a unique
morphism  $\lambda' \to \lambda''$ for any $\lambda',\lambda'' \in
\Lev(T)$; this morphism is necessarily  an isomorphism.
We extend the notation by
$\Gamma_k := \Delta_k$. For a 1-connected collection we define
\begin{equation}
\label{Krtek_na_mne_kouka.}
E(T,\lambda) := E(\Gamma_1) \ot \cdots \ot E(\Gamma_k).
\end{equation}
For different $\lambda$'s this expression differs only by the order of
the factors, so we may, using the commutativity constraint for graded
vector spaces, promote formula~(\ref{Krtek_na_mne_kouka.}) into a functor
\begin{equation}
\label{Zviratka_mi_pomahaji.}
E: \Lev(T) \longrightarrow \Vect
\end{equation} 
into the category of graded vector spaces.

\begin{theorem}
\label{podlehl_jsem}
Given a $1$-connected collection $E$, one has the following
description of the arity~$\GAmma$ piece of the free
operad $\Free(E)$:
\begin{equation}
\label{Zbijecky duni.}
\Free(E)(\GAmma) \cong 
\begin{cases}
\displaystyle \bigoplus_{T \in \gTr(\GAmma)} \  \colim{\lambda \in
  \Lev(T)} E(T,\lambda)&\hbox {if $\GAmma$ has at least one internal edge,
  and}
\\
\bfk&\hbox {if $\GAmma$ has no internal edges.}
\end{cases}
\end{equation}
\end{theorem}

\begin{proof}
The statement is proved by applying the formulas of
\cite[Section~3]{SydneyII} to the particular case of $\Grc$. Notice
that $\GAmma$ has no internal edges if and only if $\gTr(\GAmma) = \emptyset$. 
\end{proof}

The reader may wonder how the formula in~(\ref{Zbijecky duni.})
reflects any relation of an algebra between operations corresponding
to the same underlying

Let us describe the operad structure of $\Free(E)$ given
in~(\ref{Zbijecky duni.}). 
Recall first that the local terminal
objects in the category $\Grc$ are ordered graphs with no internal
edges, i.e.~ordered corollas. The operad $\Free(E)$ is strictly
unital in the sense of~\cite[Definition~6.2]{SydneyI}, 
with the transformation $\eta$
in~\cite[eqn.~(44)]{SydneyI} given by the defining identity
\[
\Free(E)(\GAmma) = \bfk \hbox { if $\GAmma$ is local terminal.}
\]
We describe next the action of the groupoid $\QV(e)$
generated by local isomorphisms, local
reorderings and morphisms changing the global orders of legs of 
graphs. Let us start with the latter.

Let $T \in \gTr(\GAmma)$ be a graph tree and $\vartheta : \GAmma \to \UPsilon$ be an
isomorphism changing the global orders of the legs. In other words, the
graph $\UPsilon$ differs from $\GAmma$ only by the order of its legs.
Since the legs of $\GAmma$ are the same as the legs of the graph $\Gamma_1$
decorating the root of $T$, one also has the induced isomorphism
$\vartheta_1 :  \Gamma_1 \to \Delta_1 \in \QV(e)$, 
where $\Delta_1$ is obtained from $\Gamma_1$ by reordering its legs 
according to $\vartheta$.

We define $S\in \gTr(\UPsilon)$ to 
be the graph-tree whose underlying tree is the same as the
underlying tree of $T$, its edges
have the same decorations as the corresponding edges in $T$, 
and also the vertices have the same decorations as in $T$ except for the root
vertex of $S$ which is decorated by $\Delta_1$. If $T$ has levels
$\lambda \in \Lev(T)$, we
equip $S$ with the same levels. One then has the action
\[
\xymatrix@1@C=5em{
E(T,\lambda) = E(\Gamma_1) \ot \cdots \ot E(\Gamma_k)\
\ar[r]^{E(\vartheta_1) \ot \id^{\ot{k-1}}} &  \ E(\Delta_1) \ot \cdots \ot E(\Gamma_k) =  E(S,\lambda)
}
\]
induced by the $\QV(e)$-action
$E(\vartheta_1): E(\Gamma_1) \to E(\Delta_1)$ on the generating
collection.  The above actions assemble into an action $\Free(E)(\GAmma)
\to \Free(E)(\UPsilon)$ on the
colimits~(\ref{Zbijecky duni.}). 

The actions of local isomorphisms and local reorderings are defined
similarly, so we can be brief.  Given $T \in \gTr(\GAmma)$, a local
reordering of $\GAmma$ induces in the obvious way local reorderings of
the graphs decorating the vertices of $T$, and therefore also on the
products~(\ref{Krtek_na_mne_kouka.}). The reader may have a look at
the proof of Proposition 5.10 in \cite{SydneyII} for a detailed
description of the action of the groupoid of local isomorphisms.  The
example presented in Figure~\ref{Pristi tyden mne ceka laborator.}
should also be helpful.

Local isomorphisms act by reorderings of
the set $V$ of vertices of $\GAmma$. Note that,  
by the definition of a graph-tree, the set $V$ and its order determine
the labels of the edges of $T$, so a reordering of $V$ may
change the
labels of the edges of $T$. Thus, according to Compatibility~1 for 
graph-trees, it induces local isomorphisms of the graphs 
decorating the vertices of $T$ which
in turn act on the
products~(\ref{Krtek_na_mne_kouka.}).

Let us finally attend to the operad composition. That is, for an elementary
morphism $F \fib \GAmma \stackrel\phi\to \UPsilon$ in $\Grc$, we must describe a~map
\begin{equation}
\label{Vcera_mne_natacel_Oliver_a_pak_jsme_se_spolu_opili.}
\circ_\phi:
\Free(E)(\UPsilon) \ot  \Free(E)(F)  \longrightarrow \Free(E)(\GAmma) .
\end{equation}
Given such a $\phi$, one can find
as in
the previous pages a canonical contraction  $\widehat F \fib \GAmma
\stackrel{\widehat \phi}\to
\widehat \UPsilon$ and an isomorphism $\sigma : \UPsilon \to
\widehat \UPsilon$ in the commutative diagram
\[
\xymatrix@C1em{&\GAmma \ar[ld]_\phi \ar[rd]^{\widehat \phi}  & 
\\
\UPsilon\ar[rr]^\sigma_\cong &&\ \widehat \UPsilon.
}
\] 
Using the
equivariance~\cite[eqn.~(21)]{SydneyII} with $\omega = \id$,
$\phi'=\phi$ and $\phi'' = \widehat \phi$, 
we see that $\circ_\phi$ is uniquely determined by
$\circ_{\widehat \phi}$. So we may assume that $\phi$
in~(\ref{Vcera_mne_natacel_Oliver_a_pak_jsme_se_spolu_opili.}) is a canonical contraction.

Let $S \in \gTr(\UPsilon), R \in \gTr(F), \lambda' \in \Lev(S)$ and  $\lambda''
\in \Lev(R)$. Let also $x \in \Vert(\UPsilon)$ be the vertex over which the unique
nontrivial fiber of $\phi$ lives. We define $T \in \gTr(\GAmma)$ as the
graph-tree whose underlying tree is obtained by grafting the root of the
underlying tree of~$R$ to the leg of the underlying tree of $S$
labelled by $x$. The decorations of $T$ is inherited from the
decorations of its graph-subtrees $S$ and $R$. It is simple to check that,
since $\phi$ is a cc, $T$~is indeed a graph-tree.

We finally define $\lambda = \lambda' \circ_\phi \lambda''\in \Lev(T)$ 
by postulating that all
vertices of $R$ are below the vertices of $S$ and that the restriction
of $\lambda$ to the subtrees $S$ resp.~$R$ is $\lambda'$
resp.~$\lambda''$. The
map~(\ref{Vcera_mne_natacel_Oliver_a_pak_jsme_se_spolu_opili.}) is
then the colimit of the obvious canonical isomorphisms
\[
E(T,\lambda) \cong E(S,\lambda') \ot E(R,\lambda'').
\]

\begin{remark}
\label{Je_streda_a_melu_z_posledniho.}
When the generating collection is evenly graded, the elements of the
product~(\ref{Krtek_na_mne_kouka.}) represents the {\em same\/}
elements of $\Free(E)(\GAmma)$ for an arbitrary  choice of $\lambda$,
thus~(\ref{Zbijecky duni.}) can be replaced by a more friendly formula
\[
\Free(E)(\GAmma) \cong \bigoplus_{T \in \gTr(\GAmma)} 
\bigotimes_{v \in \Vert(T)} E(\Gamma_v).
\]
As illustrated in Warning~\ref{Dnes jdu s Jaruskou na koncert.},
this simplification is not possible for general collections.
Yet, since the input edges of each graph-tree are ordered, there
exists a preferred choice of the levels specified by the following
lexicographic rule.

Assume that $a < b$ are (the labels of) two input edges of a vertex $v
\in \Vert(T)$. Then all levels of the subtree of $T$ with the root $a$ are
{\em below\/} the levels of the subtree with the root $b$.
Denoting by $\lambda_\lex$ the above arrangement, then 
\begin{equation}
\label{Prvni_na_Vivat_tour.}
\Free(E)(\GAmma) \cong \bigoplus_{T \in \gTr(\GAmma)} \ 
E(T,\lambda_\lex).
\end{equation}
One must however keep in mind that the combination $\lambda_\lex'
\circ_\phi \lambda_\lex''$ of two lexicographic arrangements may not be
lexicographic.  Thus, if we want to use~(\ref{Prvni_na_Vivat_tour.})
the operadic composition based on the isomorphism
\[
E(T,\lambda_\lex'
\circ_\phi \lambda_\lex'') \cong E(S,\lambda_\lex') \ot E(R,\lambda_\lex'')
\]
must be followed by bringing the result back into the preferred form.  
\end{remark}

\section{Minimal model for $\termGrc$.}
\label{hadrova_panenka}

The aim of this section is to construct an explicit minimal model of the
terminal $\Grc$-operad $\termGrc$ governing non-genus graded modular
operads. 

\subsection{Free operads and derivations.}
Free $\Grc$-operads are graded,
\[
\Free(E)(\GAmma) = \bigoplus_{n \geq 0} \Free^n(E)(\GAmma),\ \Gamma
\in \Grc,
\]
where $\Free(E)^0(\GAmma) = \bbk$ and the higher pieces are given
by the modification of~(\ref{Zbijecky duni.}):
\begin{equation}
\label{Athalia}
\Free^n(E)(\GAmma) \cong 
\bigoplus_{T \in \gTr^n(\GAmma)} \  \colim{\lambda \in
  \Lev(T)} E(T,\lambda),
\end{equation}
in which $\gTr^n(\GAmma)$ is, for $n \geq 1$, the subset of
$\gTr(\GAmma)$ consisting of graph-trees $T$ with exactly $n$ vertices.
Clearly $\Free^1(E)(\GAmma) \cong E(\GAmma)$. To describe
$\Free^2(E)(\GAmma)$, we realize that there is precisely one way to
introduce levels into a graph-tree $T \in  \gTr^2(\GAmma)$,
so~(\ref{Athalia}) takes the~form
\begin{equation}
\label{Athalia1}
\Free^2(E)(\GAmma) \cong 
\bigoplus_{T \in \gTr^2(\GAmma)} \  E(\Gamma_v) \ot  E(\Gamma_u),
\end{equation}
where $\Gamma_v$ (resp.~$\Gamma_u$) is the graph decorating the vertex
$v$ at the top level of $T$ (resp.~the vertex
$u$ at the bottom level of $T$). We also have the obvious

\begin{definition}
A degree $s$ linear map $\varpi
: \Free(E) \to \Free(E)$ of collections
is a degree $s$ {\em derivation\/}~if
\[
\varpi\,\circ_\phi = \circ_\phi (\varpi \ot \id) +  \circ_\phi(\id \ot
\varpi),
\]
for every elementary morphism $F \fib \GAmma \stackrel\phi\to \UPsilon$
and $\circ_\phi$ as
in~(\ref{Vcera_mne_natacel_Oliver_a_pak_jsme_se_spolu_opili.}). 
\end{definition}

As expected, every derivation $\varpi$ is 
determined by its restriction
$
\varpi|_E : E = \Free^1(E) \to   \Free(E),
$
and every such a map extends to a derivation.

\begin{remark}
\label{sempre_dolens}
Given a linear map $\omega: E \to \Free(E)$, its extension $\varpi :
\Free(E) \to \Free(E)$ into a~derivation is obtained 
by subsequent applications of $\omega$
to the factors $E(\Gamma_i)$, $1 \leq i \leq k$, of $E(T,\lambda)$
in~(\ref{Krtek_na_mne_kouka.}), replacing each of these factors by its $\omega$-image.
\end{remark}

\subsection{Minimal models} 
They came to life, for dg associative commutative
resp.~dg Lie algebras, as the Sullivan
resp.~Quillen minimal models of rational homotopy types, 
see~\cite{T} and citations therein. 
Minimal models for (classical) operads were
introduced and studied in~\cite{zebrulka}, while minimal models for
(hyper)operads governing permutads were treated in~\cite{perm}. 
Below we give a~definition for $\Grc$-operads, 
definitions for other types of (hyper)operads featuring
in this paper are obvious modifications and we will thus not spell them out
explicitly. 

\begin{definition}
The {\em minimal model\/} of  a  dg $\Grc$-operad $\oP$ is dg
$\Grc$-operad $\Min$
together with a dg $\Grc$-operad morphism $\rho : \Min \to \oP$, such
that 
\begin{itemize}
\item [(i)]
the component $\rho(\Gamma) : \Min(\Gamma) \to
\oP(\Gamma)$ of $\rho$ is a homology isomorphism of dg vector spaces for each $\Gamma \in
\Grc$, and
\item [(ii)]
the underlying non-dg $\Grc$-operad of $\Min$ is free, and the differential
$\pa$ of $\Min$ has no constant and linear terms (the {\em minimality condition\/}).
\end{itemize}
\end{definition}

One can prove, adapting the proof of Theorem~II.3.127 in~\cite{MSS},
that minimal models are unique up to isomorphism.
Our construction of the minimal model for $\termGrc$ begins by
describing its generating $1$-connected collection.
For a vector space $A$ of dimension $k$, we denote by $\det(A) :=
\hbox {\large$\land$}^k(A)$ the
top-dimensional piece of its Grassmann algebra. If $S$ is a non-empty finite set,
we let $\det(S)$ to be the determinant of the vector space spanned
by $S$.
Given two finite sets $S_1 = \{e^1_1,\ldots,e^1_a\}$, 
$S_2 = \{e^2_1,\ldots,e^2_b\}$, we define
\[
\omega_{S_1,S_2}:
\det(S_1 \sqcup S_2) \to \det(S_1) \ot \det(S_2).
\]
by
\[
\omega_{S_1,S_2}(e^1_1 \land \cdots \land e^1_a \land  
e^2_1 \land \cdots \land e^2_b) := (e^1_1 \land \cdots \land e^1_a) \ot
(e^2_1 \land \cdots \land e^2_b).
\]
Let, for $\Gamma \in \Grc$, $\Edg(\Gamma)$ denote the set of its
internal edges, and
$\det(\Gamma) := \det(\Edg(\Gamma))$. With this notation,
the generating collection of the
minimal model for $\termGrc$ is defined as the one-dimensional vector space
\begin{equation}
\label{Flicek}
D(\Gamma) := \det(\Gamma), \ \Gamma \in \Grc, 
\end{equation}
placed in degree $|\Gamma| := {\rm card}(\Edg(\Gamma)) -1$ if $\Gamma$
has at least one internal edge, while $D(\Gamma) := 0$ if $\Gamma$ is
a corolla. Notice that for $\Gamma$ with exactly one internal edge,
$\det(\Gamma)$ is canonically isomorphic to $\bbk$.

The degree $-1$ differential~$\pa$ will be
determined by its restriction (denoted by the same symbol)
\[
\pa : D \to \Free^2(D) \subset \Free(D)
 \]
as follows. Given $T \in \gTr^2(\Gamma)$, let  $\Gamma_v,\Gamma_u \in \Grc$ 
have the same meaning
as in~(\ref{Athalia1}), and $E_v := \Edg(\Gamma_v)$, $E_u  := \Edg(\Gamma_u)$.
For $\mu \in D(\Gamma) = \det(\Gamma)$ we put 
\begin{subequations}
\begin{equation}
\label{Flicek_na_mne_kouka.}
\pa_T(\mu) := \bigoplus_{T \in \gTr^2(\Gamma)} \pa_T(\mu),
\end{equation} 
where
\begin{equation}
\label{Posledni patek v Sydney}
\partial_T(\mu) := (-1)^{|\Gamma|} \omega_{E_v,E_u} (\mu) \in
D(\Gamma_v) \ot D(\Gamma_u) \subset \Free^2(D)(\Gamma).
\end{equation}
\end{subequations}

\begin{lemma}
The derivation $\pa$ defined above is a differential, i.e.\ $\pa^2=0$.  
\end{lemma}

\begin{proof}
It is simple to see that $\partial^2$ is a derivation as well, 
so it suffices only to  verify that $\partial^2$ vanishes on the generating
collection. We leave this as an exercise to the reader.
\end{proof}

Let $\rho : \Free(D) \to \termGrc$ be the unique map of $\Grc$-operads
whose restriction $\rho|_{D(\Gamma)}$ is, for $\Gamma \in \Grc$, given by 
\begin{equation}
\label{Obehnu_to_dnes?}
\rho|_{D(\Gamma)} := 
\begin{cases}
\id_\bbk : D(\Gamma) = \bbk \to \bbk =
\termGrc(\Gamma), &\hbox {if $|\Edg(\Gamma)| = 1$,
  while}
\\
0,& \hbox {if  $|\Edg(\Gamma)| \geq 2$.}  
\end{cases}
\end{equation}
Having all this, we formulate:

\begin{theorem}
\label{Woy-Woy}
The object $\minGrc := (\Free(D),\pa) \stackrel\rho\to (\termGrc,\pa
=0)$ is a minimal model of the terminal $\Grc$-operad.
\end{theorem}

The rest of this section is devoted to the proof of
Theorem~\ref{Woy-Woy} and of the necessary auxiliary material.

\subsection{Constructs represent graph-trees.}
\label{Srni}
The material of this subsection is based on modification and
generalization of~\cite{JO}.  We start by associating to each object
$\Gamma$ of ${\tt Grc}$ a hypergraph ${\bf H}_{\Gamma}$ defined as
follows: the vertices of ${\bf H}_{\Gamma}$ are the internal edges of
$\Gamma$ and two vertices are connected by an edge in
${\bf H}_{\Gamma}$ whenever, as edges of $\Gamma$, they share a common
vertex. Observe that the leaves of $\Gamma$ play no role in the
definition of ${\bf H}_{\Gamma}$.

  \begin{example}\label{ex1}
Here is an example of the association of a hypergraph to a  graph:
 \begin{center} \raisebox{2em}{$\Gamma=$}
 \begin{tikzpicture}
    \node (E)[circle,draw=black,thick,fill=white,minimum size=2mm,inner sep=0.2mm] at (0.5,0) {$1$};
    \node (F) [circle,draw=black,thick,fill=white,minimum size=2mm,inner sep=0.2mm] at (-0.5,1.5) {$2$};
    \node (A) [circle,draw=black,thick,fill=white,minimum size=2mm,inner sep=0.2mm] at (1.5,1.5) {$3$}; 
     \node (x) [draw=none,minimum size=4mm,inner sep=0.1mm] at (0.5,2.075) { $x$}; 
          \node (y) [draw=none,minimum size=4mm,inner sep=0.1mm] at (0.5,0.85) {$y$}; 
           \node (z) [draw=none,minimum size=4mm,inner sep=0.1mm] at (2.8,1.5) {$z$}; 
            \node (u) [draw=none,minimum size=4mm,inner sep=0.1mm] at (-0.15,0.65) { $u$}; 
           \node (v) [draw=none,minimum size=4mm,inner sep=0.1mm] at (1.15,0.65) {$v$}; 
     \draw[thick] (F) to[out=40,in=140](A) node  {};
     \draw[thick] (A) to[out=40,in=-40,looseness=21](A) node {};
      \draw[thick] (F) to[out=-40,in=-140](A) node {};
    \draw[thick](E)--(F) node {};
    \draw[thick] (E)--(A) node  {};
   \draw[thick] (F)--(-1.3,1.75) node  {};
 \draw[thick] (F)--(-1.3,1.25) node  {};
 \end{tikzpicture} \quad\quad\quad 
 \begin{tikzpicture}
 \node (x) [circle,draw=black,fill=black,inner sep=0mm,minimum size=1.3mm,label={[xshift=-0.22cm,yshift=-0.28cm]{$x$}}] at (0,1) {};
\node (y) [circle,draw=black,fill=black,inner sep=0mm,minimum size=1.3mm,label={[xshift=0.2cm,yshift=-0.14cm]{$y$}}] at (1,1) {};
\node (z) [circle,draw=black,fill=black,inner sep=0mm,minimum size=1.3mm,label={[xshift=0.22cm,yshift=-0.28cm]{$z$}}] at (2,1) {};
\node (u) [circle,draw=black,fill=black,inner sep=0mm,minimum size=1.3mm,,label={[xshift=-0.22cm,yshift=-0.28cm]{ $u$}}] at (0.5,0.1) {};
\node (v) [circle,draw=black,fill=black,inner sep=0mm,minimum size=1.3mm,,label={[xshift=0.22cm,yshift=-0.28cm]{$v$}}] at (1.5,0.1) {};
\draw[thick] (z)--(v)--(u)--(x)--(y)--(z);
\draw[thick] (u)--(y)--(v)--(x);
\draw[thick] (x) to[out=40,in=140] (z);
 \end{tikzpicture}\raisebox{2.2em}{$={\bf H}_{\Gamma}$}
 \end{center}
 \end{example}

Assume that $\Gamma = (V,F)$ is a graph with 
the structure map  $g : F \to V$. Choose a subset $V' \subset V$ and a
subset $E'$  of edges of $\Gamma$
formed by the half-edges in $g^{-1}(V')
\subset F$ such that the subgraph of $\Gamma$ spanned by $E'$ is
connected. Let $\Gamma'$ be the graph $\Gamma' = (V',F')$ with $F' :=
g^{-1}(V')$, with the structure map $g' : F' \to V'$ given by the restriction of
$g$, and the involution which coincides with the
involution of $\Gamma$ on the half-edges forming the edges in~$E'$,
and which is
trivial on the remaining half-edges of $\Gamma'$. 

To simplify the
terminology, we will still call $\Gamma'$ a subgraph of $\Gamma$
determined by the set of edges $E$ though,
formally speaking, $\Gamma'$ is obtained from an actual subgraph of $\Gamma$
by cutting some of its edges in two half-edges. 
For example, the `subgraph' 
of the graph $\Gamma$ from Example~\ref{ex1} determined by the internal edge $x$ is 
\begin{center}
\begin{tikzpicture}
    \node (F) [circle,draw=black,thick,fill=white,minimum size=2mm,inner sep=0.2mm] at (-0.5,1.5) {$2$};
    \node (A) [circle,draw=black,thick,fill=white,minimum size=2mm,inner sep=0.2mm] at (1.5,1.5) {$3$}; 
     \node (x) [draw=none,minimum size=4mm,inner sep=0.1mm] at (0.5,2.075) { $x$};  
     \draw[thick] (F) to[out=40,in=140](A) node  {};
  \draw[thick] (2.3,1.25)--(A)--(2.3,1.75);
    \draw[thick](0.7,1.3)--(A)--(0.9,0.9);
    \draw[thick] (0.3,1.3)--(F)--(0.1,0.9);
   \draw[thick] (F)--(-1.3,1.75) node  {};
 \draw[thick] (F)--(-1.3,1.25) node  {};
\end{tikzpicture}
\end{center}  

\begin{lemma}
\label{connectededges}  
The connected 
subgraphs, in the above relaxed sense, 
of a graph $\Gamma$ that have at least one internal edge
are in one-to-one correspondence with the connected subsets of
${H}_{\Gamma}$, i.e.~with the non-empty subsets $X$ of vertices of
\,${\bf H}_{\Gamma}$ such that the hypergraph $({\bf H}_{\Gamma})_{X}$
is~connected.
\end{lemma}

 \smallskip
\begin{remark}
Thanks to Lemma \ref{connectededges}, for  a graph $\Gamma$ and $\emptyset\neq X\subseteq \Edg(\Gamma)$, we can   index  the connected components of ${\bf H}_{\Gamma}\backslash X$ by the corresponding  subgraphs of $\Gamma$, by  writing  $${\bf H}_{\Gamma}\backslash X\leadsto {\bf H}_{{\Gamma}_1},\dots,{\bf H}_{{\Gamma}_n}.$$  
 Observe that the subgraphs ${\Gamma}_1,\dots,{\Gamma}_n$ of ${\Gamma}$ do not in general make a decomposition of ${\Gamma}$, in the sense that the removal  of the edges from the set $X$ may result in a number of subgraphs of ${\Gamma}$ reduced to a corolla without internal edges.
\end{remark}

\begin{proposition}
\label{jovanica}
There exists a natural isomorphism $\alpha_\Gamma : {\mathcal A}({\bf
  H}_{\Gamma}) \stackrel\cong\longrightarrow {\tt gTr}(\Gamma)$ between the abstract
  polytope ${\mathcal A}({\bf H}_{\Gamma})$ of constructs of the
  hypergraph $\H_\Gamma$ and
the poset ${\tt gTr}(\Gamma)$ of graph-trees such that $\gr(T)
=\Gamma$.
\end{proposition}

\begin{proof} 
We define the announced one-to-one correspondence $\alpha_{\Gamma}$ between
constructs  \hbox{$C:{\bf H}_{\Gamma}$} and  graph-trees $T \in \gTr(\Gamma)$  by induction on the number of vertices of   $C$.  
If $C$ is the maximal construct $\Edg(\Gamma): \bfH_\Gamma$, then
$\alpha_\Gamma(T)$ is the planar rooted corolla
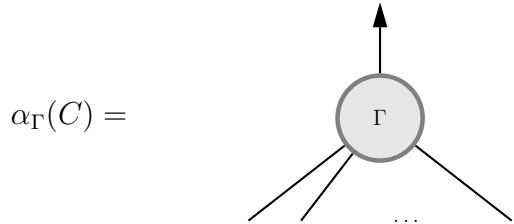
\begin{center} \raisebox{3em}{$\alpha_{\Gamma}(C)=$}\,
\psscalebox{0.7 0.7} 
{\begin{pspicture}(0,0.0582557)(13,4.0582557)
\newgray{kgray}{.9}
\psline[linecolor=black, linewidth=0.04, arrowsize=0.2cm 2.0,arrowlength=1.7,arrowinset=0.0]{->}(4.5,1.9762975)(4.5,4.1512976)
\psline[linecolor=black, linewidth=0.04](4.5,1.9512974)(2,0)
\psline[linecolor=black, linewidth=0.04](4.5,1.9512974)(3,0)
\psline[linecolor=black, linewidth=0.04](4.5,1.9512974)(7,0)
\psdots[linecolor=gray, fillstyle=solid, dotstyle=o, dotsize=1.7, fillcolor=kgray](4.5,1.9512974)
\rput[l](4.3,1.95){$\,\Gamma$}
\rput[l](4.75,0){$\dots$}
\end{pspicture}}  
\end{center}
with the vertex decorated by $\Gamma$ and legs labelled by
the ordered set $\Vert(\Gamma)$.

Suppose that $C=X\{C_1,\dots,C_p\}$, $X \subset \Edg(\Gamma)$,
${\bf H}_{\Gamma}\backslash X\leadsto {\bf H}_{1},\dots, {\bf H}_{p}$
and \hbox{$C_i:{\bf H}_i$} for $1 \leq i \leq p$. By
Lemma~\ref{connectededges}, there are connected 
subgraphs $\Gamma_i$ of $\Gamma$
such that $\bfH_i = {\bf H}_{\Gamma_i}$. There, moreover, exists a
graph $\Gamma_X \in \Grc$ such that 
$\Rada \Gamma1p$ are the fibers of the iterated canonical contraction
$\Gamma \to {\Gamma}_X$.
This understood, we are in the situation when
${\bf H}_{\Gamma}\backslash X\leadsto {\bf H}_{{\Gamma}_1},\dots, {\bf
  H}_{{\Gamma}_p}$ and \hbox{$C_i:{\bf H}_{{\Gamma}_i}$}, $1 \leq i
\leq p$.

The root vertex  of the graph-tree $\alpha_{\Gamma}(C)$ will be
decorated by $\Gamma_X$.
We already have, by induction, the
graph-trees $\alpha_{\Gamma_i}(C_i)$, and each of these
trees is connected with  
the root of $\alpha_{\Gamma}(C)$ by the edge bearing the
label of the vertex of $\Gamma_X$ to which $\Gamma_i$ has been 
contracted. We believe that
Figure~\ref{Zviratka-se-fotila-ve-Woy-Woy.} makes this construction clear.
The inductive step is finished by joining to the root of the
graph-tree $\alpha_{\Gamma}(C)$ the legs indexed by the remaining vertices
of $\Gamma_X$. 

\begin{figure}[h]
\begin{center}
\newgray{kgray}{.9}
  \psscalebox{1.0 1.0} 
{
\begin{pspicture}(0,-3.1006334)(7.8153715,3.1006334)
\psline[linecolor=black, linewidth=0.04, arrowsize=0.2cm 2.0,arrowlength=1.7,arrowinset=0.0]{->}(3.7053716,1.0186751)(3.6,3.193675)
\psline[linecolor=black, linewidth=0.04, arrowsize=0.05291667cm 2.0,arrowlength=1.4,arrowinset=0.0]{->}(3.0053718,-1.806325)(3.9053717,1.0186751)
\psline[linecolor=black, linewidth=0.04, arrowsize=0.05291667cm 2.0,arrowlength=1.4,arrowinset=0.0]{->}(6.4053717,-0.6063249)(3.9053717,1.0186751)
\psline[linecolor=black, linewidth=0.04, arrowsize=0.05291667cm 2.0,arrowlength=1.4,arrowinset=0.0]{->}(1.2053717,-0.6063249)(3.9053717,0.99367505)
\psdots[linecolor=gray, dotstyle=o, dotsize=1.7, fillcolor=kgray](6.6053715,-0.70632493)
\psdots[linecolor=gray, dotstyle=o, linewidth=0.04, dotsize=1.7, fillcolor=kgray](3.0053718,-2.106325)
\psdots[linecolor=gray, dotstyle=o, dotsize=1.7, fillcolor=kgray](1.0053717,-0.70632493)
\psdots[linecolor=gray, dotstyle=o, dotsize=1.7, fillcolor=kgray](3.6053717,0.8936751)
\rput(1,-0.70632493){$\alpha_{\Gamma_1}(C_1)$}
\rput[l](2.3,-2.106325){$\alpha_{\Gamma_2}(C_2)$}
\rput[l](5.9,-0.70632493){$\alpha_{\Gamma_p}(C_p)$}
\rput[l](3.3718,0.9367505){$\Gamma_X$}
\rput[l](4.4053717,-0.8063249){$\cdots$}
\end{pspicture}
}
\end{center}
\caption{An inductive construction of $\alpha_\Gamma(C)$.
\label{Zviratka-se-fotila-ve-Woy-Woy.}}
\end{figure}
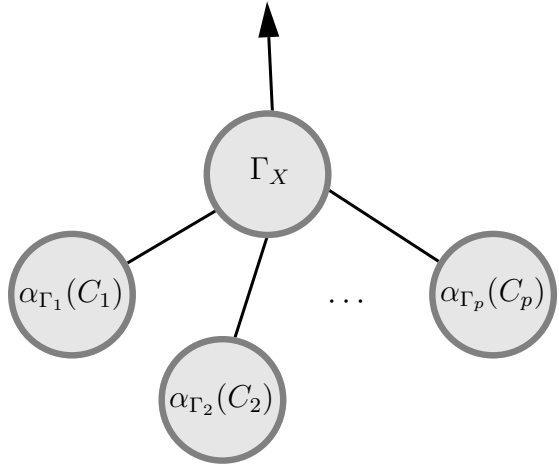

\begin{figure}[h]
\begin{center}
\resizebox{4.5cm}{!}{\begin{tikzpicture}
\node (a) [circle,draw=none] at (0,0) {
\begin{tikzpicture}
\node (A) [circle,draw=black,thick,fill=white,minimum size=5mm,inner sep=0.3mm] at (1.5,1.5) {$1$}; 
 \node (z) [draw=none,minimum size=4mm,inner sep=0.1mm] at (2.85,1.5) {$y$}; 
 \node (z) [draw=none,minimum size=4mm,inner sep=0.1mm] at (2.5,2.1) {$c$}; 
 \node (z) [draw=none,minimum size=4mm,inner sep=0.1mm] at (2.2,0.9) {$d$}; 
 \node (z) [draw=none,minimum size=4mm,inner sep=0.1mm] at (1.7,2.85) {$x$}; 
 \node (z) [draw=none,minimum size=4mm,inner sep=0.1mm] at (1.1,2.3) {$a$}; 
 \node (z) [draw=none,minimum size=4mm,inner sep=0.1mm] at (0.85,1.05) {$i$}; 
 \node (z) [draw=none,minimum size=4mm,inner sep=0.1mm] at (0.7,2) {$j$}; 
 \node (z) [draw=none,minimum size=4mm,inner sep=0.1mm] at (2.25,2.35) {$b$}; 
  \draw[thick] (A) to[out=50,in=110,looseness=21](A) node {};
 \draw[thick] (A) to[out=30,in=-30,looseness=21](A) node {};
\draw[thick] (A) -- (0.5,1.8);
\draw[thick] (A) -- (0.5,1.2);
\end{tikzpicture}
};
\node (b) [circle,draw=none] at (0,-5) {
\begin{tikzpicture}
    \node (E)[circle,draw=black,thick,fill=white,minimum size=5mm,inner sep=0.3mm] at (0.5,0) {$1$};
    \node (F) [circle,draw=black,thick,fill=white,minimum size=5mm,inner sep=0.3mm] at (-0.5,1.5) {$2$};
    \node (A) [circle,draw=black,thick,fill=white,minimum size=5mm,inner sep=0.3mm] at (1.5,1.5) {$3$}; 
\node (z) [draw=none,minimum size=4mm,inner sep=0.1mm] at (2.8,1.5) {$z$}; 
\node (u) [draw=none,minimum size=4mm,inner sep=0.1mm] at (-0.15,0.65) { $u$}; 
\node (v) [draw=none,minimum size=4mm,inner sep=0.1mm] at (1.15,0.65)
{$v$}; 
 \node (z) [draw=none,minimum size=4mm,inner sep=0.1mm] at (0,1.9) {$a$}; 
 \node (z) [draw=none,minimum size=4mm,inner sep=0.1mm] at (0.1,0.98) {$c$}; 
 \node (z) [draw=none,minimum size=4mm,inner sep=0.1mm] at (1,1.9) {$b$}; 
 \node (z) [draw=none,minimum size=4mm,inner sep=0.1mm] at (.85,0.98) {$d$}; 
 \node (z) [draw=none,minimum size=4mm,inner sep=0.1mm] at (-1.1,1.05) {$i$}; 
 \node (z) [draw=none,minimum size=4mm,inner sep=0.1mm] at (-1.1,2) {$j$}; 
     \draw[thick] (.3,1.2)-- (F) -- (0.3,1.8);
     \draw[thick] (A) to[out=30,in=-30,looseness=21](A) node {};
      \draw[thick] (0.7,1.2)--(A)--(0.7,1.8);
    \draw[thick](E)--(F) node {};
    \draw[thick] (E)--(A) node  {};
   \draw[thick] (F)--(-1.3,1.75) node  {};
 \draw[thick] (F)--(-1.3,1.25) node  {};
\end{tikzpicture}};
\draw[thick] (-0.2,-0.15) circle (1.5cm);\draw[thick] (-0.2,-5) circle (2.4cm);
\draw [thick,decoration={markings,mark=at position 1 with
    {\arrow[scale=3,>=stealth]{>}}},postaction={decorate}]  (-0.2,1.35) -- (-0.2,2.2);
\draw[thick] (-0.2,-1.65)--(-0.2,-2.6);
\draw[thick] (-0.2,-7.4)--(-0.2,-8.75);
\draw[thick] (-1.5,-8.75)-- (-1,-7.26);
\draw[thick] (1.1,-8.75)-- (0.6,-7.26);
\end{tikzpicture}}
\end{center}
\caption{An example of $\alpha_\Gamma(C)$.
The symbols $x,y,z$ are labels internal edges, 
the symbols $a,b,c,d,i,j$ are labels of half-edges.
\label{patek}}
\end{figure}

The inverse of $\alpha_{\Gamma}$ is defined by extracting the
construct from a graph-tree $T$ in the following way. First, remove
all the leaves of $T$ and then, for each vertex of $T$, replace the
graph that decorates that vertex by the maximal construct of its
associated hypergraph. 
In more detail, assume that $T \in \gTr(\Gamma)$, $\Gamma \in \Grc$. The
underlying rooted tree of the construct 
$\alpha^{-1}_{\Gamma}(T)$ is obtained from the underlying tree of $T$ 
by removing its legs.
The vertex of  $\alpha^{-1}_{\Gamma}(T)$ corresponding to a
vertex $v \in \Vert(T)$ decorated by $\Gamma_v\in \Grc$ is decorated
by the set $\Edg(\Gamma_v) \subset \Edg(\Gamma)$ of edges of $\Gamma_v$.  

There is the following inductive, alternative construction of
$\alpha^{-1}_\Gamma(T)$ that leads manifestly to 
a construct of $\bfH_\Gamma$. Assume that
$e_1,\ldots,e_s \in V$ are  the labels of the 
incoming internal edges of a vertex $v \in \Vert(T)$, and 
that $v_1,\ldots,v_s \in \Vert(T)$ are the initial vertices of these
edges. Further, let $T_i$ be the maximal rooted
graph-subtree of $T$ with the root $v_i$ and 
$\Gamma_i := \gr(T_i)$, $1 \leq i \leq s$. Then the
corresponding subtree of  $\alpha^{-1}_{\Gamma}(T)$ is the construct 
\[
\Edg(\Gamma_v)\{\alpha^{-1}_{\Gamma_1}(T_1),\ldots,\alpha^{-1}_{\Gamma_s}(T_s)\}.
\]  
Notice that the construct $\alpha^{-1}_{\Gamma}(T)$ inherits the
planar structure of $T$.
It is easy to verify that the correspondence 
\begin{equation}
\label{semper}
\gTr(\Gamma) \ni T \longleftrightarrow
\alpha_\Gamma(T) \in \{C\  |\ C: \H_\Gamma\}
\end{equation}
preserves the poset structures.
\end{proof}

\begin{example}
 For the graph $\Gamma$ from Example \ref{ex1}, the
 graph-tree $\alpha_\Gamma(C)$ associated to the  
construct $C=\{x,y\}\{\{u,v,z\}\}$ of
 the hypergraph ${\bf H}_{\Gamma}$ is shown in 
Figure~\ref{patek}.
\end{example}

For an object $\Gamma$ of ${\Grc}$ and a construct
$C:{\bf H}_{\Gamma}$, let ${\Lev}(C)$ denote the chaotic groupoid whose
objects are all possible arrangements of levels of $C$, whereby a
level of a construct is defined analogously as the one of a graph
tree. It is clear that the correspondence~(\ref{semper}) defines a
canonical isomorphism between ${\Lev}(C)$ and
$\Lev(\alpha_\Gamma(C))$, thus each $1$-connected collection $E$
promotes into a functor
$E: \Lev(C) \longrightarrow \Vect$ in the diagram 
\[
\xymatrix{&\Vect
\\
\Lev(C)\ar[r]^\cong \ar[ur]^E&\ar[u]_E\Lev(\alpha_\Gamma(C))
}
\]
where the vertical up-going arrow is~(\ref{Zviratka_mi_pomahaji.}).
The following reformulation of Theorem~\ref{podlehl_jsem} is a~direct
consequence of Proposition~\ref{jovanica}.
\begin{theorem}
For a $1$-connected collection $E$,  the arity $\Gamma$ piece of the free
operad $\Free(E)$ is given by 
\begin{equation}
\label{Free constructs.}
\Free(E)(\Gamma) \cong 
\begin{cases}
\displaystyle \bigoplus_{C:{\H}_\Gamma} \  \colim{\varsigma \in
  {\Lev}(C)} E(C,\varsigma)&\hbox {if $\Gamma$ has at least one internal edge,
  and}
\\
\bfk&\hbox {if $\Gamma$ has no internal edges.}
\end{cases}
\end{equation}
\end{theorem}

\subsection{A chain complex.} 
\label{Za_chvili_volam_Mikesovi.}
In this subsection we recall a chain complex associated to a convex polyhedron
featuring in Lemma~\ref{Jeste_ani_nevim_kde_budu_v_Melbourne_bydlet.} below.
Let therefore $K$ be such an~$n$-dimensional polyhedron realized as the convex hull of finitely
many points in $\bbR^n$.  Each $k$-dimensional
face  $e$ of~$K$, $0 \leq k \leq n$, 
is then embedded canonically into a $k$-dimensional affine
subspace $\A_e$ of $\bbR^n$, namely into the span of its vertices. 
By an {\em orientation\/} of $e$ we
understand an orientation of $\A_e$.  For $k > 0$, that orientation is
given by choice of a frame in $\A_e$. If $k=0$, $\A_e$ is a~point, and
the orientation is a sign assigned to that point.  We say that $K$ is
{\em oriented\/}, if an orientation of each face has been specified.

Assume that $a$ is a codimension one subface of $e$ and that the dimension
of $a$ is $\geq 1$. 
Clearly $\A_a$ divides $\A_e$ into two half-spaces. Denote by $\A^a_e
\subset \A_e$ the
one having non-empty intersection with $K$. Let the orientation 
of~$a$ be given by  linearly independent vectors $(\Rada v1{k-1})$ in
$\A_a$. We say that an orientation of $a$ is {\em compatible\/} with
the orientation of $e$ if  the frame $(\Rada v1{k-1}, n)$ in $\A^a_e$,
where $n$ is a vector normal to
$\A_a \subset \A^a_e$, defines the orientation of $e$, 
cf.~Figure~\ref{S_Jarkou_u_Pakousu.} (left) where $k=2$. A modification
of this notion to $0$-dimensional $a$'s is~obvious.

\begin{figure}[h]
\[
\psscalebox{.8 .8} 
{
\begin{pspicture}(0,-1.4805)(11.420455,2.1204805)
\psline[linecolor=black, linewidth=0.04](0.4,2.1003487)(0.4,-2.099651)
\rput{-112.21667}(0.68873215,0.46336102){\psarc[linecolor=black, linewidth=0.04, linestyle=dashed, dash=0.17638889cm 0.10583334cm, dimen=outer](0.5,0.0){2.1}{20.410843}{206.00946}}
\psdots[linecolor=black, dotsize=0.22](0.4,-0.09965119)
\psline[linecolor=black, linewidth=0.04, arrowsize=0.05291667cm 3.05,arrowlength=2.14,arrowinset=0.0]{<-}(0.4,0.7003488)(0.4,0.100348815)(0.4,0.100348815)
\psline[linecolor=black, linewidth=0.04, arrowsize=0.05291667cm 3.05,arrowlength=2.14,arrowinset=0.0]{<-}(1.2,-0.09965119)(0.4,-0.09965119)
\psline[linecolor=black, linewidth=0.04, arrowsize=0.05291667cm 3.05,arrowlength=2.14,arrowinset=0.0]{<-}(1.0,1.3003488)(1.0,0.5003488)
\psline[linecolor=black, linewidth=0.04, arrowsize=0.05291667cm 3.05,arrowlength=2.14,arrowinset=0.0]{<-}(1.8,0.5003488)(1.0,0.5003488)
\rput[bl](-.2,0.30034882){$v_1$} \rput[bl](-.2,2.30034882){$\A_a$}
\rput[bl](0.8,-0.4996512){$n$} \rput[bl](1.5,-1){$\A^a_e$}
\rput[bl](1.2,1.1003488){$r_1$}
\rput[bl](1.6,0.7003488){$r_2$}
\psline[linecolor=black, linewidth=0.04](8.8,-0.89965117)(11.25,0.1)
\psline[linecolor=black, linewidth=0.04](8.8,-0.89965117)(6.06,-.09)
\rput{-21.812042}(0.91996664,3.1753144){\psarc[linecolor=black, linewidth=0.04, linestyle=dashed, dash=0.17638889cm 0.10583334cm, dimen=outer](8.699906,-0.7996512){2.6999066}{40.410843}{186.00946}}
\psdots[linecolor=black, dotsize=0.22](8.8,-0.89965117)
\rput[bl](8.8,-0.4996512){$a$}
\rput(7.2,0){$e'$}
\rput(10.0,0){$e''$}
\rput[bl](8.6,0.90034884){$h$}
\end{pspicture}
}
\]
\caption{Configurations of $a$ and $e$ (left) and $a, e', e''$ and
  $h$ (right).\label{S_Jarkou_u_Pakousu.}}
\end{figure}
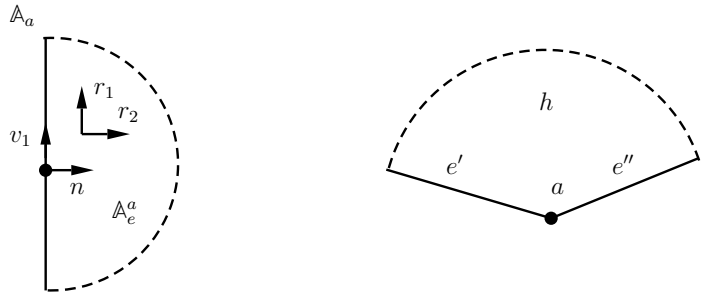

We assign to $K$ a chain complex $(C_*(K),\pa)$ of free abelian groups
whose $k$th piece $C_k(K)$ is generated by $k$-dimensional faces
of $K$. The value of the differential on a $k$-dimensional generator
$\lambda$ is defined by
\[
\pa(\lambda) = \sum \eta^\delta_\lambda \cdot \delta,
\]
where  $\delta$ runs over all
codimension one faces of $\lambda$ and 
\[
\eta^\delta_\lambda :=
\begin{cases}
+1,&\hbox {if $\delta$ is oriented compatibly with $\lambda$, and}
\\
-1,&\hbox {otherwise.}
\end{cases}
\]
It follows from standards methods of algebraic topology that
$(C_*(K),\pa)$ is acyclic in positive dimensions while its $0$th
homology equals ${\mathbb Z}$.

\subsection{An ingenious lemma}
Let $L = (L,\prec)$ be the face poset of an $n$-dimensional
polyhedron $K$, ordered by the inclusion. Assume that $K$
is such that the following `diamond' condition  is satisfied.

\begin{DiamonD}
Let $0< k < n$ and let $a$ be a $(k-1)$-dimensional face of $K$ which is
a common boundary of two $k$-dimensional faces $e',e''$. Then there
exists a $(k+1)$-dimensional face $h$ with $e'$ and $e''$ in its boundary. 
\end{DiamonD}

A concise way to formulate the diamond condition is to say that the
existence of $e'$ and $e''$ with $a \prec e',e''$ implies the
existence of some $h$ with $e',e'' \prec h$, diagrammatically
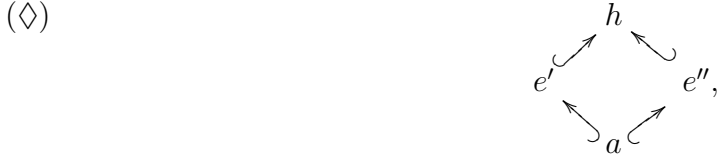
\[
\xymatrix@C=1em@R=.7em{&h&
\\
e'    \ar@{^{(}->}[ur]       && \ e'',  \ar@{_{(}->}[ul]  
\\
&a    \ar@{_{(}->}[ur]    \ar@{^{(}->}[ul]         &
}\tag{$\Diamond$}
\]
hence the name. It follows from the properties of abstract polytopes
that $e'$ and $e''$ are the only faces in the interval $[a,h]$.
The diamond
condition 
guarantees that the inductive construction of an orientation 
in the proof of
Lemma~\ref{Jeste_ani_nevim_kde_budu_v_Melbourne_bydlet.} below is
independent of the choices. 

The diamond condition need not be satisfied by
a~general~polytope. An example is the pyramid, with $e',e''$  a pair of
opposite $1$-dimensional edges meeting at the apex.

Assume that $(C_*(L),\pa)$ is a chain complex such that each $C_k(L)$ is the
free abelian group generated by $k$-dimensional elements of $L$, $0
\leq k \leq n$. Suppose moreover that, for each $\lambda \in L$,
$\pa(\lambda)$ is of the form
\[
\pa(\lambda) = \sum \eta^\delta_\lambda \cdot \delta,
\]
where $\eta^\delta_\lambda \in \{-1,+1\}$ and $\delta$ runs over all
codimension one faces of $\lambda$. Then one has:

\begin{lemma}
\label{Jeste_ani_nevim_kde_budu_v_Melbourne_bydlet.}
The faces of $K$ could be oriented so that $(C_*(L),\pa)$
is the chain complex $(C_*(K),\pa)$ recalled in
Subsection~\ref{Za_chvili_volam_Mikesovi.}. 
\end{lemma}

\begin{proof}
The lemma will be proved by downward induction on the dimension of
the faces of~$K$. We start by choosing  an orientation of 
the unique $n$-dimensional face of
$K$ arbitrarily. 

Assume that we have oriented all faces of $K$ of dimensions  $\geq k$
for some $n > k \geq 0$. Let $a$ be a \hbox{$(k\!-\!1)$} dimensional face of $K$, and
choose some $k$ dimensional face $e$ such that $a \prec e$. This is
always possible, since otherwise the face $a$ would be maximal, which
contradicts the properties of a polytope.
If $a$ occurs in $\pa (e)$ with the $+1$
sign, we equip it with the compatible orientation, if it occurs with the
$-1$ sign, we equip it with the orientation opposite to the compatible
one. We need to show that this recipe does not depend on the choice of
$e$.

Assume therefore that $e'$ and $e''$ are two faces of $K$ with the properties
described above. Let $h$ be a cell required by the diamond
property. Then 
\begin{align*}
\pa(h)& = \eta' \cdot e' +  \eta'' \cdot e'' + \hbox{other terms,}
\  \eta', \eta'' \in \{-1,+1\},
\\
\pa(e')&= \varepsilon' \cdot a +  \hbox{other terms,}
\  \varepsilon' \in \{-1,+1\},\ \hbox {and}
\\
\pa(e'')&= \varepsilon'' \cdot a +  \hbox{other terms,} \  
\varepsilon'' \in \{-1,+1\}.
\end{align*}
The condition $\pa^2(h) = 0$ together with the fact that $e'$ and $e''$
are the only faces in the interval $[a,h]$ imply 
\begin{equation}
\label{Dnes_zmena_casu.}
\eta'\varepsilon' + \eta''
\varepsilon'' =0.
\end{equation}
The configuration of the relevant cells is indicated in
Figure~\ref{S_Jarkou_u_Pakousu.} (right) which shows a section of $h$ with a
hyperplane orthogonal to $\A_a$.

Assume e.g.~that $\eta' = \eta'' = 1$. Then both $e'$ and $e''$ have
the orientation compatible with the orientation of $h$. 
By~(\ref{Dnes_zmena_casu.}) one has $\varepsilon'
= - \varepsilon''$; assume for instance that $\varepsilon' = 1,
\varepsilon'' = -1$. Then $a$ gets from $e'$ the compatible orientation,
and from $e''$ the orientation opposite to the compatible one. 
It easily follows from the local
geometry of the section in Figure~\ref{S_Jarkou_u_Pakousu.} that these
two orientations of $a$ are the same. The remaining cases can be
analyzed similarly.
\end{proof}

\subsection{Splits and collapses}
The proof of Lemma \ref{Panenka_na_mne_kouka.} below relies on the
actions of {\em splitting\/} the vertices and {\em collapsing\/} the edges
of constructs of a hypergraph ${\bf H}$. We formalize the
corresponding constructions below. Let $C:{\bf H}$.
\paragraph{\em Splitting the vertices of $C$.} Let
$V\in \mbox{vert}(C)$ be such that $|V|\geq 2$. Let ${\bf H}\!-\!V$ be the
hypergraph defined by
$${\bf H}\!-\!V:=\{X\backslash V\,|\, X\in {\it Sat}({\bf
  H})\}\backslash\{\emptyset\}.$$ Observe that, in general,
${\bf H}\!-\!V\neq {\bf H}\backslash V$. For example, for the hypergraph
${\bf H}_{\Gamma}$ from Example~\ref{ex1}, we have that
${\bf H}_{\Gamma}\!-\!\{x,y\}$ is the complete graph on the vertex set
$\{z,u,v\}$, whereas ${\bf H}_{\Gamma}\backslash\{x,y\}$ can be
obtained from ${\bf H}_{\Gamma}\!-\!\{x,y\}$ by removing the edge
$\{u,z\}$ and, hence, is a linear graph. Let $\{X,Y\}$ be a partition
of $V$ such that the tree $X\{Y\}$ is a construct of ${\bf H}\!-\!V$.  We
define the construct $C[X\{Y\}/X\cup Y]:{\bf H}$, obtained from C by
splitting the vertex $V$ into the edge $X\{Y\}$, by induction on the
number of vertices of $C$, as follows. If $C=H$, we set
$C[X\{Y\}/X\cup Y]:=X\{Y\}$. 

Suppose that, for $Z\subset H$,
$C=Z\{C_1,\dots,C_p\}$,
${\bf H}\backslash Z\leadsto {\bf H}_{1},\dots, {\bf H}_{p}$ and
$C_i:{\bf H}_{i}$.  If there exists an index $i$, $1\leq i\leq p$, such
that $V\in \mbox{vert}(C_i)$, we define
$$C[X\{Y\}/X\cup Y]:=Z\{C_1,\dots,C_{i-1},C_i[X\{Y\}/X\cup Y]
,C_{i+1},\dots,C_p\}.$$ 
Assume that $V=Z$ and let
$\{i_1,\dots,i_q\}\cup \{j_1,\dots,j_r\}$ be the partition of the set
$\{1,\dots,p\}$ such that the hypergraphs ${\bf H}_{i_s}$, for
$1\leq s\leq q$, contain a vertex adjacent to some vertex of~$Y$,
while the hypergraphs ${\bf H}_{i_t}$, for $1\leq t\leq q$, have no
vertices adjacent to a vertex of~$Y$.  We define
$$C[X\{Y\}/X\cup Y]:=X\{Y\{C_{i_1},\dots
C_{i_q}\},C_{j_1},\dots,C_{j_r}\}.$$ If, exceptionally,
$\{i_1,\dots,i_q\}=\emptyset$ resp.\ $\{j_1,\dots,j_r\}=\emptyset$,
we set 
\[
C[X\{Y\}/X\cup Y]:=X\{Y,C_1,\dots,C_p\}
\hbox { resp.} \ C[X\{Y\}/X\cup Y]:=X\{Y\{C_1,\dots,C_p\}\}.
\]

The proof that the non-planar rooted tree $C[X\{Y\}/X\cup Y]$ is
indeed a construct of ${\bf H}$ goes easily by induction on the number
of vertices of $C$, the only interesting case being 
$C=Z\{C_1,\dots,C_p\}$. In that
case, the argument is based on the fact that the set of vertices
$Y\cup \bigcup_{i\in \{i_1,\dots,i_q\}}\mbox{vert}({\bf H}_i)$
determines a connected component ${\bf H}'$ of ${\bf H}$ and,
furthermore, that $Y\{C_{i_1},\dots C_{i_q}\}:{\bf H}'$.

\paragraph{\em Collapsing the edges of $C$.} One can similarly define the construct \hbox{$C[X\cup Y/X\{Y\}]\!:\!{\bf H}$}, obtained from C by collapsing the edge $X\{Y\}$ into the vertex $X\cup Y$. 
\begin{lemma}
\label{Panenka_na_mne_kouka.}
The polytope ${\mathcal C}(\bfH, \Pi)$ that realizes the 
abstract polytope ${\mathcal A}(\H)$ (see Lemma~\ref{abspol}) of a
hypergraph $\H$   satisfies the
diamond property.
\end{lemma}

\begin{proof}
We prove the lemma  by constructing, for each construct $C:{\bf H}$ of rank $k-1$ for which there exist constructs $C'$ and $C''$ of rank $k$ such that 
\begin{equation}\label{joca}
C\leq_{{\bf H}} C'\quad \mbox{and} \quad C\leq_{{\bf H}} C'',
\end{equation}
  a construct $D:{\bf H}$ of rank $k+1$ such that $C'\leq_{{\bf H}} D$ and $C''\leq_{{\bf H}} D$.

By definition of the partial order $\leq_{{\bf H}}$ of ${\mathcal
  A}({\bf H})$, the relations \eqref{joca}, together with the fact
that the rank of $C$ differs by $1$ from the rank of $C'$ and $C''$,
mean that there exists a vertex $X\cup Y$ of $C'$ and a vertex $U\cup
V$ of $C''$, such that 
\[
C=C'[X\{Y\}/X\cup Y]=C''[U\{V\}/U\cup V].
\] 
As vertices of $C$, the sets $X$, $Y$, $U$ and $V$ satisfy one of the
following relations: 
they can either
be mutually disjoint, or it can be the case that $X=U$ and
$Y\cap V=\emptyset$, or it can be the case that $Y=U$ and
$X\cap V=\emptyset$, plus the `mirror' reflection of the last case, namely $X=V$
and $U\cap V = \emptyset$. 

It is easily seen that other possible relations
are forbidden. For example, the relation $Y=V$ would imply that $C$ is
not a rooted tree. Depending on the mutual relation of the vertices
$X$, $Y$, $U$ and $V$ of $C$, the above equality implies that the
action of collapsing a~particular edge of $C'$ and a~particular edge
of $C''$ leads to the same construct. Indeed, if $X$, $Y$, $U$ and $V$
are mutually disjoint, then
$$C'[U\cup V/U\{V\}]=C''[X\cup Y/X\{Y\}],$$ if $X=U$ and
$Y\cap V=\emptyset$, then
$$C'[(X\cup Y)\cup V/(X\cup Y)\{V\}]=C''[(X\cup V)\cup Y/(X\cup
V)\{Y\}],$$ and if $Y=U$ and $X\cap V=\emptyset$, then
$$C'[(X\cup Y)\cup V/(X\cup Y)\{V\}]=C''[X\cup (Y\cup V)/X\{Y\cup
V\}].$$ 

We define $D$ to be precisely the construct obtained from $C'$
(or, equivalently, from $C''$) by such a collapse. The three diamonds
corresponding to the three possible constructions of $D$ can be
pictured respectively as follows:
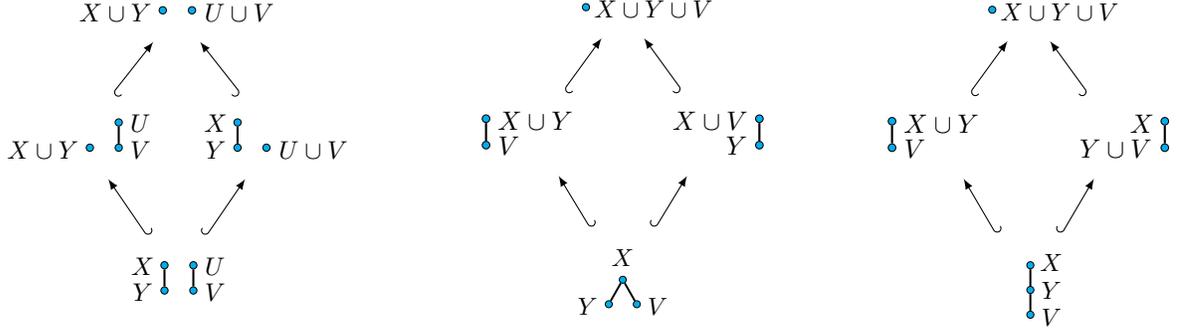
\begin{center}
\raisebox{0.32cm}{\begin{tikzpicture}
\node (r) [rectangle,draw=none] at (0.7,-0.5) {\resizebox{!}{0.85cm}{\begin{tikzpicture}
\node(k) at (-0.5,0) {};
\node (b) [circle,fill=cyan,draw=black,minimum size=0.1cm,inner sep=0.2mm,label={[xshift=-0.65cm,yshift=-0.35cm]{\footnotesize $X\cup Y$}}] at (0,0) {};
\node (b1) [circle,fill=cyan,draw=black,minimum size=0.1cm,inner sep=0.2mm,label={[xshift=0.3cm,yshift=-0.35cm]{\footnotesize $V$}}] at (0.4,0) {};
 \node (d1) [circle,fill=cyan,draw=black,minimum size=0.1cm,inner sep=0.2mm,label={[xshift=0.3cm,yshift=-0.33cm]{\footnotesize $U$}}] at (0.4,0.35) {};\draw[thick] (b1)--(d1);\end{tikzpicture}}};
\node (k1)[rectangle,draw=none] at (2,-2.4) {\resizebox{!}{0.85cm}{\begin{tikzpicture}
\node(k) at (-0.5,0) {};
\node (b) [circle,fill=cyan,draw=black,minimum size=0.1cm,inner sep=0.2mm,label={[xshift=-0.3cm,yshift=-0.35cm]{\footnotesize $Y$}}] at (0,0) {};
 \node (d) [circle,fill=cyan,draw=black,minimum size=0.1cm,inner sep=0.2mm,label={[xshift=-0.3cm,yshift=-0.33cm]{\footnotesize $X$}}] at (0,0.35) {};
\node (b1) [circle,fill=cyan,draw=black,minimum size=0.1cm,inner sep=0.2mm,label={[xshift=0.3cm,yshift=-0.35cm]{\footnotesize $V$}}] at (0.4,0) {};
 \node (d1) [circle,fill=cyan,draw=black,minimum size=0.1cm,inner sep=0.2mm,label={[xshift=0.3cm,yshift=-0.33cm]{\footnotesize $U$}}] at (0.4,0.35) {};
 \draw[thick] (b)--(d);\draw[thick] (b1)--(d1);\end{tikzpicture}}};
  \node (jo) at  (2,1.15)  {\resizebox{!}{0.5cm}{\begin{tikzpicture}
\node(k) at (-0.5,0) {};
\node (b) [circle,fill=cyan,draw=black,minimum size=0.1cm,inner sep=0.2mm,label={[xshift=-0.65cm,yshift=-0.35cm]{\footnotesize $X\cup Y$}}] at (0,0) {};
\node (b1) [circle,fill=cyan,draw=black,minimum size=0.1cm,inner sep=0.2mm,label={[xshift=0.65cm,yshift=-0.35cm]{\footnotesize $U\cup V$}}] at (0.4,0) {};
\end{tikzpicture}}};
\node (mq1) [rectangle,draw=none] at (3.3,-0.5) {\resizebox{!}{0.85cm}{\begin{tikzpicture}
\node(k) at (-0.5,0) {};
\node (b) [circle,fill=cyan,draw=black,minimum size=0.1cm,inner sep=0.2mm,label={[xshift=-0.3cm,yshift=-0.35cm]{\footnotesize $Y$}}] at (0,0) {};
 \node (d) [circle,fill=cyan,draw=black,minimum size=0.1cm,inner sep=0.2mm,label={[xshift=-0.3cm,yshift=-0.33cm]{\footnotesize $X$}}] at (0,0.35) {};
\node (b1) [circle,fill=cyan,draw=black,minimum size=0.1cm,inner sep=0.2mm,label={[xshift=0.65cm,yshift=-0.35cm]{\footnotesize $U\cup V$}}] at (0.4,0) {};
 \draw[thick] (b)--(d);\end{tikzpicture}}};
\draw[left hook-latex] (k1)--(r);
\draw[right hook-latex] (k1)--(mq1);
\draw[right hook-latex](mq1)--(jo);
\draw[left hook-latex] (r)--(jo);
\end{tikzpicture}} \quad\quad\quad
\raisebox{0.145cm}{\begin{tikzpicture}
\node (r) [rectangle,draw=none] at (0.8,-0.5) {\resizebox{!}{0.85cm}{\begin{tikzpicture}
\node(k) at (0.6,0) {};
\node (b) [circle,fill=cyan,draw=black,minimum size=0.1cm,inner sep=0.2mm,label={[xshift=0.65cm,yshift=-0.35cm]{\footnotesize $X\cup Y$}}] at (0,0.35) {};
 \node (d) [circle,fill=cyan,draw=black,minimum size=0.1cm,inner sep=0.2mm,label={[xshift=0.3cm,yshift=-0.33cm]{\footnotesize $V$}}] at (0,0) {};
 \draw[thick] (b)--(d);\end{tikzpicture}}};
\node (k1)[rectangle,draw=none] at (2,-2.475) {\resizebox{1.45cm}{!}{\begin{tikzpicture}
\node (b) [circle,fill=cyan,draw=black,minimum size=0.1cm,inner sep=0.2mm,label={[yshift=0cm]{\footnotesize $X$}}] at (4,0.15) {};
\node (c) [circle,fill=cyan,draw=black,minimum size=0.1cm,inner sep=0.2mm,label={[xshift=-0.3cm,yshift=-0.35cm]{\footnotesize $Y$}}] at (3.8,-0.2) {};
 \node (d) [circle,fill=cyan,draw=black,minimum size=0.1cm,inner sep=0.2mm,label={[xshift=0.3cm,yshift=-0.35cm]{\footnotesize $V$}}] at (4.2,-0.2) {};
 \draw[thick]  (c)--(b)--(d);\end{tikzpicture}}};
  \node (jo) at  (2,1.15)  {\resizebox{2.65cm}{!}{\begin{tikzpicture}
\node (aa) at (-0.75,-0.1) {};
 \node (d) [circle,fill=cyan,draw=black,minimum size=0.1cm,inner sep=0.2mm,label={[xshift=0.95cm,yshift=-0.35cm]{\footnotesize $X\cup Y\cup V$}}] at (0,0) {};\end{tikzpicture}}};
\node (mq1) [rectangle,draw=none] at (3.2,-0.5) {\resizebox{!}{0.85cm}{\begin{tikzpicture}
\node(k) at (-0.6,0) {};
\node (b) [circle,fill=cyan,draw=black,minimum size=0.1cm,inner sep=0.2mm,label={[xshift=-0.65cm,yshift=-0.35cm]{\footnotesize $X\cup V$}}] at (0,0.35) {};
 \node (d) [circle,fill=cyan,draw=black,minimum size=0.1cm,inner sep=0.2mm,label={[xshift=-0.3cm,yshift=-0.33cm]{\footnotesize $Y$}}] at (0,0) {};
 \draw[thick] (b)--(d);\end{tikzpicture}}};
\draw[left hook-latex] (k1)--(r);
\draw[right hook-latex] (k1)--(mq1);
\draw[right hook-latex](mq1)--(jo);
\draw[left hook-latex] (r)--(jo);
\end{tikzpicture}}
\quad\quad\quad
\begin{tikzpicture}
\node (r) [rectangle,draw=none] at (0.8,-0.5) {\resizebox{!}{0.85cm}{\begin{tikzpicture}
\node(k) at (0.6,0) {};
\node (b) [circle,fill=cyan,draw=black,minimum size=0.1cm,inner sep=0.2mm,label={[xshift=0.65cm,yshift=-0.35cm]{\footnotesize $X\cup Y$}}] at (0,0.35) {};
 \node (d) [circle,fill=cyan,draw=black,minimum size=0.1cm,inner sep=0.2mm,label={[xshift=0.3cm,yshift=-0.33cm]{\footnotesize $V$}}] at (0,0) {};
 \draw[thick] (b)--(d);\end{tikzpicture}}};
\node (k1)[rectangle,draw=none] at (2,-2.55) {\resizebox{!}{1.2cm}{\begin{tikzpicture}
\node(k) at (-0.5,0) {};
\node (b) [circle,fill=cyan,draw=black,minimum size=0.1cm,inner sep=0.2mm,label={[xshift=0.3cm,yshift=-0.35cm]{\footnotesize $V$}}] at (0,0) {};
\node (c) [circle,fill=cyan,draw=black,minimum size=0.1cm,inner sep=0.2mm,label={[xshift=0.3cm,yshift=-0.29cm]{\footnotesize $X$}}] at (0,0.7) {};
 \node (d) [circle,fill=cyan,draw=black,minimum size=0.1cm,inner sep=0.2mm,label={[xshift=0.3cm,yshift=-0.33cm]{\footnotesize $Y$}}] at (0,0.35) {};
 \draw[thick] (b)--(d)--(c);\end{tikzpicture}}};
  \node (jo) at  (2,1.15)  {\resizebox{2.65cm}{!}{\begin{tikzpicture}
\node (aa) at (-0.75,-0.1) {};
 \node (d) [circle,fill=cyan,draw=black,minimum size=0.1cm,inner sep=0.2mm,label={[xshift=0.95cm,yshift=-0.35cm]{\footnotesize $X\cup Y\cup V$}}] at (0,0) {};\end{tikzpicture}}};
\node (mq1) [rectangle,draw=none] at (3.2,-0.5) {\resizebox{!}{0.85cm}{\begin{tikzpicture}
\node(k) at (-0.6,0) {};
\node (b) [circle,fill=cyan,draw=black,minimum size=0.1cm,inner sep=0.2mm,label={[xshift=-0.3cm,yshift=-0.35cm]{\footnotesize $X$}}] at (0,0.35) {};
 \node (d) [circle,fill=cyan,draw=black,minimum size=0.1cm,inner sep=0.2mm,label={[xshift=-0.65cm,yshift=-0.33cm]{\footnotesize $Y\cup V$}}] at (0,0) {};
 \draw[thick] (b)--(d);\end{tikzpicture}}};
\draw[left hook-latex] (k1)--(r);
\draw[right hook-latex] (k1)--(mq1);
\draw[right hook-latex](mq1)--(jo);
\draw[left hook-latex] (r)--(jo);
\end{tikzpicture}
\end{center}
where we only display the edges involved in the construction. By
definition, the construct $D$ satisfies the required properties.
\end{proof}

\subsection{Proof of Theorem~\ref{Woy-Woy}.}
\label{Posledni tyden.}
We establish first that $\minGrc$ is acyclic in positive
dimensions and that $H_0(\minGrc) \cong \bbk$.
By Proposition~\ref{jovanica}, each construct 
$C : \H_\Gamma$ is, for $\Gamma \in \Grc$ with at
least one internal edge, 
of the form $\alpha_\Gamma(T)$ for some graph-tree $T \in
\gTr(\Gamma)$. It is therefore
supported by a rooted planar 
tree, so we may introduce the lexicographic arrangement $\varsigma_{\,\lex}$ of
levels of its underlying tree. 
Consequently we get from~(\ref{Free constructs.}) an analog
\begin{equation*}
\Free(E)(\GAmma) \cong \bigoplus_{C : \H_\GAmma} \ 
E(C,\varsigma_{\,\lex})
\end{equation*}
of formula~(\ref{Prvni_na_Vivat_tour.}).

The case which interests us is when 
$E$ is the collection $D$ in~(\ref{Flicek}) generating $\minGrc$.
A~vertex $v$ of $C$ is decorated by a subset $X_v \subset \Edg(\Gamma)$,
thus it contributes to $D(C,\varsigma_{\,\lex})$ by the multiplicative 
factor $\det(X_v)$. 
Let us fix an order of $\Edg(\Gamma)$. Then each
$X_v$ bears an induced order, hence $\det(X_v)$
has a preferred basis element
\[
x_1\land \cdots\land x_r \in \det(X_v),\ x_1< \cdots< x_r,\ X_v = \{\Rada x1r\},
\]
so it is canonically isomorphic to $\bfk$
placed, according to our conventions, 
in degree $|X_v| - 1$. Combining the above facts, we
arrive at the canonical isomorphism
\begin{equation}
\label{Flicek_a_Misa_s_motylkem}
\Free(D)(\GAmma) \cong \bigoplus_{C : \H_\GAmma} \Span(\{e_C\}), 
\end{equation}
where $\Span(\{e_C\})$ is the vector space spanned by a generator $e_C$
placed in degree that equals the rank of $C$, which in this case
equals $|\Edg(\Gamma)|-|\Vert(C)|$.

The differential $\pa$ of the minimal model transfers, via
isomorphism~(\ref{Flicek_a_Misa_s_motylkem}), into a differential
denoted by the same symbol of the graded vector space at the right hand side
of~(\ref{Flicek_a_Misa_s_motylkem}). 
It is straightforward to verify that the transferred differential has
the form required by 
Lemma~\ref{Jeste_ani_nevim_kde_budu_v_Melbourne_bydlet.},~i.e.\ 
\begin{equation}
\label{Uz_jedou_fukary.}
\pa(e_C) = \sum \eta^F_C \cdot e_F,
\end{equation}
where $\eta_F \in \{-1,+1\}$ and $F$ runs over all $F : \H_\Gamma$ 
such that $\grad(F) = \grad(C)-1$.

\begin{remark}
It is possible to establish the explicit values of the coefficients
$\eta_C^F$ in~(\ref{Uz_jedou_fukary.}), but the ingenuity of
  Lemma~\ref{Jeste_ani_nevim_kde_budu_v_Melbourne_bydlet.} makes it
  unnecessary.  
\end{remark}

Now we invoke that the 
poset  ${\mathcal A}({\H_\Gamma})$ 
of constructs  of $\H_\Gamma$ is, by Lemma~\ref{abspol}, the
poset of faces of a convex polytope $K$
which moreover fulfills the diamond
property by Lemma~\ref{Panenka_na_mne_kouka.}.
By Lemma~\ref{Jeste_ani_nevim_kde_budu_v_Melbourne_bydlet.}, the cells
of $K$ can be oriented so that 
\[
\left( \bigoplus_{C : \H_\GAmma} 
\Span(\{e_C\}),\pa \right)
\]
is the cell complex $C_*(K)$. It is thus
acyclic in positive dimension, and so is $(\Free(D)(\GAmma),\pa) = \minGrc(\Gamma)$, for each
$\Gamma\in \Grc$.
By the same reasoning, 
\begin{equation}
\label{Lisegacev1}
H_0(\minGrc)(\Gamma) \cong \bbk\ \hbox { for each $\Gamma \in \Grc$.}
\end{equation}

The next step is to prove that the operad morphism $\rho : \Free(D)
\to \termGrc$ commutes with the differentials, which clearly amounts to proving
that $\rho(\pa x) = 0$ for each degree $1$ element $\mu\in
\Free(D)(\Gamma)_1$. By the derivation property of $\pa$, it is in
fact enough to address only the case when $\mu$ is a generator of degree $1$,
i.e.\ an element of $D(\Gamma) = \det(\Edg(\Gamma))$ with $\Gamma$
having exactly two internal edges.

Let thus $\Gamma$ be such a graph and $a,b$ its two internal
vertices. There are precisely two graph-trees $T',T''\in
\gTr^2(\Gamma)$, both with two vertices and one internal edge.
The root vertex of $T'$ is decorated by some graph $\Gamma_v'$ with the only
internal edge $a$, and the other vertex of $T'$ by $\Gamma_u'$ with
the only internal edge $b$. The graph-tree $T''$ has similar decorations
$\Gamma_v''$ and $\Gamma_u''$, but this time $\Edg(\Gamma_v'') =
\{b\}$ and $\Edg(\Gamma_u'') = \{a\}$.
For a generator 
$\mu := a \land b \in D(\Gamma) = \det(\{a,b\})$ 
formula~(\ref{Flicek_na_mne_kouka.}) gives
\[
  \pa( a \land b) = a \ot b - b \ot a \in 
(D(\Gamma'_v) \ot D(\Gamma'_u)) \oplus( D(\Gamma''_v) \ot D(\Gamma''_u)) 
\subset \Free^2(D)(\Gamma).
\]
By the definition~(\ref{Obehnu_to_dnes?}) of the morphism $\rho$,
\[
\rho(\pa( a \land b)) = \rho( a \ot b - b \ot a) = 1\cdot 1 - 1\cdot 1 =0
\]
as required.

The last issue that has to be established is that $\rho$ induces an
isomorphism
\[
H_0(\rho) : H_0(\minGrc) \stackrel\cong\longrightarrow \termGrc.
\]
To this end, in view of~(\ref{Lisegacev1}), it is enough to prove that
\[
  H_0(\rho)(\Gamma) : H_0(\minGrc)(\Gamma)
  \longrightarrow \termGrc(\Gamma) = \bbk 
\] 
is nonzero for each $\Gamma\in \Grc$.
Equation~(\ref{Athalia}) readily gives 
\begin{equation}
\label{Jeptha}
\Free(D)(\GAmma)_0 \cong 
\bigoplus_{T \in \gTr_0(\GAmma)} \  \colim{\lambda \in
  \Lev(T)} D(T,\lambda),
\end{equation}
in which $\gTr_0(\GAmma)$ is the subset of
$\gTr(\GAmma)$ consisting of graph-trees for which each decorating graph
$\Gamma_v$, $v \in \Vert(\Gamma)$, has exactly one internal edge. For
such a graph, $D(\Gamma_v) = \det(\Edg(\Gamma_v))$ is canonically
isomorphic to $\bfk$ placed in degree $0$.
The groupoid $\Lev(T)$ therefore acts trivially on $D(T,\lambda)$ which
is canonically isomorphic to $\bfk$, so~(\ref{Jeptha}) leads to
\begin{equation}
\label{Lisegacev}
\Free(D)(\GAmma)_0 \cong \Span(\gTr_0(\GAmma)),
\end{equation}
in which 
each $T \in \gTr_0(\GAmma)$ corresponds to a vertex of the
polytope $K$ associated to ${\mathcal A}(\H_\Gamma)$ and therefore
represents a cycle that linearly generates $H_0(\minGrc)$.
We will show that \hbox{$\rho(T) \not= 0$}. 

Under isomorphism~(\ref{Lisegacev}), each
$T$ is an operadic composition of graph trees in $\gTr_0^1(\GAmma)$,
i.e.~graph trees whose underlying tree has one vertex which is decorated by a
graph with one internal edge. By~(\ref{Obehnu_to_dnes?}), $\rho(S) = 1
\in \bbk$ for  $S\in \gTr_0^1(\GAmma)$. Since all operadic compositions
in $\termGrc$ are the identities  $\id : \bbk \ot
\bbk \to \bbk$, $\rho(T) = 1$ for the composite $T$ as well.
This finishes the proof of Theorem~\ref{Woy-Woy}.

\section{Other cases}
\label{Mourek a Terezka}

\begin{figure}
\[
\xymatrix@R=1em{&\RTr\ar[dr] &&& \ar[d]\ggGrc&&
\\
\PRTr   \ar[ur]\ar[dr] &&\ar[rr] \Tr && \ar[rr]\Grc && \Gr
\\
&\PTr\ar[ur] &&&\ar[u]\Whe&\Dio\ar[l]&\, \Dio_3 \ar@{_{(}->}[l]&\hGr\ar[l]
}
\]
\caption{\label{V utery jdu na endo.}%
  Relations between graph-related operadic categories,
  see Section~4 of~\cite{SydneyI} for the notation. All arrows are
  discrete operadic fibrations except for $\ggGrc \to \Grc$ and
  $\Whe\to \Grc$ which are discrete operadic opfibrations.}
\end{figure}
As the diagram in Figure~\ref{V utery jdu na endo.} taken 
from~\cite{SydneyI} teaches us, many operadic
categories of interest are obtained from the basic category $\Grc$ of
ordered connected graphs by iterated discrete operadic fibrations or
opfibrations. This is in particular true for the category $\ggGrc$ of
genus-graded graphs, the category $\Tr$ of trees, 
and the category $\Whe$ of wheeled graphs; they all are
discrete operadic opfibrations over $\Grc$. 
Moreover, the inclusion $\RTr \hookrightarrow \Grc$ of the
operadic category of rooted trees is a discrete operadic fibration
with finite fibers.   
Corollary~\ref{grantova_zprava} of~Subsection~\ref{Dnes_hori_bus.}
below states that the
restrictions along discrete operadic opfibrations  or fibrations with
finite fibers preserve minimal
models of the terminal operads. 
Therefore the minimal models of the terminal operads in the above
mentioned categories are suitable restrictions of the minimal model
$\minGrc$ of the terminal $\Grc$-operad constructed in 
Section~\ref{hadrova_panenka}. We close this section by describing
the minimal model of the terminal operad in the category $\SRTr$ of
strongly rooted~trees.

\subsection{Operadic (op)fibrations and minimal models}
\label{Dnes_hori_bus.} 
The following material uses the terminology of~\cite{SydneyI,SydneyII}.
All operadic categories in this subsection will be
factorizable, graded, and such that  all quasibijections are invertible,
the blow up and unique fiber axioms are fulfilled, and a~morphism is an
isomorphisms if it is of grade $0$. These assumptions are fulfilled by
all operadic categories discussed in the present paper.

Assume that $\ttO$ is such an operadic category. As argued
in~\cite[Section~3]{SydneyII}, one has the natural forgetful functor 
$\zap_\ttO : \OperV\ttO \to\CollectV\ttO$  from the
category of $1$-connected strictly
unital Markl's $\ttO$-operads with values in a symmetric monoidal
category $\ttV$ to the category of $1$-connected
$\ttO$-collections in $\ttV$. Its left adjoint $\Free_\ttO : \CollectV\ttO \to
\OperV\ttO$ is the free operad functor.

Each strict operadic functor $p : \ttO \to \ttP$ induces the
restriction $p^*: \OperV\ttP \to \OperV\ttO$ acting on objects by
the formula
\begin{equation}
\label{Zitra mam treti prednasku.}
p^*(\oP)(t) := \oP(p(t)),  \ \oP \in  \OperV\ttP , \ t \in \ttO.
\end{equation}
The restriction $p^*$ may or may not have a right adjoint
$p_*: \OperV\ttO \to \OperV\ttP$ and even if if it exists
its form  may not be simple unless $p$ has some special properties.

Recall the following general categorical definition. 
Assume we are given a commutative diagram  of right adjoints
\begin{equation}
\label{lrbc}
\xymatrix@C=3em@R=1.5em{\ttA
 \ar[dd]^{u^*} 
&  & 
\ar[dd]_{v^*}\ttB \ar[ll]_{p^*}
\\
 & &
\\
\ttC
 &  & \ttD
\ar[ll]_{q^*}
}
\end{equation}
in which $p^*$ and $q^*$ are also left adjoints.
These functors can be organized into the following diagram of adjunctions
\begin{equation}
\label{vseadj}
\xymatrix@C=4em{\ttA\ar@/_1.5em/[rr]_{p_*} \ar@/^1.5em/[rr]^{p_!}\ar@/^1em/[dd]^{u^*} &  & 
\ar@/_1em/[dd]_{v^*}\ttB \ar[ll]_{p^* \hspace{0.5mm} \perp}^{\perp} \\
\dashv & &\vdash
\\
\ar@/^1em/[uu]^{u_!}
\ttC\ar@/_1.5em/[rr]_{q_*} \ar@/^1.5em/[rr]^{q_!}& & \ttD
\ar[ll]_{q^* \hspace{0.5mm} \perp}^{\perp}
\ar@/_1em/[uu]_{v_!}
}
\end{equation}
The square (\ref{lrbc})
is called {\it  right Beck-Chevalley square} if the following composite $$u_! q^*  \to u_! q^* v^* v_! =
u_! u^* p^* v_! \to p^* v_!$$ is an isomorphism.
Symmetrically, (\ref{lrbc})
is {\it a left Beck-Chevalley square}  if the composite
$$ q_!u^*  \to  q_! u^*p^* p_! =
q_! q^* v^* p_! \to  v^*p_!$$
is an isomorphism, cf.~\cite{Malt}.

\begin{lemma}
\label{left and right BC}
The  following two conditions are equivalent:
\begin{itemize}
\item [(i)]
 the mate $q_* u^* \leftarrow q_* u^* p^*p_* 
= q_* q^* v^* p_* \leftarrow 
v^* p_*$ is an isomorphism and
\item [(ii)]
the square (\ref{lrbc})
is a right Beck-Chevalley square. 
\end{itemize}
If $p_!$ is also a right adjoint to $p^*$ (that is, $p_!\cong p_*$) and
$q_!$ is a right adjoint to $q^*$ then (\ref{lrbc}) is a right
Beck-Chevalley square if and only if it is a left Beck-Chevalley
square.
\end{lemma}

\begin{proof}
Condition (i) just says that the right adjoints commute up to
isomorphism. It follows that the left adjoints commute up to
isomorphism as well, which is the right Beck-Chevalley condition
(ii). The converse is clearly true as well.

If $p_!$ is also a right adjoint to $p^*$ and $q_!$ is a right adjoint
to $q^*$ then obviously the left Beck-Chevalley condition is again about
commutation of right adjoints, hence their left adjoints commute and 
the right Beck-Chevalley condition holds. The inverse
implication is similar.
\end{proof}

\begin{remark} 
It was pointed to us by our anonymous referee that in
  \cite[Lemma 7.10]{ward0} an analogue of our Lemma \ref{left and
    right BC} is given under the so called `Wirthm\"uller context' for
  the six operations formalism (the existence of $p_*$ is a sufficient
  condition). The referee also asked which morphisms
  between operadic categories may induce the 'Grothendieck
  context.' The existence of such a context would provide an
  alternative condition for the preservation of minimal models by the
  restriction functor $p^*.$ We do not have an immediate answer but we
  are grateful to our referee for raising this interesting question,
  which certainly deserves further study.
\end{remark}

In the following proposition,  \hbox{$p^*: \OperV\ttP \to \OperV\ttO$} is the
restriction functor defined by~(\ref{Zitra mam treti prednasku.}) and
 $p_0^*: \CollectV\ttP \to \CollectV\ttO$ is the obvious similar
 restriction between the categories of collections.

\begin{proposition} 
\label{Koronavirus se siri.}
The square
\begin{equation}
\label{rightBC}
\xymatrix{\OperV\ttO
 \ar[dd]^{\zap_\ttO} 
&  & 
\ar[dd]_{\zap_\ttP}\OperV\ttP \ar[ll]_{p^*}
\\
 & &
\\
\CollectV\ttO
 &  & \CollectV\ttP
\ar[ll]_{p_0^*}
}
\end{equation}
is a right Beck-Chevalley square provided any of the two following
conditions hold: 
\begin{enumerate} 
\item[(i)]
$p$ is a discrete operadic opfibration
and $\ttV$ a cocomplete symmetric monoidal category; 
\item[(ii)] 
$p$ is
a discrete operadic fibration with finite fibers and $\ttV$ an
additive cocomplete symmetric monoidal category. 
\end{enumerate}
\end{proposition} 

\begin{proof}   
The  right adjoint $(p_0)_*: \CollectV\ttO \to
\CollectV\ttP$ to the restriction   
$p_0^*: \CollectV\ttP \to \CollectV\ttO$ is given on objects by
\begin{subequations}
\begin{equation}
\label{Jolom}
(p_0)_*(E)(T) := \prod_{p(t) = T} E(t), \ E \in  \CollectV{\ttO}, \
T \in \ttP.
\end{equation}
Assume that $p : \ttO \to \ttP$ is a~discrete operadic opfibration.
By dualizing~\cite[Theorem~2.4]{duodel} one verifies that the right
adjoint \hbox{$p_* : \Oper{\ttO} \to \Oper{\tt P}$} is defined on objects by
\begin{equation}
\label{dnes_moje_posledni_prednaska_v_MSRI}
p_*(\oO)(T) := \prod_{p(t) = T} \oO(t), \ \oO \in  \Oper{\ttO} \
T \in \tt P.
\end{equation}
\end{subequations}
Comparing~(\ref{Jolom})
with~(\ref{dnes_moje_posledni_prednaska_v_MSRI}) we see that
$(p_0)_*\, \zap_\ttO = \zap_\ttP p_*$,
which is condition~(i) of 
Lemma~\ref{left and right BC}. Thus~(\ref{rightBC}) is right Beck-Chevalley by the same
lemma. This finishes the proof of the case of a discrete opfibration.

Let us assume  that
$p : \ttO \to \ttP$ is a~discrete operadic fibration with finite
fibers. We want to verify the assumptions of the second part of 
Lemma~\ref{left and right BC}, i.e.\ to check that
$(p_0)_!$ is a right adjoint to $p_0^*$ and that
$p_!$ is a right adjoint to $p^*$. 

It is clear that  $(p_0)_!$ is for an arbitrary $p: \ttO \to \ttP$ 
given on objects by the formula 
\begin{equation*}
(p_0)_!(E)(T) := \bigoplus_{p(t) = T} E(t), \ E \in  \CollectV{\ttO},\
T \in \ttP.
\end{equation*}
Since $V$ is additive and $p$ has finite fibers, this functor coincides
with the right adjoint $(p_0)_*$ described in~(\ref{Jolom}).
On the other hand,  
\cite[Theorem~2.4]{duodel}  gives the following formula for the 
underlying collection of $p_!(\oO)$:
\begin{equation*}
p_!(\oO)(T) := \bigoplus_{p(t) = T} \oO(t), \ \oO \in  \OperV{\ttO}, \
T \in \ttP.
\end{equation*}
It is not hard to see, using the additivity of $V$ and the finiteness 
of the fibers of $p$,
that this formula describes also a right adjoint to $p^*$, which completes
the proof for operadic fibrations.
\end{proof}

In the rest of this section, the coefficient category $\ttV$ will be
that of differential graded vector spaces. It clearly satisfies all
assumptions required in Proposition~\ref{Koronavirus se siri.}. 

\begin{proposition}
Assume that (\ref{rightBC}) is a right Beck-Chevalley square and
$\rho : \minP \to \termP$ is the minimal model of the terminal
$\ttP$-operad $\termP$.~Then
\[
\xymatrix@1{
\minO := p^*(\minP)\ \ar[r]^(.55){p^*(\rho)}& \ p^*(\termP) = \termO
}
\]
is the minimal model of the terminal $\ttO$-operad $\termO$.
\end{proposition}

\begin{proof}
It is clear that $p^*(\termP) = \termO$. Let $\minP =
(\Free_\ttP(E_\ttP),\pa_\ttP)$. 
Diagram~(\ref{rightBC}) is, by definition, a~right Beck-Chevalley
square if $p^*\, \Free_\ttP \cong \Free_\ttO\, p_0^*$. In particular,
\[
p^*(\Free_\ttP(E_\ttP)) \cong \Free_\ttO(p_0^*(E_\ttP)),
\]
thus $p^*(\minP)$ is the free operad generated by the collection
$E_\ttO := p_0^*(E_\ttP)$. It is easy to verify that $p^*$ brings
derivations to derivations and differentials to differentials. We
therefore conclude that
\[
p^*(\minP) \cong (\Free_\ttO(E_\ttO),\pa_\ttO),
\]
where the minimality of $\pa_\ttO$ can also be established easily. 

It remains to prove that $p^*(\rho)$ induces a component-wise
isomorphism of homology. This however follows immediately from the
definition of the restriction functor requiring that
\[
p^*(\rho)(t) = \rho(p(t)) : \minP(p(t)) \to \termP(p(t)) = \bbk, \ t \in \ttO,
\]
where $\rho(p(t))$ is a homology isomorphism since $\rho : \minP \to
\termP$ is the minimal model of $\termP$ by assumption. 
\end{proof}

\begin{corollary}
\label{grantova_zprava}
Let $p: \ttO \to \ttP$ be either a discrete operadic opfibration, or a
discrete operadic fibration with finite fibers, and
$\rho : \minP \to \termP$ the minimal model of the terminal
$\ttP$-operad.~Then
\[
\xymatrix@1{
\minO := p^*(\minP)\ \ar[r]^(.55){p^*(\rho)}& \ p^*(\termP) = \termO
}
\]
is the minimal model of the terminal $\ttO$-operad.
\end{corollary}

\begin{remark} 
The assumptions and conclusion of 
Corollary~\ref{grantova_zprava} were verified in
the context of operadic categories related to permutads
in~\cite{perm}. 
\end{remark}

\subsection{Minimal model for $\termggGrc$.}
\label{Podari se mi koupit to auto?} 
The operadic category $\ggGrc$ consists of 
graphs $\Gamma \in \Grc$ equipped with
a {\em genus grading\/}, which is a non-negative integer $g(v) \in \bbN$ specified
for each $v \in \Vert(\Gamma)$. The genus of the entire graph $\Gamma$
is defined by
\[
g(\Gamma):=
\sum_{v \in \Vert(\Gamma)} g(v) + \dim(H^1(|\Gamma|; {\mathbb Z})),
\]
where $|\Gamma|$ is the obvious geometric realization of $\Gamma$. As shown
in~\cite[Section~5]{SydneyII}, algebras for $\termggGrc$ are
modular operads introduced in~\cite{GK98}.

Assume that $\Gamma \in \ggGrc$ and that $T \in \Tr(\Gamma)$ is a
graph-tree. Then there exists a unique genus grading of each of the
graphs $\Gamma_v$ decorating the vertices of $T$ subject, along with
the compatibilities required in
Subsection~\ref{Dnes_je_Michalova_oslava.}, also to:

\noindent 
{\em Genus compatibility.} Let $e$ be an internal edge of $T$ pointing
from the vertex labelled by  $\Gamma_u$ to  the vertex labelled by  
$\Gamma_v$. By Compatibility~1, $e$ is also (the label of) a vertex of
$\Gamma_v$. With this convention in mind we require that
\[
g(e) = g(\Gamma_u).
\]
In words, the vertex of $\Gamma_v$ to which $\Gamma_u$
is contracted bears the genus $ g(\Gamma_u)$.

The statement can be verified directly, which we leave as an exercise
to the reader. It can also be established by inductive applications of

\begin{lemma}
\label{Musim psat grantovou zpravu}
Let $\phi : \Gamma \to \Gamma''$ be an elementary
morphism in $\Grc$ with fiber $\Gamma'$, in shorthand
\begin{equation}
\label{V sobotu letim do Melbourne}
\Gamma' \fib \Gamma \stackrel\phi\longrightarrow \Gamma''.
\end{equation}
Assume moreover that $\Gamma$ bears a genus grading. Then there are
unique genus gradings of $\Gamma'$ and $\Gamma''$ such that~(\ref{V
  sobotu letim do Melbourne}) becomes a diagram, 
in $\ggGrc$, of an elementary map and its fiber. 
\end{lemma}

\begin{proof}
A consequence of the fact that the obvious projection 
$p: \ggGrc \to \Grc$ is a discrete operadic opfibration, though it can
also be verified directly.
\end{proof}

For $\Gamma \in \ggGrc$ having at least one internal edge and for a
$1$-connected $\ggGrc$-collection $E$, the right hand side of
\begin{equation}
\label{Za tyden do SaFra}
\Freegg(E)(\GAmma) :=
\bigoplus_{T \in \gTr(\GAmma)} \  \colim{\lambda \in
  \Lev(T)} E(T,\lambda),
\end{equation}
makes sense because, as explained above, 
each of the graphs $\Gamma_i$, $1 \leq i \leq k$,
in~(\ref{Krtek_na_mne_kouka.}) where $E(T,\lambda)$ was defined, 
bears a~unique genus grading induced by
the genus grading of $\Gamma$.

Let $p: \ggGrc \to \Grc$ be as before the canonical projection that
forgets the genus grading, and  
$p^*: \Oper\Grc \to \Oper\ggGrc$  resp.~$p_0^*: \Collect\Grc \to
\Collect\ggGrc$ the induced restrictions. 
The values of the $\ggGrc$-collection $\Dgg \in
\Collect\ggGrc$ given by
\[
\Dgg(\Gamma) := \det(\Gamma), \ \Gamma \in \ggGrc,
\]
do not depend on the genus grading, thus $\Dgg = p_0^*(D)$,
where $D \in \Collect\Grc$ is as in~(\ref{Flicek}). For the same
reasons
\[
\Freegg(\Dgg) = p^* \Free (D),
\]
so, since $p: \ggGrc \to \Grc$ is a discrete operadic opfibration,
$\Freegg(\Dgg)$ defined by~(\ref{Za tyden do SaFra}) with $E = \Dgg$
represents 
the free $\ggGrc$-operad on $\Dgg$ by Proposition~\ref{Koronavirus se siri.}. 
The differential $\pa$ on $\Freegg(\Dgg)$ is given by an obvious
analog of~(\ref{Posledni patek v Sydney}). 

As expected, we define
$\rho : \Freegg(\Dgg) \to \termggGrc$ as the unique map of $\ggGrc$-operads
whose restriction $\rho|_{\Dgg(\Gamma)}$ is, for $\Gamma \in \ggGrc$,
given by a modification of~(\ref{Obehnu_to_dnes?}), namely by
\[
\rho|_{D(\Gamma)} := 
\begin{cases}
\id_\bbk : D(\Gamma) = \bbk \to \bbk =
\termggGrc(\Gamma), &\hbox {if $|\Edg(\Gamma)| = 1$,
  while}
\\
0,& \hbox {if  $|\Edg(\Gamma)| \geq 2$.}  
\end{cases}
\]

\begin{theorem}
\label{neco na jazyku}
The object
$\minggGrc = (\Freegg(\Dgg),\pa) \stackrel\rho\longrightarrow
(\termggGrc,\pa =0)$ is a minimal model of the terminal
$\ggGrc$-operad $\termggGrc$.
\end{theorem}

\begin{proof}
A consequence of Corollary~\ref{grantova_zprava}, though the
acyclicity of $\minggGrc$ in positive dimensions follows directly from the
acyclicity of  $\minGrc$ proven in Subsection~\ref{Posledni
  tyden.}, thanks to the isomorphism
\[
\minggGrc(\Gamma) \cong \minGrc(\widehat \Gamma),\ \Gamma \in \ggGrc,
\]
of dg vector spaces, where  $\widehat \Gamma \in
\Grc$ is $\Gamma$ stripped of the genus grading. 
\end{proof}

\subsection{Minimal model for $\termTr$}
\label{Dnes prvni vylet na kole}
Let $\Tr \subset \Grc$ be the full subcategory of contractible, 
i.e.~simply connected graphs. Algebras over the terminal $\Tr$-operad
$\termTr$ are cyclic operads. 
Although it was not stated in~\cite{SydneyI}, the inclusion  $p: \Tr
\hookrightarrow \Grc$ is a discrete operadic opfibration as well, 
we thus are still in the comfortable situation of Subsection~\ref{Dnes_hori_bus.}.
Also an analog of Lemma~\ref{Musim psat grantovou zpravu} is obvious: if
$\Gamma \in \Grc$ is contractible, then $\Gamma'$, as a connected subgraph of
$\Gamma$, is contractible too, and so is the quotient $\Gamma''$.
The minimal model for $\termTr$ can therefore be constructed by mimicking the
methods of Subsection~\ref{Podari se mi koupit to auto?}, so we will
be telegraphic.

For a graph $\Gamma \in \Tr$ having at least one internal edge and a
$1$-connected $\Tr$-collection $E$, the expression in the right hand
side of
\begin{equation}
\label{Jaruska jede za M1.}
\FreeTr(E)(\GAmma) :=
\bigoplus_{T \in \gTr(\GAmma)} \  \colim{\lambda \in
  \Lev(T)} E(T,\lambda)
\end{equation}
makes sense, since each of the graphs $\Rada \Gamma1k$
in the definition~(\ref{Krtek_na_mne_kouka.}) of $E(T,\lambda)$ is connected.
Let $\DTr \in \Collect\Tr$ be the collection with
\[
\DTr(\Gamma) := \det(\Gamma), \ \Gamma \in \Tr.
\]
For $\DTr$ in place of $E$, formula~(\ref{Jaruska jede za M1.})
describes the pieces of the free operad $\FreeTr(\DTr)$.
The differential $\pa$ on $\FreeTr(\DTr)$ is given by an obvious modification of
formula~(\ref{Posledni patek v Sydney}).
Also the definition of 
$\rho : \FreeTr(\DTr) \to \termTr$ is the expected one. We have

\begin{theorem}
\label{Druhou panenku jsem nechal v Praze.}
The object $\minTr = (\FreeTr(\DTr),\pa) \stackrel\rho\longrightarrow (\termTr,\pa
=0)$ is a minimal model of the terminal $\Tr$-operad $\termTr$.
\end{theorem}

\begin{proof}
Verbatim modification of the proof of Theorem~\ref{neco na jazyku}.
\end{proof}

\subsection{Minimal model for $\termWhe$.}
\label{Poslu Jarce obrazky kyticek.}
We say, following~\cite[Example~4.19]{SydneyI}, that 
an ordered connected graph $\Gamma \in \Gr$ is {\em oriented\/} if
\begin{itemize}
\item[(i)]
each internal edge if $\Gamma$ is oriented, 
meaning that one of the half-edges forming this edge is 
marked as the input one, and the other as the output, and
\item[(ii)]
also the legs of $\Gamma$ are marked as either input or output ones.
\end{itemize}
Oriented ordered graphs form an operadic category
$\Whe$. Algebras for the terminal $\Whe$-operad $\termWhe$ are wheeled
properads introduced in~\cite{mms}. As noted in Example~2.19 loc.~cit., the
functor $p : \Whe\to\Grc$ that forgets the orientation 
is a discrete operadic opfibration,
thus the constructions of the previous two subsections, including the
description of the minimal model for $\termWhe$, translate
verbatim. We leave the details to the reader.   

\subsection{Minimal model for $\termRTr$}
\label{Prvni tyden v Berkeley konci.}
We will call the leg of  $\Gamma \in \Tr$, minimal in
the global order, the {\em root\/} of $\Gamma$.
Let us orient  edges of $\Gamma \in \Tr$ so
that they point to the root. We say that $\Gamma$ is
{\em rooted\/} if the outgoing half-edge of each vertex is the smallest
in the local order at that vertex. In~\cite{SydneyI} we considered the 
full subcategory $\RTr$ of $\Tr$ consisting of rooted trees 
and identified algebras over the terminal $\RTr$ operad
$\termRTr$ with ordinary, classical operads. The inclusion
$p: \RTr \hookrightarrow \Tr$  is, however, a discrete operadic
{\em fibration\/}, not an opfibration,
\hbox{cf.~\cite[Example~4.9]{SydneyI}}.
Nevertheless, the fibers of $p$ are  finite, being either empty or
an one-point set, thus Corollary~\ref{grantova_zprava}
applies, so we can construct an explicit minimal model for $\termRTr$ 
by obvious modifications of the methods used in the previous subsections.

\begin{example}
\label{Myska a Tucinek}
Figure~\ref{Vcera jsem si koupil kolo.} illustrates the failure of
Lemma~\ref{Musim psat grantovou zpravu} for $\Tr$ in place of $\Grc$
and $\RTr$ in place of $\ggGrc$. The graph
$\Gamma$  in that figure
has vertices (indexed by) $\{1,2,3\}$ and half-edges
$\{1,2,3,4,5,6\}$, the graph $\Gamma''$  has vertices $\{1,2\}$ and half-edges
$\{1,2,3,4\}$. The map $\phi : \Gamma \to \Gamma''$ 
sends the vertices $1$ and~$3$ of $\Gamma$ to
the vertex $1$ (the fat one) of $\Gamma''$, and the vertex $2$ of
$\Gamma$ to the vertex of $\Gamma''$ with the same label.
The labels in the circles indicate the global orders.
While $\Gamma$ is rooted, $\Gamma''$ is not, 
although $\phi$ is even a canonical
contraction.

\begin{figure}[h] 
\[
\psscalebox{1.0 1.0} 
{
\begin{pspicture}(0,-1.6698403)(7.6223903,1.6698403)
\psline[linecolor=black, linewidth=0.04, arrowsize=0.05291667cm 3.05,arrowlength=2.14,arrowinset=0.0]{<-}(3.8860292,-0.6644454)(3.8860292,0.3355546)
\psdots[linecolor=black, dotsize=0.41680175](4.93,-1.4)
\rput{-90.00662}(4.18,1.6714104){\psdots[linecolor=black, dotsize=0.22148077](3.0792458,-1.4076726)}
\rput{-90.00662}(4.6881533,6.6716647){\psdots[linecolor=black, dotsize=0.22148077](5.6795235,0.9920268)}
\rput{-90.00662}(1.0873207,3.0720809){\psdots[linecolor=black, dotsize=0.22148077](2.0795233,0.99244297)}
\rput{-90.00662}(2.8877368,4.871873){\psdots[linecolor=black, dotsize=0.22148077](3.8795233,0.9922349)}
\psline[linecolor=black, linewidth=0.04, arrowsize=0.05291667cm 3.05,arrowlength=2.14,arrowinset=0.0]{<-}(6.270205,-1.4266953)(1.6392437,-1.42616)
\psline[linecolor=black, linewidth=0.04, arrowsize=0.05291667cm 3.05,arrowlength=2.14,arrowinset=0.0]{<-}(7.075869,0.98936534)(0.63279223,0.9901102)
\psline[linecolor=black, linewidth=0.025](3.85,-1.24)(3.85,-1.5990185)
\psline[linecolor=black, linewidth=0.025](3.0202122,1.2)(3.0201623,0.8172518)
\psline[linecolor=black, linewidth=0.025](4.832328,1.2)(4.832278,0.8170423)
\rput{-90.00662}(-0.7887707,1.3642031){\pscircle[linecolor=black, linewidth=0.025, dimen=outer](0.35,1.1){0.29}}
\rput{-90.00662}(6.259974,8.410503){\pscircle[linecolor=black, linewidth=0.025, dimen=outer](7.4,1.0756266){0.29}}
\rput{-90.00662}(2.6340673,-0.045894984){\pscircle[linecolor=black, linewidth=0.025, dimen=outer](1.3940888,-1.28){0.29}}
\rput{-90.00662}(7.8702774,5.1884995){\pscircle[linecolor=black, linewidth=0.025, dimen=outer](6.6,-1.3404341){0.29}}
\rput[b](7.34,.9){1} \rput[b](0.31460062,.9){2}
\rput[b](6.53,-1.53){1} \rput[b](1.35,-1.55){2}

\rput[b](3.8860292,1.2){1} \rput[b](5.7,1.2){3} \rput[b](2.1,1.2){2}

\rput[b](4.8860292,-1){1} \rput[b](2.75,-1.18){2}

\rput[t](6.3146005,.85){5}
\rput[t](5.1717434,.85){6}
\rput[t](4.3146005,.85){1}
\rput[t](3.4574578,.85){2}
\rput[t](2.6003149,.85){3}
\rput[t](1.3146006,.85){4}

\rput[t](5.4574575,-1.55){2}
\rput[t](4.4574575,-1.55){1}
\rput[t](3.4574578,-1.55){3} 
\rput[t](2.146007,-1.55){4}
\rput[l](4.1,-0){$\phi$}

\rput[r](-.3,1){$\Gamma:$}
\rput[r](.7,-1.4){$\Gamma'':$}
\end{pspicture}
}
\]
\caption{\label{Vcera jsem si koupil kolo.}
Failure of Lemma~\ref{Musim psat grantovou zpravu}: $\Gamma''$ is
not rooted whereas $\Gamma$ is.
}
\end{figure}
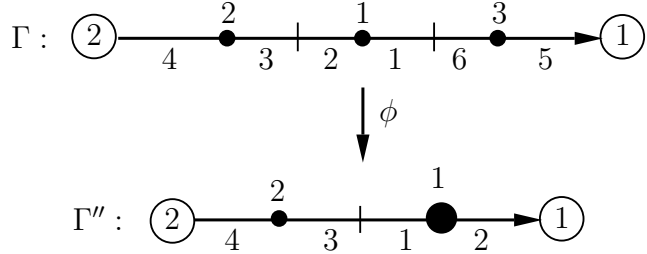
\end{example}

\subsection{Minimal model for $\termSRTr$}
\label{Ze by byla monografie uz konecne dokoncena?}

It turns out that the operadic category $\RTr$ contains much smaller
subcategory which still captures the classical operads in the same way
$\RTr$ does. It is defined as follows.
We say that a rooted tree $\Gamma \in \RTr$ is {\em strongly rooted\/}, if the
order of its set~$V$ of vertices is compatible with the rooted
structure. By this we mean that, if $v\in V$ lies on the path
connecting $u \in V$ with the root, then $v < u$ in
$V$. We denote by $\SRTr \subset \RTr$ the full subcategory of
strongly rooted trees. It is easy to show that all fibers of a map $\phi :
\Gamma' \to \Gamma''$ between strongly rooted trees are strongly rooted, and
also that all rooted corollas are clearly strongly rooted. Consequently, $\SRTr$ is
an operadic category.

We claim that algebras over the terminal
$\SRTr$-operad $\termSRTr$ are the same as $\termRTr$-algebras,
i.e.\ that they are ordinary operads. This might sound surprising,
since  $\SRTr$ has less objects than $\RTr$, therefore 
$\termSRTr$-algebras have less operations than $\termRTr$-algebras.  
Each operation of a $\termRTr$-algebra can however be obtained from an
operation of a $\termSRTr$-algebra via certain permutation of inputs,
since each rooted tree is isomorphic with a strongly rooted
tree, by a~local~isomorphism.

\begin{example}
Consider the rooted trees in Figure~\ref{Zitra se podivam na Golden
  Bridge.}. The left one belongs to $\SRTr$ and represents the
operation
\begin{figure}[h]
  \centering
\[
\psscalebox{1.0 1.0} 
{
\begin{pspicture}(0,-1.8698152)(6.1154737,1.8698152)
\psline[linecolor=black, linewidth=0.04, arrowsize=0.05291667cm 3.05,arrowlength=2.14,arrowinset=0.0]{<-}(1.1890492,2.0569606)(1.1870459,-0.54303545)
\psline[linecolor=black, linewidth=0.04](1.1584744,-0.4287497)(0.015617327,-1.8573211)
\psline[linecolor=black, linewidth=0.04](1.1584744,-0.4287497)(2.1584744,-1.8573211)
\psline[linecolor=black, linewidth=0.04](1.1584744,0.8569646)(0.015617327,-0.28589258)
\psline[linecolor=black, linewidth=0.04](1.1584744,0.8569646)(2.3013315,-0.28589258)
\psdots[linecolor=black, dotsize=0.22148077](1.18054,-0.48614708)
\psdots[linecolor=black, dotsize=0.22148077](1.18054,0.8)
\psline[linecolor=black, linewidth=0.04, arrowsize=0.05291667cm 3.05,arrowlength=2.14,arrowinset=0.0]{<-}(4.9890494,2.0569606)(4.987046,-0.54303545)
\psline[linecolor=black, linewidth=0.04](4.9584746,-0.4287497)(3.8156173,-1.8573211)
\psline[linecolor=black, linewidth=0.04](4.9584746,-0.4287497)(5.9584746,-1.8573211)
\psline[linecolor=black, linewidth=0.04](4.9584746,0.8569646)(3.8156173,-0.28589258)
\psline[linecolor=black, linewidth=0.04](4.9584746,0.8569646)(6.1013317,-0.28589258)
\psdots[linecolor=black, dotsize=0.22148077](4.98054,-0.48614708)
\psdots[linecolor=black, dotsize=0.22148077](4.98054,0.8)
\rput[bl](1.4,1.1426789){1}
\rput[bl](1.4,-0.2){2}
\rput[bl](5.244189,1.085536){2}
\rput[bl](5.2,-0.2){1}
\rput[br](0.7013316,1){$\Gamma'$}
\rput[br](4.6,1){$\Gamma''$}
\end{pspicture}
}
\]
\caption{
Rooted trees $\Gamma' \in \SRTr$ and $\Gamma'' \in \RTr$. Only the labels of
vertices are shown.
\label{Zitra se podivam na Golden Bridge.}
}
\end{figure}
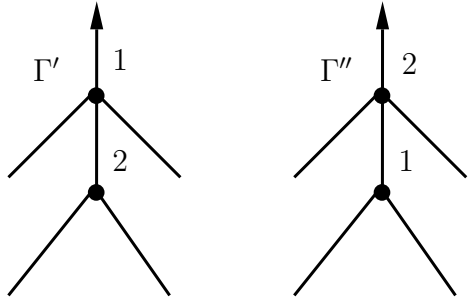
\[
{\mathcal O}_{\Gamma'} : P(3) \ot P(2) \longrightarrow P(4)
\]
given by ${\mathcal O}_{\Gamma'}(x \ot y) = x \circ_2 y$, where $\circ_2$
is the standard $\circ$-operation in a unital operad $P$, while
\[
{\mathcal O}_{\Gamma''} : P(2) \ot P(3) \longrightarrow P(4)
\]
is given by  ${\mathcal O}_{\Gamma''}(a\ot b) = b \circ_2 a$. Thus
${\mathcal O}_{\Gamma''} = {\mathcal O}_{\Gamma'} \circ \sigma$ with
$\sigma \in \Sigma_2$ the transposition.
\end{example}

Neither the inclusion $\SRTr \hookrightarrow \RTr$, nor
the composite  $\SRTr \hookrightarrow \RTr \hookrightarrow  \Tr$ is 
a~fibration or opfibration, but the category $\SRTr$ is, unlike $\RTr$, 
closed under canonical contractions.
It can indeed be easily verified that, if $\Gamma' \in \SRTr$ and if $\pi :
\Gamma' \to \Gamma''$ is the canonical contraction, then~$\Gamma''$ and
also the fiber of $\pi$ belongs to $\SRTr$. The methods developed in
Subsection~\ref{Dnes_je_Michalova_oslava.} can therefore 
be used with $\SRTr$ in place of $\Grc$. 
Namely, each tower~(\ref{t1}) in $\SRTr$ can be brought into the
canonical form where  $\ell = \id_\GAmma$
and all $\tau$'s are canonical contractions, and as such be
represented by a graph tree in $\gTr(\GAmma)$. The right hand side of 
formula~(\ref{Zbijecky duni.}) then, for $\Gamma \in \SRTr$ and $E \in
\Collect\SRTr$, expresses
the component of the free $\SRTr$-operad $\FreeSRTr(E)$.

Our description of a minimal model for $\termSRTr$ is the expected
one. We define the collection
$\DSRTr \in \Collect\SRTr$ by
\[
\DSRTr(\Gamma) := \det(\Gamma), \ \Gamma \in \SRTr,
\]
and the differential $\pa$ on the free operad $\FreeSRTr(\DSRTr)$ 
whose components are
\begin{equation}
\label{Jaruska}
\FreeSRTr(\DSRTr)(\GAmma) :=
\bigoplus_{T \in \gTr(\GAmma)} \  \colim{\lambda \in
  \Lev(T)} \DSRTr(T,\lambda)
\end{equation}
by the verbatim version of formula~(\ref{Posledni patek v
  Sydney}). The morphism $\rho : \FreeSRTr(\DSRTr) \to \termSRTr$ is
given by an obvious analog of~(\ref{Obehnu_to_dnes?}). One has

\begin{theorem}
The object $\minSRTr := (\FreeSRTr(\DSRTr),\pa) \stackrel\rho\to (\termSRTr,\pa
=0)$ is a minimal model of the terminal $\SRTr$-operad.
\end{theorem}

\begin{proof}
The only possibly nontrivial issue is the acyclicity $\minSRTr$ 
in positive dimensions. Comparing the formula
\[
\Free(D)(\GAmma) :=
\bigoplus_{T \in \gTr(\GAmma)} \  \colim{\lambda \in
  \Lev(T)} D(T,\lambda)
\]
defining the component of the minimal model $\minGrc$ for $\termGrc$
with~(\ref{Jaruska}) we notice the {\em equality\/}
\[
(\FreeSRTr(\DSRTr)(\Gamma),\pa) = (\Free(D)(\Gamma),\pa)
\]
for $\Gamma \in \SRTr$. In other words
\[
\minSRTr(\Gamma) = \minGrc(\Gamma), \ \hbox { for } 
\Gamma \in \SRTr \subset \Grc.
\]
The acyclicity of $\minSRTr$ thus follows from the acyclicity of
$\minGrc$ established in the proof of Theorem~\ref{Woy-Woy}.
\end{proof}


\bibliographystyle{plain}

\begin{thebibliography}{00}

\bibitem{BKW}  M.A. Batanin, J. Kock and M.~Weber,
 \newblock{Regular patterns, substitudes, Feynman categories and operads}.
 \newblock {\em Theory and Application of Categories}, 33, 6-7,
 p.148--192, 2018. 

\bibitem{SydneyI}
M.A. Batanin and M.~Markl,
\newblock {Operadic categories as a natural environment for Koszul duality}.
\newblock Preprint {\tt arXiv:1812.02935}, {\bf version~4}, July 2022.

\bibitem{SydneyII}
M.A. Batanin and M.~Markl,
\newblock {Koszul duality for operadic categories}.
\newblock Preprint {\tt arXiv:2105.05198}, {\bf version~2}, July 2022.


\bibitem{duodel}
M.A. Batanin and M.~Markl,
\newblock {Operadic categories and duoidal Deligne's conjecture}.
\newblock {\em Adv. Math.}, 285:1630--1687, 2015.

\bibitem{CD}
M.P. Carr and S.L. Devadoss,
Coxeter complexes and graph-associahedra. {\em Topology Appl.}
153(12):2155--2168, 2006.


\bibitem{CSZ} 
C. Ceballos, F. Santos and 
G. Ziegler, Many non-equivalent realizations of the associahedron. {\em Combinatorica}, 35(5):513--551, 2015. 

\bibitem{CIO} 
P.-L. Curien, J. Ivanovi\' c and J. Obradovi\' c, Syntactic aspects of hypergraph polytopes. {\em J. Homotopy Relat. Struct.} 14(1):235--279, 2019.


\bibitem{deh-val}
M.~Dehling and B.~Vallette.
\newblock Symmetric homotopy theory of operads.
{\em Algeb. Geom. and Topol.}, 21:1595--1660, 2021.

\bibitem{Dev}
 S.L. Devadoss,
\newblock A realization of graph associahedra.
{\em Discrete Math.}, 309(1):271--276, 2009.


\bibitem{DP-HP}  
K. Do\v sen and Z. Petri\'c, Hypergraph polytopes. {\em Topology and its Applications} 158:1405--1444. 2011. 

\bibitem{DMJ}
M. Doubek, B. Jur\v co, M. Markl and I. Sachs.
Algebraic structure of string field theory. 
{\em Lecture Notes in Physics}, vol. 973, Springer Verlag, Cham, 2020.

\bibitem{OC}
M. Doubek and M. Markl.
Open-closed modular operads, the Cardy condition and string field
theory.
{\em J.\ of Noncommutative Geometry}, 12(4):1359--1424, 2018.


\bibitem{gan}
W.L.~Gan.
\newblock {Koszul duality for dioperads}.
\newblock {\em Math. Res. Lett.}, 10(1):109--124, 2003.


\bibitem{GK98}
E.~Getzler and M.M. Kapranov.
\newblock {Modular operads}.
\newblock {\em Compos. Math.}, 110(1):65--126, 1998.

\bibitem{KW}
R.~Kaufmann and B.Ward.
\newblock Feynman categories. 
\newblock {\em Ast\'erisque} 387, 2017.  

\bibitem{laan:03}
P.~van~der Laan.
\newblock Coloured {Koszul} duality and strongly homotopy operads.
\newblock Preprint {\tt math.QA/0312147}, December 2003.


\bibitem{Laplante-Anfossi}
G. Laplante-Anfossi.
\newblock {The diagonal of the operahedra}.
\newblock Preprint {\tt arXiv:2110.14062}, version~1, October 2021.

\bibitem{Lu}
J. Lurie.
\newblock {\em Higher algebra}, Available at J.~Lurie's home page.


\bibitem{Malt} 
G. Maltsiniotis. 
\newblock Carr\'e exacts homotopiques, et d\'erivateurs. 
{\em Cahiers de Top. et G\'eom. Diff. Cat\'egoriques,}
LIII(1):3--63, 2012.


\bibitem{markl12:defor}
M.~Markl.
\newblock {\em Deformation theory of algebras and their diagrams}, volume 116
  of {\em CBMS Regional Conference Series in Mathematics}.
\newblock Published for the Conference Board of the Mathematical Sciences,
  Washington, DC, 2012.

\bibitem{haha}
M. Markl,
\newblock 
Homotopy algebras are homotopy algebras. 
\newblock {\em Forum Matematicum,}  {16(1)}:129--160, 2004.
 
\bibitem{ib}
 M. Markl, 
\newblock{Intrinsic brackets and the $L_\infty$-deformation
   theory of bialgebras}. {\em J. Homotopy and Relat.  Struct.,}
 5(1):177--212, 2010. 
   
\bibitem{zebrulka}
M.~Markl,
\newblock { Models for operads\/}.
\newblock  {\it Communications in Algebra,} 24(4):1471--1500, 1996.


\bibitem{mms}
M.~Markl, S.A. Merkulov, and S.~Shadrin,
\newblock {Wheeled {PROPs}, graph complexes and the master equation}.
\newblock {\em Journal of Pure and Applied Algebra}, 213:496--535, 2009.

\bibitem{MSS} 
M.~Markl, S.~Shnider, and J.~D. Stasheff,
{Operads in algebra, topology and physics}. \emph{ Mathematical
Surveys and Monographs,} vol.~96, American Mathematical Society,
Providence, RI, 2002.


\bibitem{perm}
M.~Markl.
\newblock {Permutads via operadic categories, and the hidden
  associahedron}.
\newblock {\em J. Combin. Theory Ser.~A}, 175, 105277, 2020.


\bibitem{MW}
I. Moerdijk and I. Weiss. 
\newblock
On inner Kan complexes in the category of dendroidal sets. 
\newblock {\em Adv. Math.}, 221(2):343--389, 2009.


\bibitem{JO} 
J. Obradovi\' c, Combinatorial homotopy theory for operads, 
Preprint {\tt arXiv:1906.06260}, 2019.

\bibitem{Pet}
D. Petersen. 
\newblock
The operad structure of admissible $G$-covers. 
{\em Algebra Number Theory\/}, 7(8):1953--1975, 2013.

\bibitem{Sh} 
L.S. Shapley,  \newblock{Cores of convex games\/}.
\newblock {\it International Journal
of Game Theory}, 1:12--26, 1971. 

\bibitem{sta}
J.D. Stasheff.
\newblock From operads to `physically' inspired theories.
\newblock Operads: Proceedings of Renaissance Conferences,
\newblock editors J.-L. Loday, J.D.~Stasheff and A.A.~Voronov,
\newblock {\em Contemporary Mathematics}, 202:53--82, 1997.


\bibitem{T}
D. Tanr\'e,
\newblock Homotopie Rationnelle: Mod\`eles de {Chen}, {Quillen},
{Sullivan}.
\newblock Springer-Verlag,
\emph{Lect. Notes in Math}.~1025, 1983.


\bibitem{ward0} 
B.~Ward, 
\newblock Six operations formalism for generalized operads, 
{\em Appl. Categorical Structures}, 34 (6): 121-169, 2019. 

\bibitem{ward}
B. Ward, 
\newblock Massey products for graph homology.
Preprint {\tt arXiv:1903.12055}, 2019.  

\end{thebibliography}

\end{document}